\documentclass[12pt]{amsart}
\usepackage{amssymb}
\usepackage{longtable}
\newtheorem{Thm}{Theorem}[section]
\newtheorem{Prop}[Thm]{Proposition}
\newtheorem{Cor}[Thm]{Corollary}
\newtheorem{Lem}[Thm]{Lemma}
\theoremstyle{definition}
\newtheorem{Ex}[Thm]{Example}
\newtheorem{defi}[Thm]{Definition}
\newtheorem{Rem}[Thm]{Remark}

\newcommand{\A}{\mathcal{A}}
\newcommand{\C}{\mathbb{C}}
\newcommand{\Orb}{\mathbb{O}}
\newcommand{\U}{\mathcal{U}}
\newcommand{\J}{\mathcal{J}}
\newcommand{\g}{\mathfrak{g}}
\newcommand{\GL}{\operatorname{GL}}
\newcommand{\lf}{\mathfrak{l}}

\newcommand{\HC}{\operatorname{HC}}
\newcommand{\Walg}{\mathcal{W}}
\newcommand{\Hom}{\operatorname{Hom}}
\newcommand{\B}{\mathcal{B}}

\newcommand{\Spec}{\operatorname{Spec}}

\newcommand{\Weyl}{\mathbb{A}}
\newcommand{\slf}{\mathfrak{sl}}
\newcommand{\gr}{\operatorname{gr}}
\newcommand{\F}{\operatorname{F}}

\newcommand{\Q}{\mathbb{Q}}
\newcommand{\SL}{\operatorname{SL}}

\newcommand{\param}{\mathfrak{P}}
\newcommand{\Z}{\mathbb{Z}}

\newcommand{\quo}{/\!/}
\newcommand{\I}{\mathcal{I}}
\newcommand{\Str}{\mathcal{O}}
\address{Department
of Mathematics, Yale University, New Haven, CT 06511, USA}
\email{ivan.loseu@gmail.com}
\numberwithin{equation}{section}
\title{Deformations of symplectic singularities and Orbit method for semisimple Lie algebras}
\author{Ivan Losev}
\begin{document}
\begin{abstract}
We classify filtered quantizations of conical symplectic singularities and use this to show
that all filtered quantizations of symplectic quotient singularities are spherical
symplectic reflection algebras of Etingof and Ginzburg. We further apply our classification
and a classification of filtered Poisson deformations obtained by Namikawa to establish a
version of the Orbit method for semisimple Lie algebras. Namely, we produce a natural
map from the set of coadjoint orbits of a semisimple algebraic group to the set of primitive
ideals in the universal enveloping algebra. We show that the map is injective for classical Lie
algebras and conjecture that in that case the image consists of the primitive ideals corresponding
to one-dimensional representations of W-algebras. Along the way, we get several new results on
the Lusztig-Spaltenstein induction for coadjoint orbits.
\end{abstract}
\maketitle
\begin{center}{\it Dedicated to Sasha Premet, on his 60th birthday, with admiration.}\end{center}
\markright{DEFORMATIONS OF SYMPLECTIC SINGULARITIES AND ORBIT METHOD}
\tableofcontents
\section{Introduction}

\subsection{Filtered deformations of Poisson algebras}\label{SS_intro_filt}
Our general setting is as follows. Let $A$ be a finitely generated Poisson algebra
over $\C$ that is equipped with
\begin{itemize}
\item an algebra grading $A=\bigoplus_{i=0}^\infty A_i$
such that $A_0=\C$ (such a grading will be called {\it positive})
\item and the Poisson bracket has degree $-d$, where $d$ is a positive
integer.
\end{itemize}
A basic example is as follows. We take a symplectic vector space $V$
and a finite group $\Gamma$ of its linear symplectomorphisms. Then we can take
the algebra $A=\C[V]^\Gamma$ of $\Gamma$-invariants in $\C[V]$, this is a graded
Poisson subalgebra of $\C[V]$ with $d=2$.

We are interested in  filtered deformations of $A$, i.e., in filtered associative
algebras $\A$ together with an isomorphism $\gr\A\xrightarrow{\sim} A$ of graded
algebras. We only consider two classes of deformations that are compatible
with the bracket on $A$. First, we consider {\it filtered Poisson deformations},
i.e., commutative algebras $\A^0$ equipped with a Poisson bracket decreasing the filtration degree by
$d$ such that $\gr\A^0\xrightarrow{\sim} A$ becomes a Poisson algebra isomorphism.
Second, we consider {\it filtered quantizations} $\A$. Those are non-commutative
algebras such that the Lie bracket decreases filtration degree by $d$ and
$\gr\A\xrightarrow{\sim} A$ is an isomorphism of  Poisson algebras.

We want to describe filtered Poisson deformations and filtered quantizations of $A$
up to an isomorphism (of deformations). This problem doesn't have a nice
 solution if we do not impose  any restrictions on $A$. We will assume that $X:=\operatorname{Spec}A$
has {\it symplectic singularities}. In this case, filtered Poisson deformations
were classified by Namikawa, \cite{Namikawa2}, while the classification of
quantizations is one of the main results of the present paper (Theorem \ref{Thm:quant_classif}).

\subsection{Symplectic singularities}
Recall, following Beauville, \cite{Beauville_sing}, the definition of a variety with symplectic
singularities. Let $X$ be a normal algebraic variety such that $X^{reg}$ carries a symplectic
form, $\omega^{reg}$. Since $X$ is normal, the form $\omega^{reg}$ gives rise to the
Poisson bracket on $\mathcal{O}_X$ so $X$ becomes a Poisson algebraic variety.

Let $\widehat{X}$ be a resolution of singularities for $X$. We say that $X$ {\it has  symplectic singularities}
if $\omega^{reg}$ extends to a regular form on $\widehat{X}$. In fact, this is independent of
the choice of $\widehat{X}$. Below in this section, we will assume that  $X:=\operatorname{Spec}(A)$
has symplectic singularities.

Symplectic quotient singularities $V/\Gamma$ provide an example of symplectic singularities,
\cite{Beauville_sing}. Another classical example arises as follows. We take a semisimple
Lie algebra $\g$ and the nilpotent cone $\mathcal{N}\subset \g^*$. Then $X:=\mathcal{N}$
has symplectic singularities (for $\widehat{X}$ we can take the Springer resolution $T^*(G/B)$).

Return to the general situation: $X:=\operatorname{Spec}(A)$ with symplectic singularities,
where $A$ a positively graded algebra.
Namikawa has constructed a ``Cartan space'' $\param(=\param_X)$ and a ``Weyl group'' $W(=W_X)$ for $X$.
The latter acts on the former as a crystallographic reflection group. In the case of the nilpotent cone,
we recover the Cartan subalgebra and the Weyl group of $\g$.

The simplest form of the main result of \cite{Namikawa2} can be stated as follows: the filtered Poisson
deformations are canonically indexed by the points of the quotient $\param/W$.
In this paper, we will show that 
the filtered quantizations of $A$
are indexed by the points of the same quotient $\param/W$.

Let us explain how the quantizations are constructed in the general case.
The variety $X$ admits distinguished partial resolutions called {\it $\Q$-factorial
terminalizations}. The filtered Poisson deformations/quantizations of $X$ are essentially produced by
taking global sections of Poisson deformations/quantizations of $\mathcal{O}_{\tilde{X}}$,
where $\tilde{X}$ is a $\Q$-factorial terminalization of $X$. The filtered Poisson
deformations/quantizations of $\mathcal{O}_{\tilde{X}}$ are parameterized by the
points of $\param$. Moreover, $W$-conjugate parameters give rise to isomorphic algebras of
global sections, see e.g. \cite[Proposition 3.10]{BPW}.

In the case when $X=\mathcal{N}$ our result says that all filtered quantizations
of $\C[\mathcal{N}]$ are obtained as the central reductions of the universal enveloping
algebra $U(\g)$.

In the case of $A=\C[V]^{\Gamma}$ filtered deformations of $A$ were constructed algebraically
by Etingof and Ginzburg, \cite{EG}, as spherical subalgebras $e H_{t,c}e$ in symplectic
reflection algebras $H_{t,c}$, where $t\in \C$ and $c$ is a vector in a space of
dimension $\dim \param$. The algebra $eH_{t,c}e$ is a filtered Poisson deformation
of $\C[V]^\Gamma$ when $t=0$, and is a filtered quantization when $t=1$. Results
of  Bellamy, \cite{Bellamy},
show that any filtered Poisson deformation of $\C[V]^{\Gamma}$ is isomorphic to $eH_{0,c}e$
for some $c$.
The results of the present paper show that every filtered quantization of $\C[V]^\Gamma$ has the
form $eH_{1,c}e$ for some $c$. The connection to deformations of $\Q$-terminalizations turns out
to be an important tool to study the symplectic reflection algebras, \cite{SRA_der}.

\subsection{Orbit method}
One of the general principles of Lie representation theory is that interesting irreducible
representations of Lie groups or Lie algebras should have to do with the orbits
of the coadjoint representation of the corresponding group. The most famous manifestation
of this principle is Kirillov's orbit method, \cite{Kirillov}, that describes irreducible
unitary representations of nilpotent Lie groups. Namely, let $G$ be a nilpotent Lie group.
Kirillov has constructed a natural bijection between $\g^*/G$, the set of coadjoint
$G$-orbits, and the set of isomorphism classes of  irreducible unitary representations of $G$.

An algebraic version of this result was found by Dixmier, \cite{Dixmier}, see also
\cite[Section 6]{Dixmier_book}. Namely, let
$\g$ be a nilpotent Lie algebra over $\C$. Consider the universal enveloping algebra
$U(\g)$. Recall that by a {\it primitive ideal} in an associative algebra one means
the annihilator of  a simple module. One of the points of considering primitive ideals
is that, while the set of isomorphism classes of irreducible representations
is huge and wild for almost all $\g$, the set of primitive  ideals has reasonable size
and it is often possible to describe it. In \cite{Dixmier}, Dixmier has proved that
the set $\operatorname{Prim}(\g)$ of primitive ideals in $U(\g)$ is in a natural
bijection with $\g^*/G$.

For a long time, there was, and  still is, a question of how to adapt the Orbit method
to semisimple Lie groups or algebras that are far more interesting than
nilpotent ones from the representation theoretic perspective. In this paper, we study
the algebraic version and seek to find an analog
of Dixmier's result. The classification of primitive ideals in $U(\g)$ is known, thanks to
the work of Barbasch, Joseph, Lusztig, Vogan and others, see \cite[Section 10]{CM}
for a review, but a connection with (co)adjoint orbits is very subtle and indirect. The algebraic
version of the Orbit method was studied previously, for example, in  \cite{Vogan,McG2}.

To explain our result, we need the notion of a {\it Dixmier algebra} (due to Vogan, \cite{Vogan}). Let $G$ be
a semisimple algebraic group, $\g$ its Lie algebra.

\begin{defi} By a Dixmier algebra
we mean an algebra $\A$ equipped with a rational action of $G$ together
with a quantum comoment map $U(\g)\rightarrow \A$ such that $\A$ is finitely
generated as a $U(\g)$-module.
\end{defi}
In Theorem \ref{Thm:orb_method},
we will produce a natural embedding, $\Orb_1\mapsto
\A(\Orb_1)$, of $\g^*/G$ into the set of isomorphism classes of Dixmier algebras.
All algebras occurring in the image are {\it completely prime}, i.e., have no zero divisors.

For the classical Lie algebras, we can get a stronger result. Let $\J(\Orb_1)$
denote the kernel of $U(\g)\rightarrow \A(\Orb_1)$.  In Theorem \ref{Thm:orb_method},
we will see that the map $\Orb_1\mapsto \J(\Orb_1)$ is injective provided $\g$
is classical (we also expect this to be the case when $\g$ is exceptional). The ideals
in the image correspond to 1-dimensional representations of W-algebras and, for $\g$ classical,
we expect that all such ideals occur in the image\footnote{This conjecture has been recently proved in
\cite{Topley}.}.

\subsection{Birational induction}\label{SS_bir_ind}
Let us elaborate on how the embedding of $\g^*/G$ into the set of completely prime
Dixmier algebras is constructed. First, we will identify $\g^*/G$ with the set of equivalence
classes of {\it birationally minimal induction data} defined as follows.

Take a Levi subalgebra $\lf\subset\g$, a nilpotent orbit $\Orb'\subset \lf^*$,
and $\xi\in (\lf/[\lf,\lf])^*$. Following Lusztig and Spaltenstein, \cite{LS}, we include
$\lf$ into a parabolic subalgebra $\mathfrak{p}=\lf\ltimes \mathfrak{n}$. We write
$\mathfrak{p}^\perp$ for the annihilator of $\mathfrak{p}$ in $\g^*$.
Consider the generalized Springer map $G\times^P(\xi+\overline{\Orb}'+\mathfrak{p}^\perp)\rightarrow \g^*$. Here we
view $\xi+\overline{\Orb}'+\mathfrak{p}^\perp$ as a  closed subvariety in $\mathfrak{n}^\perp$, this
subvariety is $P$-stable. The image  of the generalized Springer map is the closure of a single $G$-orbit
and it is generically finite onto its image.  We denote this map considered as a map onto
its image by $\pi$.

The following definition introduces an important terminology
to be used in the paper.

\begin{defi}\label{defi:induction}
Recall that $\lf$ stands for a Levi subalgebra of $\g$, $\Orb'$ for a nilpotent orbit in $\lf^*$
and $\xi$ for an element of $(\lf/[\lf,\lf])^*$.
\begin{itemize}
\item A triple
$(\lf,\Orb',\xi)$ is called an {\it induction datum}.
\item The orbit $\Orb_\xi$ is said to be {\it induced} from $(\lf,\Orb',\xi)$.
\item  When $\pi$ is generically injective, i.e., birational, we will say that $\Orb_\xi$
is {\it birationally induced} from $(\lf,\Orb',\xi)$.
\item We say that $\Orb'$ is {\it birationally
rigid} if it cannot be birationally induced from a proper Levi.
\item If $\pi$ is birational and $\Orb'$
is birationally rigid, then we say that $(\lf,\Orb',\xi)$
is a {\it birationally minimal induction datum}.
\end{itemize}
\end{defi}

The group $G$ acts naturally on the set
of  induction data and one can show that the map $(\lf,\Orb',\xi)\mapsto
\Orb_\xi$ is a bijection between the set of $G$-equivalence classes of the
birationally minimal induction data and the set $\g^*/G$ of coadjoint orbits,
see Theorem \ref{Thm:bir_sheets}.

Now pick a coadjoint orbit $\Orb_1$ and let $(\lf,\Orb',\xi)$ be the corresponding
birationally minimal induction datum. Let $G/H$ denote the open $G$-orbit in
$G\times^P(\overline{\Orb}'+\mathfrak{n})$. This is a finite cover of some nilpotent orbit
$\Orb$ and hence a symplectic variety. Moreover, $X:=\operatorname{Spec}(\C[G/H])$
has symplectic singularities and $\param=(\lf/[\lf,\lf])^*$,
see Proposition \ref{Prop:NW_comput}. We let $\A(\Orb_1)$
to be the quantization of $\C[X]$ corresponding to the parameter $\xi$. This is
the Dixmier algebra that we assign to $\Orb_1$.

\begin{Rem}
In this paper, we only need to consider induction from coadjoint orbits. More generally,
one can consider induction from their covers. Many  geometric results of this  paper generalize to that
setting, see, for example, \cite{Mitya},\cite{Namikawa_cover}.
\end{Rem}

\subsection{Content of the paper}
Let us describe the content of this paper.

In Section \ref{S_sympl} we discuss symplectic singularities, their $\Q$-factorial terminalizations
and their deformations. In Section \ref{SS_sympl_sing_Q_term} we introduce the notion of
a symplectic singularity following Beauville, give some examples and discuss the
$\Q$-factorial terminalizations.  In Section \ref{SS_filt_deform} we recall the notion
of a filtered Poisson deformation of a not necessarily affine Poisson variety and
the classification of such deformations of $\Q$-factorial terminalizations following
\cite{Namikawa_flop}.
In Section \ref{SS_deform_Q_fact} we apply this to studying  Poisson deformations
 of  conical
symplectic singularities. In Section \ref{SS_univ_deform} we recall results of
Namikawa, \cite{Namikawa1,Namikawa2}, on the universal conical Poisson deformation
of a conical symplectic singularity. The only new results in Section \ref{S_sympl}
are contained in Section \ref{SS_Poiss_deriv}, where we study the
negative degree Poisson derivations of  the algebra of functions on a conical symplectic
singularity.

In Section \ref{S_quant} we study quantizations of conical symplectic singularities.
We start, Section \ref{SS_quant_period}, by recalling the general definition
of a filtered quantization and classification results in the symplectic
case obtained in \cite{BK_quant,quant_iso}. In Section \ref{SS_quant_Q_term} we
produce quantizations of a conical symplectic singularity starting from
those of its $\Q$-terminalization following \cite{BPW}.  Also in that section
we state the main classification result, Theorem \ref{Thm:quant_classif}.
This theorem is then proved in the three subsequent sections.  Then,
in Section \ref{SS_SRA} we treat the case of symplectic quotient singularities
and show that in this case all quantizations are spherical symplectic reflection
algebras of Etingof and Ginzburg, \cite{EG}. Finally,
 in Section \ref{SS_iso},
we study the question of when filtered Poisson deformations (resp., quantizations)
are isomorphic as filtered Poisson (resp., associative) algebras.

In Section \ref{S_bir_sheet} we study various questions related
to the geometry of coadjoint orbits. In Section \ref{SS_LS_induction} we recall
some basic results on the Lusztig-Spaltenstein induction. In Section
\ref{SS_bir_sheet} we recall sheets in semisimple Lie algebras,
introduce the related notion of {\it birational sheets} and
state a result, Theorem \ref{Thm:bir_sheets}, describing their structure.
In Section \ref{SS_nilp_Q_term} we study $\Q$-terminalizations
of normalizations of nilpotent orbit closures giving conceptual proofs
of results previously obtained by Namikawa, \cite{Namikawa_induced},
and Fu, \cite{Fu}.  Then we compute Weyl groups of these normalizations
and their suitable covers, Section \ref{SS_Weyl_comput}, generalizing
results of Namikawa, \cite[Section 2]{Namikawa2}. We use results of Sections
\ref{SS_nilp_Q_term} and \ref{SS_Weyl_comput} to prove Theorem
\ref{Thm:bir_sheets} in Section \ref{SS_sheet_conseq}.

In Section \ref{S_W_Orb} we establish our version of the orbit
method. We start by recalling results and constructions
related to W-algebras. In Section \ref{SS_W_algebras} we present
their construction following \cite{Wquant,HC}. Then in
Section \ref{SS_HC_restr} we recall the restriction functor
for Harish-Chandra bimodules from \cite{HC} and its basic properties.
Section \ref{SS_Orb_method} is the main part, there we state  and prove
Theorem \ref{Thm:orb_method} that provides a map $\g^*/ G\rightarrow
\operatorname{Prim}(\g)$. We discuss the image of this map
in Section \ref{SS_image}.

{\bf Acknowledgements}. This paper would have never appeared without help
from Pavel Etingof and Dmitry Kaledin. I would like to thank them
as well as Yoshinori Namikawa, Sasha Premet, and David Vogan for
stimulating discussions. Finally, I would like to thank the referees for numerous comments
that helped me to improve the exposition. I am very happy to dedicate the paper
to Sasha Premet on his 60th birthday, this paper, as well as
much of my other work, is inspired by his fascinating results.
The paper was partially supported by the NSF under grants DMS-1161584, DMS-1501558.
This work has also been funded by the  Russian Academic Excellence Project '5-100'.

\section{Symplectic singularities and their deformations}\label{S_sympl}
\subsection{Symplectic singularities and $\Q$-terminalizations}\label{SS_sympl_sing_Q_term}
Let $X$ be a normal Poisson algebraic variety over $\C$ such that the smooth locus $X^{reg}$
is a symplectic variety. Let $\omega^{reg}$ denote the symplectic form on $X^{reg}$.
Following \cite[Definition 1.1]{Beauville_sing}, we say that $X$ has {\it symplectic singularities} if there is a projective resolution of singularities
$\rho:\widehat{X}\rightarrow X$ such $\rho^*(\omega^{reg})$ extends to a regular (but not necessarily
symplectic) $2$-form on  $\widehat{X}$. Once such $\widehat{X}$ exists, for any other projective resolution
$\rho_1:\widehat{X}_1\rightarrow X$, the form $\rho_1^*(\omega^{reg})$ also extends to a regular 2-form,
see \cite[(1.2)]{Beauville_sing}.

Now let $X$ be an affine Poisson variety. We say that $X$ is {\it conical} if there is an algebra grading
$\C[X]=\bigoplus_{i=0}^{\infty}\C[X]_i$ and a positive integer $d$ such that
\begin{itemize}
\item $\C[X]_0=\C$.
\item $\{f,g\}\in \C[X]_{i+j-d}$ for any $i,j$ and $f\in \C[X]_i, g\in \C[X]_j$.
\end{itemize}

By a {\it conical symplectic singularity} we mean a conical variety with symplectic
singularities. Let us give two classical examples of conical symplectic singularities,
see \cite[(2.5),(2.6)]{Beauville_sing}.

\begin{Ex}\label{Ex:sympl_quot}
Let $V$ be a symplectic vector space and $\Gamma\subset \operatorname{Sp}(V)$ be a finite
subgroup. Then $V/\Gamma$ is a conical symplectic singularity with $d=2$.
\end{Ex}

\begin{Ex}\label{Ex:nilp_orbit}
Let $\g$ be the Lie algebra of an algebraic group.
Then $\g^*$ is a Poisson algebraic variety whose symplectic
leaves are coadjoint orbits. Assume that $\g$ is semisimple. Let $\Orb$ be a nilpotent orbit in $\g^*$.
The algebra
$\C[\Orb]$ is finitely generated, in fact, it is the normalization of $\C[\overline{\Orb}]$.
The variety $X:=\operatorname{Spec}(\C[\Orb])$ is a conical symplectic singularity
with $d=1$.
\end{Ex}

We will need the definition and some properties of $\Q$-factorial terminalizations.
Let $\tilde{X}$ be a normal algebraic variety. Recall that  $\tilde{X}$ is called {\it $\Q$-factorial}
if for any Weil divisor a nonzero integral multiple is Cartier.

\begin{Prop}[Proposition 2.1 in \cite{SRA_der}]\label{Prop:terminalization}
Let $X$ have symplectic singularities. Then there is a birational projective morphism
$\rho:\tilde{X}\rightarrow X$, where $\tilde{X}$ has the following properties:
\begin{itemize}
\item[(a)] $\tilde{X}$ is an irreducible, normal, Poisson variety (and hence has symplectic singularities).
\item[(b)] $\tilde{X}$ is $\Q$-factorial.
\item[(c)] $\tilde{X}$ has terminal singularities.
\item[(d)] $\operatorname{codim}_{\tilde{X}}\tilde{X}^{sing}\geqslant 4$.
\item[(e)] If $X$ is, in addition, conical, then $\tilde{X}$ admits a $\C^\times$-action such that
$\rho$ is $\C^\times$-equivariant.
\end{itemize}
\end{Prop}


Below we will say that $\tilde{X}$ is a {\it $\Q$-factorial terminalization} (or
{\it $\Q$-terminalization}) of $X$.

\begin{Ex} Consider the situation of Example \ref{Ex:sympl_quot}. Suppose, first, that
$\Gamma$ does not contain {\it symplectic reflections}, i.e., elements $\gamma$ with
$\operatorname{rk}(\gamma-\operatorname{id})=2$. Then $X$ itself satisfies properties
(a)-(e). An opposite extreme is when $X$ has a symplectic resolution of singularities.
This happens, for example, when $\Gamma$ is a so called {\it wreath-product} group
$\Gamma=S_n\ltimes \Gamma_1^n$, where $\Gamma_1$ is a finite subgroup of
$\SL_2(\C)$ and $V=\C^{2n}$.

$\Q$-terminalizations of the varieties $\Spec(\C[\Orb])$ will be considered in more
detail in Section \ref{SS_nilp_Q_term}.
\end{Ex}

To finish let us provide an important property of a terminalization $\tilde{X}$.

\begin{Prop}\label{Prop:coh_vanish}
If $X$ is affine, then $\C[\tilde{X}^{reg}]=\C[\tilde{X}]=\C[X],$ $H^i(\tilde{X},\mathcal{O}_{\tilde{X}})=0$ for all $i>0$, and  $H^i(\tilde{X}^{reg}, \mathcal{O}_{\tilde{X}})=0$ for $i=1,2$.
\end{Prop}
\begin{proof}
While this is classical (see, e.g., the proof of
\cite[Lemma 12]{Namikawa_flop}), we provide a proof for reader's convenience.

Let $\iota$ denote the inclusion $\tilde{X}^{reg}\hookrightarrow \tilde{X}$.
Since $\tilde{X},X$ are singular symplectic, (a) of Proposition
\ref{Prop:terminalization}, they have rational singularities,
see \cite[Proposition 2.3]{Beauville_sing}. In particular, they are Cohen-Macaulay.
Thanks to this, we have
$\iota_* \mathcal{O}_{\tilde{X}^{reg}}=\mathcal{O}_{\tilde{X}},
R^i\iota_*\mathcal{O}_{\tilde{X}^{reg}}$ for $0<i<\operatorname{codim}_{\tilde{X}}\tilde{X}^{sing}-1$,
i.e., by (d) of Proposition \ref{Prop:terminalization}, for $i=1,2$. This reduces
the proof to showing that $\C[\tilde{X}]=\C[X]$ and  $H^i(\tilde{X}, \mathcal{O}_{\tilde{X}})=0$
for $i>0$. These two claims follow  because $\tilde{X}\rightarrow X$
is proper and birational and $X$ has rational singularities.
\end{proof}

\subsection{Filtered Poisson deformations}\label{SS_filt_deform}
Let $X'$ be a  Poisson variety.

We are going to recall three notions of Poisson deformations of
$X'$.




We can talk about formal Poisson deformations of $X'$. Let $\widehat{B}$ be a quotient of $\C[[\hbar_1,\ldots,\hbar_k]]$ for some $k>0$. By definition, a {\it formal
Poisson deformation} of $X'$ over $\widehat{B}$  a pair $(X'_{\widehat{B}},\iota)$, where
\begin{itemize}
\item $X'_{\widehat{B}}$ is a flat formal Poisson scheme over the formal spectrum $\operatorname{Specf}(\widehat{B})$
\item
and $\iota$ is an isomorphism $X'_{\widehat{B}}\times_{\operatorname{Specf}(\widehat{B})}\operatorname{pt}\xrightarrow{\sim} X'$
of Poisson schemes.
\end{itemize}
An isomorphism of deformations $(X'^1_{\widehat{B}},\iota^1),(X'^2_{\widehat{B}},\iota^2)$ is an isomorphism of formal Poisson schemes
$X'^1_{\widehat{B}}\xrightarrow{\sim} X'^2_{\widehat{B}}$ over $\operatorname{Specf}(\widehat{B})$ such that the induced isomorphism
$X'^1_{\widehat{B}}\times_{\operatorname{Specf}(\widehat{B})}\operatorname{pt}\xrightarrow{\sim}
X'^2_{\widehat{B}}\times_{\operatorname{Specf}(\widehat{B})}\operatorname{pt}$ intertwines $\iota^1$ with $\iota^2$.

Now suppose that $X'$ is
equipped with a $\C^\times$-action
such that the Poisson bracket has degree $-d$ for some $d\in \Z_{>0}$.
Let $\widehat{B}$ be the completion at $0$ of a positively graded algebra.
In this situation  we can talk about {\it graded} formal Poisson deformation of $X'$.
By definition, this is a formal deformation $(X'_{\widehat{B}},\iota)$
together with an algebraic action of $\C^\times$ on $X'_{\widehat{B}}$  subject to the following
properties:
\begin{itemize}
\item[(i)] The action rescales the Poisson bracket by
$t\mapsto t^{-d}$.
\item[(ii)] It is compatible with the action on $\widehat{B}$.
\item[(iii)] And $\iota$ is $\C^\times$-equivariant.
\end{itemize}
An isomorphism of graded formal deformations is defined similarly to that of formal deformations.

Let $B$ be a positively graded algebra. We can introduce the notion of a {\it graded Poisson deformation} $X'_B$ of $X'$ similarly to the notion of a graded formal deformation. The difference is
that now $X'_B$ is an actual scheme over $\operatorname{Spec}(B)$. We note that
from a graded deformation one can naturally get a graded formal deformation: replacing $B$
with its completion at $0$ and $X'_B$ by the formal neighborhood of $X'$ in $X'_B$ to be denoted
by $X'^\wedge_B$.

\begin{Prop}\label{Prop:KV_deform}
Let $\tilde{X}$ be a $\Q$-factorial terminalization of a conical symplectic singularity. Set $\param:=H^2(\tilde{X}^{reg},\C)$. Then there is a universal graded deformation $\tilde{X}_\param$ of $\tilde{X}$ over $\param$. The universal property is understood as follows: for any graded Poisson deformation $\tilde{X}_B$ over $\operatorname{Spec}(B)$
there is a unique $\C^\times$-equivariant morphism $\Spec(B)\rightarrow \param$ and an
isomorphism of graded deformations $\Spec(B)\times_\param \tilde{X}_\param\xrightarrow{\sim}
\tilde{X}_B$ (which is not required to be unique).
\end{Prop}
\begin{proof}
The proof is in several steps and is based on \cite{Namikawa_flop}.

{\it Step 1}. Namikawa proved that the Poisson deformation functor of $\tilde{X}$ is pro-representable,
\cite[Theorem 14]{Namikawa_flop} and unobstructed \cite[Corollary 15]{Namikawa_flop}. This means that
there is a universal formal deformation of $\tilde{X}$. This is the direct limit of the schemes
denoted by $X_n^{univ}$ in the discussion before \cite[Lemma 20]{Namikawa_flop}. Since the
Poisson deformation functor is unobstructed, the base of the universal deformation is
the formal neighborhood of $0$ in the space of deformations of $\tilde{X}$ over the dual numbers. By \cite[Proposition 13]{Namikawa_flop}, this space coincides with the similarly defined space for
$\tilde{X}^{reg}$. By \cite[Lemma 12]{Namikawa_flop}, the latter space coincides with $\param=H^2(\tilde{X}^{reg},\C)$. We write $\param^\wedge$ for the formal neighborhood
of $0$ in $\param$ and $\tilde{X}^\wedge_\param$ for the universal formal deformation of
$\tilde{X}$.

{\it Step 2}. By \cite[Lemma 20]{Namikawa_flop}, $\tilde{X}^\wedge_\param$ is a graded formal
deformation over $\param^\wedge$. We claim that it is universal. Let $\widehat{B}$ be the completion
of a positively graded algebra at $0$ and $\tilde{X}_{\widehat{B}}$ be a graded formal deformation
of $\tilde{X}$. By Step 1, we have a unique morphism $\operatorname{Specf}(\widehat{B})\rightarrow
\param^{\wedge}$ and an isomorphism of formal Poisson deformations $\tilde{X}_{\widehat{B}}\xrightarrow{\sim}
\tilde{X}_{\widehat{B}}^1:=\operatorname{Specf}(\widehat{B})\times_{\param^\wedge}\tilde{X}^\wedge_\param$. By the uniqueness of the former, $\operatorname{Specf}(\widehat{B})\rightarrow \param^{\wedge}$ is $\C^\times$-equivariant. So $\tilde{X}_{\widehat{B}}^1$ becomes a graded formal deformation. We need to prove that we can choose
a $\C^\times$-equivariant Poisson isomorphism $\tilde{X}_{\widehat{B}}\xrightarrow{\sim}
\tilde{X}_{\widehat{B}}^1$. The proof is standard but we provide it for reader's convenience.

Let $\mathfrak{m}$ denote the maximal ideal in $\widehat{B}$. Set $B_k:=\widehat{B}/\mathfrak{m}^k$. We write
$\tilde{X}_k,\tilde{X}^1_k$ for the base changes of $\tilde{X}_{\widehat{B}},\tilde{X}_{\widehat{B}}^1$
to $\operatorname{Spec}(B_k)$. Note that for $k=1$, both of these schemes are just $\tilde{X}$.
It is enough to show that, for all $k>0$, any $\C^\times$-equivariant isomorphism of Poisson deformations
$\varphi_k:\tilde{X}_k\xrightarrow{\sim}\tilde{X}_k^1$ can be lifted to a $\C^\times$-equivariant
isomorphism of deformations    $\tilde{X}_{k+1}\xrightarrow{\sim}\tilde{X}_{k+1}^1$.

Note that a Poisson deformation of an arbitrary Poisson variety
$X'$ over an Artinian base $B$ gives a sheaf of Poisson $B$-algebras
on $X'$ that deforms the structure sheaf: we take the structure sheaf of $X'_B$
and sheaf theoretically pull it back to $X'$. We get a deformation of $\Str_{X'}$ over $B$. Conversely, from such a sheaf, say $\Str_{X',B}$, we can get a deformation of $X'$ over $\operatorname{Spec}(B)$: we cover $X'$ with open affine subsets, $X'=\bigcup U_i$,
and glue the affine schemes $\operatorname{Spec}(\Gamma(U_i, \Str_{X',B}))$ along their intersections
$\operatorname{Spec}(\Gamma(U_i\cap U_j, \Str_{X',B}))$. We get a Poisson deformation $X'_B$ over $B$.

 Denote the sheaves on $\tilde{X}$ corresponding to $\tilde{X}_{k+1},\tilde{X}^1_{k+1}$ by $\Str_{k+1},\Str_{k+1}^1$.
We have a $\C^\times$-equivariant isomorphism of sheaves of Poisson algebras $\varphi_k:
\Str_k\xrightarrow{\sim}\Str_{k}^1$. We know that there is a lift to an isomorphism
$\Str_{k+1}\xrightarrow{\sim}\Str_{k+1}^1$. What we need to show that it can be chosen to
be $\C^\times$-equivariant. Let $\mathsf{A}_{\tilde{X}}$ denote the set of lifts. This is an affine space
with associated vector space  $\operatorname{PDer}(\Str_{\tilde{X}})\otimes \mathfrak{m}^k/\mathfrak{m}^{k+1}$, where $\operatorname{PDer}$ stands for the space of Poisson derivations. We need to show that $\C^\times$ has a fixed  point on $\mathsf{A}_{\tilde{X}}$.
Note that any affine space admits a canonical embedding into a vector space equipped with a linear function, say $t$, the image of the embedding is the level set $t=1$. Let $\mathsf{V}_{\tilde{X}}$
be this space for $\mathsf{A}_{\tilde{X}}$. It embeds as a subspace in the space of
$B_{k+1}$-linear homomorphisms of sheaves $\Hom_{B_{k+1}}(\tilde{\Str}_{k+1},\tilde{\Str}_{k+1}^1)$ as follows. For $z\neq 0$, the level set $t=z$ embeds as the subset of all maps $\psi$ such that  $z^{-1}\psi$ is an algebra homomorphism giving $\varphi_k$ after base change to $B_k$. The level set $t=0$ embeds as $\operatorname{PDer}(\Str_{\tilde{X}})\otimes \mathfrak{m}^k/\mathfrak{m}^{k+1}$ (these maps vanish on $\mathfrak{m}\Str_{k+1}$ and have image in
$\mathfrak{m}^k \Str^1_{k+1}$). We need to show that the action of $\C^\times$ on $\mathsf{V}_{\tilde{X}}$
is rational (i.e. comes from a grading), the existence of a fixed point in $\mathsf{A}_{\tilde{X}}$ will follow.

By a result of Sumihiro, \cite{Sumihiro}, we can cover $\tilde{X}$ with open affine $\C^\times$-stable subsets $U_i, i=1,\ldots,m$. Restricting sheaves $\Str_{k+1},\Str_{k+1}^1$ to $U_i$ we get graded Poisson flat $B_{k+1}$-algebras $\Str_{k+1}(U_i), \Str^1_{k+1}(U_i)$. We can consider the subspaces
$\mathsf{V}_{U_i}\subset \Hom_{B_{k+1}}(\Str_{k+1}(U_i), \Str^1_{k+1}(U_i))$ defined similarly to
$\mathsf{V}_{\tilde{X}}$. We have a $\C^\times$-equivariant embedding
$\mathsf{V}_{\tilde{X}}\hookrightarrow \bigoplus_i\mathsf{V}_{U_i}$ by restriction. So it is enough to show that
the $\C^\times$-action on $\mathsf{V}_{U_i}$ is rational.

The algebra $\Str_{k+1}(U_i)$ is generated by homogeneous elements, say $f_1,\ldots,f_\ell$,
of degrees $d_1,\ldots,d_\ell$. Note that every element of $\mathsf{V}_{U_i}\subset
\Hom_{B_{k+1}}(\Str_{k+1}(U_i), \Str^1_{k+1}(U_i))$ is uniquely determined by the images of
$f_1,\ldots,f_\ell$. This gives a $\C^\times$-equivariant embedding
$\mathsf{V}_{U_i}\hookrightarrow \bigoplus_{j=1}^\ell\Str^1_{k+1}(U_i)\langle -d_j\rangle$,
where in the brackets we have a grading shift. The action of $\C^\times$ on the target of
this embedding is manifestly rational. This finishes the proof of the claim that $\tilde{X}^{\wedge}_\param$ is a universal graded formal deformation.

{\it Step 3}. By \cite[Lemma 22]{Namikawa_flop}, the formal scheme $\tilde{X}^\wedge_{\param}$
over $\param^\wedge$ with an action of $\C^\times$ admits the algebraization: a graded
Poisson deformation $\tilde{X}_{\param}$ over $\param$ such that the formal neighborhood
of $\tilde{X}$ in $\tilde{X}_\param$ is identified with  $\tilde{X}^\wedge_{\param}$.
We claim that $\tilde{X}_\param$ has the required universal property.

Indeed, let $\tilde{X}_B$ be another graded deformation. Then we have the graded formal deformation
$\tilde{X}_B^\wedge$ over the completion $\widehat{B}$ of $B$. By Step 2, we have a unique $\C^\times$-equivariant morphism of formal
schemes $\operatorname{Specf}(\widehat{B})\rightarrow \param^\wedge$ and an isomorphism of graded formal
deformations
\begin{equation}\label{eq:formal_scheme_iso}
 \operatorname{Specf}(\widehat{B})\times_{\param^\wedge}\tilde{X}^\wedge_{\param}
\xrightarrow{\sim} \tilde{X}_B^\wedge.
\end{equation}
The morphism $\operatorname{Specf}(\widehat{B})\rightarrow \param^\wedge$ is induced by
the unique $\C^\times$-equivariant morphism $\operatorname{Spec}(B)\rightarrow \param$.
The source of (\ref{eq:formal_scheme_iso}) is the formal neighborhood of $\tilde{X}$
in $\operatorname{Spec}(B)\times_{\param}\tilde{X}_\param$.  The proof of \cite[Proposition A.5]{Namikawa_flop} shows that the algebraizations  $\operatorname{Spec}(B)\times_{\param}\tilde{X}_\param, \tilde{X}_B$ are uniquely recovered
from their formal neighborhoods of $\tilde{X}$ together with the $\C^\times$-action.
This shows that (\ref{eq:formal_scheme_iso}) comes from a unique isomorphism of graded Poisson deformations $\operatorname{Spec}(B)\times_{\param}\tilde{X}_\param\xrightarrow{\sim} \tilde{X}_B$.
\end{proof}

\begin{Rem}\label{Rem:Namikawa_KV}
In this remark we relate results of \cite{Namikawa_flop} to those of \cite{KV}.
Kaledin and Verbitsky, \cite[Theorem 3.6]{KV}, show that for a \underline{smooth} symplectic
variety $X'$ with $H^i(X',\Str_{X'})=0$ for $i=1,2$, there is  a universal
formal deformation with base being the formal neighborhood of $0$ in $H^2(X',\C)$.
The vanishing condition above holds for $X'=\tilde{X}^{reg}$
thanks to Proposition \ref{Prop:coh_vanish}.

Using the construction relating Poisson deformations over Artinian bases to sheaves (Step 2
of the proof of Proposition \ref{Prop:KV_deform}) we can relate the Poisson deformations of $\tilde{X}$ to those $\tilde{X}^{reg}$.
Namely, given  a deformation $\tilde{X}_B$ over an Artinian base $\operatorname{Spec}(B)$, we form
the corresponding sheaf $\Str_{\tilde{X},B}$, restrict it to $\tilde{X}^{reg}$ via the sheaf-theoretic
pullback and take the corresponding deformation of $\tilde{X}^{reg}$. To get back we apply the pushforward
for $\iota: \tilde{X}^{reg}\rightarrow \tilde{X}$. As we have pointed out in the proof of
Proposition \ref{Prop:coh_vanish}, $R^1\iota_* \Str_{\tilde{X}^{reg}}=0$. Any deformation of $\mathcal{O}_{\tilde{X}^{reg}}$ over $B$ has a filtration with successive quotients
$\Str_{\tilde{X}^{reg}}$. An easy argument using
the long exact sequence in cohomology shows that $\iota_*$ sends any deformation of $\Str_{\tilde{X}^{reg}}$
to a deformation of $\Str_{\tilde{X}}$ over $B$. So we get a bijection between the deformations of
$\Str_{\tilde{X}^{reg}},\Str_{\tilde{X}}$ over $B$.

It follows that the Poisson deformation functors for $\tilde{X}$ and $\tilde{X}^{reg}$ are isomorphic. In particular, the existence of the universal formal deformation for $\tilde{X}$ implies that for
$\tilde{X}^{reg}$ and vice versa.

We also note that in the setting of \cite{KV} we can talk about graded formal deformations
of $X'$ if it comes with a $\C^\times$-action that rescales the Poisson bracket by $t\mapsto t^{-d}$.
Similarly to Step 2 of Proposition \ref{Prop:KV_deform}, the universal deformation over
the formal neighborhood of $0$ in $H^2(X',\C)$ is a universal graded formal deformation.
\end{Rem}

\subsection{Deformations from $\Q$-terminalizations}\label{SS_deform_Q_fact}
Let $A$ be a positively graded Poisson algebra with $\deg \{\cdot,\cdot\}=-d$.
For example, we can take $A=\C[X]$, where $X$ is a conical symplectic singularity.

By a {\it filtered
Poisson deformation} of $A$ we mean a pair $(\mathcal{A}^0,\iota)$, where
\begin{itemize}
\item $\mathcal{A}^0$ is a $\Z_{\geqslant 0}$-filtered Poisson
algebra, $\mathcal{A}^0=\bigcup_{i\in \Z}\mathcal{A}^0_{\leqslant i}$,
such that $\{\mathcal{A}^0_{\leqslant i},\mathcal{A}^0_{\leqslant j}\}
\subset \mathcal{A}^0_{\leqslant i+j-d}$,
\item
and $\iota$ is an isomorphism $\gr\mathcal{A}^0\xrightarrow{\sim} A$
of graded Poisson algebras.
\end{itemize}

We note that to give a filtered Poisson deformation of $A$ is the same as to give a graded Poisson deformation of $\operatorname{Spec}(A)$ over $\Spec(\C[\hbar])$, where $\hbar$ has degree $1$.
Indeed, starting with $\mathcal{A}^0$ we can form its Rees algebra $R_\hbar(\mathcal{A}^0)$ and set
$X_\hbar:=\Spec(R_\hbar(\mathcal{A}^0))$. The isomorphism $\iota$ gives rise to an isomorphism
$\operatorname{pt}\times_{\Spec(\C[\hbar])}X_\hbar\xrightarrow{\sim} \Spec(A)$ that makes
$X_\hbar$ into a graded Poisson deformation of $\Spec(A)$. Conversely, starting with
a graded Poisson deformation $X_\hbar$ we set $\mathcal{A}^0:=\C[X_\hbar]/(\hbar-1)$ getting
a filtered Poisson deformation of $A$.

We proceed to a construction of filtered Poisson deformations of $\C[X]$. Let $\tilde{X}$
stand for a $\Q$-factorial terminalization of $X$.
We set $\param:=H^2(\tilde{X}^{reg},\C)$ and let $\tilde{X}_{\param}$ be the universal
graded deformation of $\tilde{X}$, see Proposition \ref{Prop:KV_deform}.
Our primary goal in this section is to construct filtered Poisson deformations of
$\C[X]$ out of $\tilde{X}_\param$.

First, we are going to describe the structure of the space $\param$ following \cite{Namikawa1}
and \cite{Namikawa2}.

Kaledin, \cite[Theorem 2.3]{Kaledin_symplectic}, proved that $X$ has finitely many
symplectic leaves. Moreover, he proved that for any leaf $\mathcal{L}\subset X$, there
is a formal slice to $\mathcal{L}$ in $X$: if we pick a point $x\in \mathcal{L}$ and let
$X^{\wedge_x},\mathcal{L}^{\wedge_x}$ denote the formal neighborhoods of $x$ in $X$
and $\mathcal{L}$, respectively, then there is an affine formal Poisson scheme $\Sigma'$
with $X^{\wedge_x}\cong \mathcal{L}^{\wedge_x}\times \Sigma'$.

Let $\mathcal{L}_1,\ldots, \mathcal{L}_k$ be the codimension $2$ symplectic leaves
of $X$. The formal slice $\Sigma'_i$ to $\mathcal{L}_i$ in $X$ is an ADE type Kleinian
singularity $(\C^2)^{\wedge_0}/\Gamma_i$ (here $\bullet^{\wedge_0}$ stands for the
formal neighborhood at $0$) so we can consider the Weyl group $\widehat{W}_i,
i=1,\ldots,k$ of that ADE type. We write $\Sigma_i$ for $\C^2/\Gamma_i$, $\tilde{\Sigma}_i$
for its minimal resolution and set $\widehat{\param}_i:=H^2(\tilde{\Sigma}_i,\C)$.
The space $\widehat{\param}_i$ is the Cartan space for $\widehat{W}_i$.

Choose a point $x\in \mathcal{L}_i$. The preimage of $x$ in $\tilde{X}$ is the same as
the preimage of $0\in \Sigma_i$ in $\tilde{\Sigma}_i$, see the discussion in
the beginning of \cite[Section 1]{Namikawa2}. So this preimage is the union of $\mathbb{P}^1$'s
intersecting according to the Dynkin diagram of the root system associated to $\widehat{W}_i$.
The fundamental group $\pi_1(\mathcal{L}_i)(=\pi_1(\mathcal{L}_i,x))$ acts on the preimage
via the monodromy action.
Hence it acts on
$\widehat{\param}_i,\widehat{W}_i$ by diagram
automorphisms. Let $\param_i\subset \widehat{\param}_i, W_i\subset \widehat{W}_i$ denote the fixed points.
Following Namikawa, see the discussion in the beginning of \cite[Section 1]{Namikawa2}, define the Namikawa-Weyl group as $W(=W_X):=\prod_{i=1}^k W_i$.

\begin{Lem}\label{Lem:Cartan_compute}
We have $\param=H^2(X^{reg},\C)\oplus \bigoplus_{i=1}^k \param_i$.
\end{Lem}
\begin{proof}
Let $X^1$ denote the complement in $X$ to the union of all symplectic leaves with codimension $4$
or higher. So $X^1=X^{reg}\sqcup \bigsqcup_{i=1}^k \mathcal{L}_i$. According to
\cite[Proposition 4.2]{Namikawa1}, $\param=H^2(X^1,\C)\oplus \bigoplus_{i=1}^k \param_i$. What we need to do therefore is to show that $H^2(X^1,\C)=H^2(X^{reg},\C)$.

Pick a point $x\in \mathcal{L}_i$. Let $\Sigma\subset X$ be a slice to $\mathcal{L}_i$ through $x$
in the complex analytic category. Using Hamiltonian flows for suitable analytic functions in a neighborhood of $x$ (and shrinking $\Sigma$ if necessary) we can identify a complex analytic neighborhood of $x$  with $B\times \Sigma$, where $B$ is a neighborhood of $x$ in $\mathcal{L}_i$. Since $X$
is normal, $\Sigma$ is also normal as a complex analytic space. It is 2-dimensional so it must be
a Kleinian singularity of the same type as $\Sigma_i$.

Pick tubular neighborhoods $Y_1,\ldots,Y_k$
of $\mathcal{L}_1,\ldots,\mathcal{L}_k$ and set $Y_i^\times:=Y_i\setminus \mathcal{L}_i$.
So, by the previous paragraph, topologically  $Y_i$ is a locally trivial fibration over $\mathcal{L}_i$
with fiber $D/\Gamma_i$, where $D$ is an open unit ball in $\C^2$.
In particular, $H^j(Y_i, \C)=H^j(\mathcal{L}_i,\C)$ for all $j$. Also note that $\Gamma_i$
acts freely on $D^\times$. It follows that $H^j(D^\times/\Gamma_i, \C)=H^j(\mathbb{S}^3,\C)^{\Gamma_i}$,
in particular, this homology group is zero when $j=1,2$. It follows that
$H^j(Y_i^\times, \C)=H^j(\mathcal{L}_i,\C)$ for  $j=1,2$. Hence the pullback
homomorphism $H^j(Y_i,\C)\rightarrow H^j(Y_i^\times,\C)$
is an isomorphism for $j=1,2$.  The Mayer-Vietoris
exact sequence for the covering $X^1=X^{reg}\cup (\bigsqcup Y_i)$ gives
\begin{align*}
&H^1(X^{reg},\C)\oplus\bigoplus_{i=1}^k H^1(Y_i,\C)\rightarrow
\bigoplus_{i=1}^k H^1(Y_i^\times,\C)\rightarrow H^2(X^1,\C)\\&
\rightarrow H^2(X^{reg},\C)\oplus\bigoplus_{i=1}^k H^2(Y_i,\C)\rightarrow
\bigoplus_{i=1}^k H^2(Y_i^\times,\C).
\end{align*}
By the argument above in this paragraph the first and the last map restrict to
isomorphisms $\bigoplus_{i=1}^k H^j(Y_i,\C)\rightarrow
\bigoplus_{i=1}^j H^1(Y_i^\times,\C)$.
It follows that the
pull-back map $H^2(X^1,\C)\rightarrow H^2(X^{reg},\C)$ is an isomorphism.
\end{proof}

Let us proceed to producing filtered Poisson deformations of $\C[X]$
from $\tilde{X}_\param$.

\begin{Prop}\label{Prop:deform_X}
The following is true:
\begin{enumerate}
\item The algebra $\C[\tilde{X}_{\param}]$ is finitely generated and free as a module
over $\C[\param]$. Moreover, we have $\C[\tilde{X}_{\param}]/(\param)=\C[X]$.
\item The group $W$ acts on $\C[\tilde{X}_{\param}]$ by graded Poisson algebra
automorphisms and the action is compatible with that on $\param$. The induced
action on $\C[X]$ is trivial.
\end{enumerate}
\end{Prop}
\begin{proof}
%
Let $\mathcal{X}$ be the flat deformation of $X$ over $\param/W$ established in
\cite{Namikawa1}, before Theorem 5.5 there. Combining \cite[(14)]{Namikawa1} with
\cite[Theorem 1.1]{Namikawa2} we get the following commutative diagram

\begin{picture}(90,30)
\put(2,22){$\tilde{X}_\param$}
\put(42,22){$\mathcal{X}$}
\put(4,2){$\param$}
\put(40,2){$\param/W$}
\put(5,20){\vector(0,-1){14}}
\put(43,20){\vector(0,-1){14}}
\put(10,23){\vector(1,0){28}}
\put(10,3){\vector(1,0){28}}
\end{picture}

Consider the induced morphism
\begin{equation}\label{eq:morphism_families}
\tilde{X}_\param\rightarrow \param\times_{\param/W}\mathcal{X}.
\end{equation}
It is projective by the construction.
We claim that
\begin{itemize}
\item[(*)]
the target is normal and (\ref{eq:morphism_families}) is birational.
\end{itemize}
This claim implies that $\C[\tilde{X}_\param]\xrightarrow{\sim} \C[\param\times_{\param/W}\mathcal{X}]$.
This, in turn, implies (2). Also note that $\mathcal{X}$ is flat over $\param/W$ by the construction
in \cite[(5.4)]{Namikawa1} and carries a $\C^\times$-action compatible with that on $\param/W$ because $\mathcal{X}$ is obtained by  algebraization of a formal deformation. So $\C[\mathcal{X}]$
is a flat graded module over $\C[\param]^W$. Hence it is free. So (*) implies (1) as well.

Let us prove that $\param\times_{\param/W}\mathcal{X}$ is normal.
As we have seen in the proof of Proposition \ref{Prop:coh_vanish}, the variety $X$ is Cohen-Macaulay.
Since $\param\times_{\param/W}\mathcal{X}$ is flat over $\param$, deforms $X$ and comes with
a contracting $\C^\times$-action, the variety $\mathcal{X}$ is also Cohen-Macaulay. To show that it
is normal we need to check that it is smooth outside of a  codimension $2$ locus.
Equivalently, we need to show that the morphism  $\param\times_{\param/W}\mathcal{X}\rightarrow \param$
is smooth outside of a codimension $2$ locus. The variety $X$ is normal hence smooth outside
of a codimension $2$ locus. Thanks to the contracting $\C^\times$-action the same is true for the
other fibers of  $\param\times_{\param/W}\mathcal{X}\rightarrow \param$. So the variety
$\param\times_{\param/W}\mathcal{X}$ is smooth outside of a codim $2$ locus. Hence it is normal.

Now we prove that (\ref{eq:morphism_families}) is birational.
For a Zariski generic point $\lambda\in \param$ the induced morphism of fibers $\tilde{X}_\lambda\rightarrow
\mathcal{X}_{W\lambda}$ is an isomorphism, see \cite[Theorem 5.5,(c)]{Namikawa1}.
This shows (\ref{eq:morphism_families}) is birational and completes the proof.
\end{proof}

Let us write $X_{\param}$ for $\operatorname{Spec}(\C[\tilde{X}_{\param}])$
(i.e., $\param\times_{\param/W}\mathcal{X}$ from the previous proof)
and $X_\lambda$ for the fiber of $X_{\param}\rightarrow \param$ over $\lambda$.

Now let us examine the situation when $X_\lambda,X_{\lambda'}$ give isomorphic filtered Poisson
deformations.

\begin{Prop}\label{Prop:deform_iso}
We have $\C[X_\lambda]\cong \C[X_{\lambda'}]$ as filtered Poisson deformations
of $\C[X]$ if and only if $\lambda'\in W\lambda$.
\end{Prop}

This follows from results of \cite{Namikawa2} (but also can be proved independently). We postpone the proof until the next section. Note also that our present setting is different from that of
Proposition \ref{Prop:KV_deform}: there we considered
varieties that are $\Q$-factorial terminal and Poisson deformations
were parameterized by $H^2(?,\C)$. Here $X$ is, generally, not of this form.

\begin{Rem}\label{Rem:param}
Recall the decomposition $\param= H^2(X^{reg},\C)\oplus \bigoplus_{i=1}^k \param_i$.
We can write a deformation parameter $\lambda$
as $(\lambda_0,\lambda_1,\ldots,\lambda_k)$ according to this decomposition. Then $\lambda_0\in H^2(X^{reg},\C)$ controls the deformation
of $X^{reg}$, while $\lambda_i$ controls the deformation of
the slice $\Sigma_i, i=1,\ldots,k$ induced by the deformation of $X$.
We will not use this claim in what follows.
\end{Rem}

%

\subsection{Universal deformation of $X$}\label{SS_univ_deform}
In \cite[Theorem 5.1]{Namikawa1}, Namikawa has proved that the Poisson deformation functor of $X$
is unobstructed. Here is variant of his result for graded Poisson deformations.


\begin{Prop}\label{Prop:X_univ_graded}
Recall that $X$ is a conical symplectic singularity.
There is a positively graded polynomial algebra $B$ and a graded Poisson $B$-algebra
$A$ such that $\C\otimes_B A\xrightarrow{\sim} \C[X]$ with the following universal
property:
\begin{itemize}
\item For any  finitely generated positively graded algebra $B'$ and a graded Poisson $B'$-algebra
$A'$ such that $\C\otimes_{B'}A'\xrightarrow{\sim}\C[X]$ there is a unique
graded algebra homomorphism $B\rightarrow B'$ and a $B'$-linear Poisson
graded algebra isomorphism $B'\otimes_B A\xrightarrow{\sim}A'$ intertwining
the isomorphisms $\C\otimes_{B'}A'\xrightarrow{\sim}\C[X]\xrightarrow{\sim} \C\otimes_B A$.
\end{itemize}
In the notation of \cite[(14)]{Namikawa1} we have $B=\C[x_1,\ldots,x_d]$ and
$A=\C[\mathcal{X}]$.
\end{Prop}
\begin{proof}
By \cite[Theorem 5.1]{Namikawa1}, the formal neighborhood of $X$ in $\mathcal{X}$
is the universal formal Poisson deformation of $X$. Arguing as in Step 2 of the proof of
Proposition \ref{Prop:KV_deform}, we see that this formal neighborhood is a universal
graded formal Poisson deformation. Then an easy analog of Step 3 of that proof shows
that $\mathcal{X}$ is a universal graded Poisson deformation. Translating to the language
of algebras, we get the claim of this proposition.
%
\end{proof}

\begin{Cor}\label{Cor:X_univ_deform}
We have $B=\C[\param]^W$ and $A=\C[X_{\param}]^W$.
\end{Cor}
\begin{proof}
By \cite[Theorem 1.1]{Namikawa2}, $B=\C[\param]^W$.
(\ref{eq:morphism_families}) shows  that $\C[\mathcal{X}]$
is $\C[X_{\param}]^W$.
\end{proof}

Using this corollary we can give a proof of Proposition \ref{Prop:deform_iso}.

\begin{proof}[Proof of Proposition \ref{Prop:deform_iso}]
Recall, see the discussion in the beginning of Section \ref{SS_deform_Q_fact}, that the filtered deformations of $\C[X]$ are in one-to-one correspondence with graded deformations of $X$
over $\C[\hbar]$. In particular, $\C[X_\lambda]$ corresponds to the graded deformation
$\C\times_{\param/W}X_\param/W$, where the morphism $\C\rightarrow \param/W$ is given
by $z\mapsto W(\lambda z^d)$. Thanks to the universal property of $X_\param/W$,
Proposition \ref{Prop:X_univ_graded}, Corollary \ref{Cor:X_univ_deform}, we see that
an isomorphism $\C[X_\lambda]\cong \C[X_{\lambda'}]$ of filtered Poisson deformation is equivalent to
the coincidence of the morphisms $\C\rightarrow \param/W$ coming from $\lambda$ and $\lambda'$, which, in turn, is equivalent to
$\lambda'\in W\lambda$.
\end{proof}

\subsection{Poisson derivations}\label{SS_Poiss_deriv}
Here is the main result of this section.

\begin{Prop}\label{Prop:Poisson_der}
Let, as before, $X$ be a conical symplectic singularity.
 Then all Poisson derivations of $\C[X]$ are Hamiltonian.
\end{Prop}

The crucial step in the proof is the following lemma (that is classical for
symplectic resolutions) whose proof in the  generality we need was communicated to
me by Kaledin.

\begin{Lem}\label{Lem:homol_vanish}
We have $H^1(\tilde{X}^{reg},\C)=0$.
\end{Lem}
\begin{proof}
Let us write $\mathcal{O}^{an}$ for the sheaves of holomorphic functions on $\tilde{X}$
and its open subvarieties.
We have the exponential exact sequence
$$H^0(\tilde{X}^{reg}, \mathcal{O}^{an})\rightarrow H^0(\tilde{X}^{reg}, (\mathcal{O}^{an})^\times)
\rightarrow H^1(\tilde{X}^{reg},\C)\rightarrow H^1(\tilde{X}^{reg}, \mathcal{O}^{an})$$
First of all, we claim that $H^1(\tilde{X}^{reg}, \mathcal{O}^{an})=0$.
Indeed, we have an exact sequence
$$H^1(\tilde{X}, \mathcal{O}^{an})\rightarrow H^1(\tilde{X}^{reg}, \mathcal{O}^{an})
\rightarrow H^2_{\tilde{X}^{sing}}(\tilde{X},\mathcal{O}^{an})$$
As in the algebraic situation, the first and the third terms are zero (see, e.g., \cite[Appendix]{Mitya_extension}) and so
$H^1(\tilde{X}^{reg}, \mathcal{O}^{an})$ is zero. So we get an exact sequence
\begin{equation}\label{eq:exact_exp} H^0(\tilde{X}^{reg}, \mathcal{O}^{an})\rightarrow H^0(\tilde{X}^{reg}, (\mathcal{O}^{an})^\times)
\rightarrow H^1(\tilde{X}^{reg},\C)\rightarrow 0.
\end{equation}

Note that, by the Hartogs theorem, we have $$H^0(\tilde{X}^{reg}, \mathcal{O}^{an})
\cong H^0(\tilde{X}, \mathcal{O}^{an}),
H^0(\tilde{X}^{reg}, (\mathcal{O}^{an})^\times)\cong H^0(\tilde{X}, (\mathcal{O}^{an})^\times).$$
From the analog of (\ref{eq:exact_exp}) for $\tilde{X}$, we conclude that
$H^1(\tilde{X}^{reg},\C)\cong H^1(\tilde{X},\C)$. So we need to show that the latter space
vanishes.

Let $\pi:\overline{X}\rightarrow \tilde{X}$ be a resolution of singularities.
By \cite[Corollary 1.5]{Kaledin_survey}, we have
$R^1(\rho\circ\pi)_*\C_{\overline{X}}=0$ (recall that $\rho$ stands for
the morphism $\tilde{X}\rightarrow X$) and $R^1\pi_*\C_{\overline{X}}=0$.
Since, obviously, $\pi_* \C_{\overline{X}}=\C_{\tilde{X}}$, we use the standard spectral sequence
for the composition of derived functors to check that $R^1\rho_*\C_{\tilde{X}}=0$.
Since $X$ is conical (and hence contractible), this implies that $H^1(\tilde{X},\C)=0$.
\end{proof}

\begin{proof}[Proof of Proposition \ref{Prop:Poisson_der}]
We will show that all Poisson vector fields on $X$ lift to $\tilde{X}^{reg}$,
then our claim will follow from Lemma \ref{Lem:homol_vanish}.

Recall the open subvariety $X^1\subset X$ from the proof of Lemma \ref{Lem:Cartan_compute}.
Let $\tilde{X}^1$ be the preimage of $X^1$ in $\tilde{X}$.
We claim that
\begin{equation}\label{eq:codim_inequality}
\operatorname{codim}_{\tilde{X}}\tilde{X}\setminus \tilde{X}^1\geqslant 2.
\end{equation}
Assume the contrary. Let $Z$ be an irreducible divisor inside $\tilde{X}\setminus \tilde{X}^1$.
Let $D_1,\ldots,D_k$ denote the irreducible components of the preimage of
$X^1\setminus X^{reg}$ in $\tilde{X}$. Since $H^i(\tilde{X}^{reg}, \mathcal{O}^{an})=0$
for $i=1,2$, we see that $H^2(\tilde{X}^{reg},\C)$ coincides with the complexified
analytic Picard group of $\tilde{X}^{reg}$. Let us write $[D_1],\ldots,[D_k],[Z]$
for the classes of the line bundles corresponding to these divisors in $H^2(\tilde{X}^{reg},\C)$.
Recall, Lemma \ref{Lem:Cartan_compute}, that $H^2(\tilde{X}^{reg},\C)=
H^2(X^{reg},\C)\oplus \operatorname{Span}([D_1],\ldots,[D_k])$.
Since $[Z]$ projects  to zero in $H^2(X^{reg},\C)$, it must lie
in  $\operatorname{Span}([D_1],\ldots,[D_k])$.

Thanks to \cite[Corollary 1.4.3]{BCHM}, we can contract all divisors
in $\tilde{X}$ but $Z$, we get a normal  variety $X'$ with a
morphism $X'\rightarrow X$. We also have a rational map $\tilde{X}\rightarrow X'$
that is defined outside of a codimension $2$ locus in $\tilde{X}$ because $\tilde{X}$
is proper over $X$.
It is easy to see that $Z$ defines a nontorsion class in the analytic
Picard group of $X'^{reg}$. Indeed, if this class is torsion, then we have an analytic function
on $X'$ whose vanishing locus is exactly $Z$. It must be pulled back from an analytic function on
$X$ whose vanishing locus has complement of codimension at least $2$ and is nonempty, which is impossible. Since $Z$ defines a nontorsion class in the analytic
Picard group of $X'^{reg}$, we apply pullback to $\tilde{X}$ to  see that $[Z]\not\in \operatorname{Span}([D_1],\ldots,[D_k])$,
a contradiction. This shows (\ref{eq:codim_inequality}).

Let $\xi$ be a Poisson
vector field on $X$, equivalently, on $\pi^{-1}(X^{reg})\subset \tilde{X}^{reg}$.
%
What we need to show is that $\xi$ extends to a regular vector field on $\pi^{-1}(X^{\wedge_x})$,
where $x$ is a point in a codimension $2$ leaf in $X$. This in turn will follow if
we check that the restriction of $\xi$ to $X^{\wedge_x}$ is Hamiltonian. But $X^{\wedge_x}$
is a symplectic quotient singularity. Note that any Poisson vector field on $(V/\Gamma)^{\wedge_0}$
lifts to a $\Gamma$-invariant Poisson vector field on $V^{\wedge_0}$. It follows that
any Poisson vector field on $(V/\Gamma)^{\wedge_0}$  and hence on $X^{\wedge_x}$ is
Hamiltonian. This completes the proof.
\end{proof}

%

\section{Quantizations of symplectic singularities}\label{S_quant}
\subsection{Quantizations and period maps}\label{SS_quant_period}
This section is a quantum counterpart of Section \ref{SS_filt_deform}.

Let $A$ be a graded Poisson algebra with bracket of degree $-d$, where $d$ is a positive
integer. By a {\it filtered quantization} of $A$ one means a pair $(\A,\iota)$,
where
\begin{itemize}
\item  $\A$ is a filtered associative algebra, $\A=
\bigcup_{i\geqslant 0}\A_{\leqslant i}$, such that $[\A_{\leqslant i},\A_{\leqslant j}]
\subset \A_{\leqslant i+j-d}$,
\item and $\iota$ is  a graded Poisson algebra
isomorphism $\gr\A\xrightarrow{\sim} A$.
\end{itemize}

By an isomorphism $\psi$ of filtered
quantizations $(\A,\iota),(\A',\iota')$ we mean a filtration preserving algebra
isomorphism $\psi:\A\rightarrow \A'$ such that $\gr\psi$ intertwines $\iota,\iota'$.

Our goal is to classify the filtered quantizations of $\C[X]$, where $X$
is a conical symplectic singularity. As with filtered Poisson deformations, we are going to produce quantizations
of $\C[X]$  from those of $\tilde{X}$. We now explain what one means by a filtered quantization of $\tilde{X}$.

Let $X'$ be a Poisson scheme equipped with a $\C^\times$-action that rescales the Poisson bracket by
$t\mapsto t^{-d}$. We assume that every point in $X'$ has a $\C^\times$-stable affine open neighborhood.
By a result of Sumihiro, \cite{Sumihiro}, this is the case when $X'$ is normal.
By a {\it filtered quantization}  $\mathcal{D}$ of $X'$ we mean a sheaf of filtered associative algebras in the conical topology on $X'$, where the filtration
is complete and separated, together with an isomorphism $\iota:\gr\mathcal{D}\xrightarrow{\sim} \mathcal{O}_{X'}$
of graded Poisson algebras. We note that quantizations of $\C[X]$ are in a natural bijection with
quantizations of $X$ via taking global sections to get from $X$ to $\C[X]$ and taking the microlocalization to get back, see the discussion on ``Variations on formal microlocalization...'' in \cite[Section 1]{Ginzburg_CC}.

We can also talk about graded formal quantizations. These are pairs $(\mathcal{D}_\hbar,\iota)$. Here
$\mathcal{D}_\hbar$ is a sheaf a $\C[[\hbar]]$-algebras in the Zariski topology on $X$ that is flat over $\C[[\hbar]]$ and complete and separated in the $\hbar$-adic topology. We require that $[\mathcal{D}_\hbar,\mathcal{D}_\hbar]\subset \hbar^d \mathcal{D}_\hbar$, which gives rise to
a Poisson bracket on $\mathcal{D}_\hbar/(\hbar)$. We also require that $\C^\times$ acts on $\mathcal{D}_\hbar$ by $\C$-algebra automorphisms with $t.\hbar=t\hbar$. For $\iota$ we take a
$\C^\times$-equivariant Poisson isomorphism $\mathcal{D}_\hbar/(\hbar)\xrightarrow{\sim}\mathcal{O}_X$.

For a $\C^\times$-stable open affine subset $U\subset X'$ the algebra $\Gamma(U,\mathcal{D}_\hbar)$ acquires a $\C^\times$-action. The action on $\Gamma(U,\mathcal{D}_\hbar)/(\hbar^k)=
\Gamma(U,\mathcal{D}_\hbar/(\hbar^k))$ is not required to be rational. However, we can replace the $\C^\times$-action on $\mathcal{D}_\hbar$ canonically to make
the action on $\Gamma(U,\mathcal{D}_\hbar)/(\hbar^k)$ rational. Let us explain how to do this.
Consider the subgroup $\C^\times_{fin}\subset \C^\times$ of finite order elements.  It acts diagonalizably in any, not necessarily rational, representation of $\C^\times$.
Note that $\Gamma(U,\mathcal{D}_\hbar)/(\hbar^k)$ is $\C^\times$-equivariantly filtered with successive
quotients $\C[U]$. It follows that the characters of $\C^\times_{fin}$ in  $\Gamma(U,\mathcal{D}_\hbar)/(\hbar^k)$ are restrictions of characters of $\C^\times$. There is a unique rational action of $\C^\times$ on $\Gamma(U,\mathcal{D}_\hbar)/(\hbar^k)$ extending the action of $\C^\times$, it is by algebra automorphisms. These actions glue to an action on $\mathcal{D}_\hbar$ as in the previous paragraph. Below we will always assume that the action of $\C^\times$ on $\Gamma(U,\mathcal{D}_\hbar)/(\hbar^k)$ is rational for all open affine subsets $U\subset X'$.

With this assumption there is a natural bijection between the sets of isomorphism classes of filtered
quantizations and of graded formal quantizations.  Namely, take a filtered quantization $\mathcal{D}$ of $\Str_{X'}$. Form the Rees sheaf $R_\hbar(\mathcal{D})$ of $\mathcal{D}$ and complete it in the $\hbar$-adic topology. Then we extend
this completion (that is still a sheaf in the conical topology) to the Zariski
topology by localizing. We get a graded formal quantization in the sense of \cite[Section 2.2]{quant_iso}.
Conversely, given a graded formal quantization $\mathcal{D}_\hbar$ of $X'$ we can restrict
it to the conical topology. Inside this restriction we can consider the subsheaf $\mathcal{D}_{\hbar,fin}$ of $\C^\times$-finite sections. Because of the rationality assumption on the action of
$\C^\times$ on $\mathcal{D}_\hbar$, we recover $\mathcal{D}_\hbar$ as the $\hbar$-adic completion of
$\mathcal{D}_{\hbar,fin}$. The sheaf $\mathcal{D}_{\hbar,fin}/(\hbar-1)$ comes with a natural filtration. This filtration is complete and separated because the $\hbar$-adic topology on $\mathcal{D}_\hbar$ is so. So the resulting sheaf is a filtered quantization of $X'$. It is straightforward to see that these procedures define mutually inverse bijections between
the sets of isomorphism classes of filtered and of graded formal quantizations.

Now suppose $X'$ is smooth and symplectic. In this case a filtered quantization
$\mathcal{D}$ defines a class in $H^2(X',\C)$ to be called  the {\it period}
of $\mathcal{D}$, see
\cite[Section 4]{BK_quant}, where the case of formal quantizations was considered,
and \cite[Section 2]{quant_iso} that treats graded formal quantizations.

The following proposition
should be thought of as a quantum version of  Proposition \ref{Prop:KV_deform}
and Remark \ref{Rem:Namikawa_KV}.

\begin{Prop}\label{Prop:BK_quant}
Assume that $X'$ is smooth and  symplectic and $H^i(X',\mathcal{O}_{X'})=0$ for $i=1,2$.
Then the following claims hold:
\begin{enumerate}
\item
taking the period
defines a bijection between the isomorphism classes of filtered quantizations
and $\param:=H^2(X',\C)$.
\item Moreover,  there is a ``universal quantization'' $\mathcal{D}'_{\param}$.
This a sheaf of filtered flat $\C[\param]$-algebras in the conical topology with complete and separated
filtration and an isomorphism $\gr \mathcal{D}'_\param/(\param)\xrightarrow{\sim} \Str_{X'}$.
The completion of $\gr\mathcal{D}'$ at $0\in \param$ is the deformation of the universal deformation of $X'$, see Remark \ref{Rem:Namikawa_KV}. Moreover, the quantization $\mathcal{D}'_\lambda$ corresponding to $\lambda$
is obtained by specializing $\mathcal{D}'_\param$ to $\lambda$.
\end{enumerate}
\end{Prop}
\begin{proof}
Both (1) and (2) were proved in \cite[Corollary 2.3.3]{quant_iso} for the graded formal quantizations
of $X'$. Let us elaborate on the proof of (2) in the setting of filtered quantizations. Take the universal graded formal deformation of $X'$,
see Remark \ref{Rem:Namikawa_KV}, and its canonical quantization as in \cite[Corollary 2.3.3]{quant_iso}.
Pull it back to $X'$ getting a sheaf of $\C[[\param,\hbar]]$-algebras on $X'$ deforming $\Str_{X'}$.
We can convert it into a sheaf of filtered algebras as discussed before the proposition.
The claim about the specializations of $\mathcal{D}'_{\param}$ (in the language of graded formal quantizations) is \cite[Proposition 3.1]{BPW}.
\end{proof}

Now consider the variety $\tilde{X}$, a $\Q$-factorial terminalization of
a conical symplectic singularity $X$, and set $\param:=H^2(\tilde{X}^{reg},\C)$.

\begin{Cor}\label{Cor:Q_quant}
The following claims hold:
\begin{enumerate}
\item The filtered quantizations of $\tilde{X}$ are classified by $\param$.
The filtered quantization $\mathcal{D}_\lambda$ corresponding to $\lambda$ is the pushforward
of the filtered quantization $\mathcal{D}'_\lambda$ of $\tilde{X}^{reg}$ corresponding to $\lambda$.
\item Let $\mathcal{D}_\param$ denote the pushforward of the universal quantization $\mathcal{D}'_\param$ from $\tilde{X}^{reg}$ to $\tilde{X}$. Then $\mathcal{D}_\lambda$
    is the specialization of $\mathcal{D}_\param$ to $\lambda\in \param$.
\end{enumerate}
\end{Cor}
\begin{proof}
As for Poisson deformations over Artinian bases in Remark \ref{Rem:Namikawa_KV}, the sheaf theoretic pullback and pushforward define mutually inverse bijections between the sets of isomorphism classes of   filtered quantizations of $\tilde{X}^{reg}$
and of $\tilde{X}$, compare to \cite[Proposition 3.4]{BPW}. Now we use (1) of Proposition \ref{Prop:BK_quant} to establish (1).

Now we prove (2). Consider the  sheaf $\mathcal{D}'_{\param,\hbar}$ as in the proof of (2) of
Proposition \ref{Prop:BK_quant}. For any Artinian quotient $B$ of $\C[[\param,\hbar]]$, the base change
$B\otimes_{\C[[\param,\hbar]]}\mathcal{D}'_{\param,\hbar}$ is filtered with quotients isomorphic  to
$\Str_{\tilde{X}^{reg}}$. So $R^1\iota_*(B\otimes_{\C[[\param,\hbar]]}\mathcal{D}'_{\param,\hbar})=0$.
Since the filtration on $\mathcal{D}'_{\param,\hbar}$ induced by the maximal ideal in $\C[[\param\hbar]]$
is complete and separated, we get  $R^1\iota_*\mathcal{D}'_{\param,\hbar}=0$. It follows that
$\iota_*\mathcal{D}'_{\param,\hbar}$ is a deformation of $\Str_{\tilde{X}}$ over $\C[[\param,\hbar]]$.
The sheaf $\iota_* \mathcal{D}_\param$ of filtered algebras is obtained from $\iota_*\mathcal{D}'_{\param,\hbar}$ by passing to $\C^\times$-finite sections and taking the quotient
by $\hbar-1$. Specializing this sheaf to $\lambda$ we get a filtered quantization of $\tilde{X}$
whose restriction to $\tilde{X}^{reg}$ is $\mathcal{D}'_\lambda$. By the construction in (1),
this is the quantization $\mathcal{D}_\lambda$.
\end{proof}


\subsection{Quantizations from $\Q$-terminalizations}\label{SS_quant_Q_term}
Now we will produce some filtered quantizations of $\C[X]$ following \cite[Section 3]{BPW} and state our main classification result.

Set $\A_\lambda:=\Gamma(\mathcal{D}_\lambda), \A_{\param}:=\Gamma(\mathcal{D}_\param)$,
where $\mathcal{D}_\lambda,\mathcal{D}_{\param}$ were introduced in the previous section.
The following is a quantum version of Proposition \ref{Prop:deform_X}.

\begin{Prop}\label{Prop:quant_X}
The following is true:
\begin{enumerate}
\item The algebras $\A_\lambda,\A_{\param}$ are filtered quantizations of $\C[X],\C[X_{\param}]$,
respectively. Moreover, $\A_\lambda$ is the specialization of $\A_\param$ to $\lambda$.
\item The group $W$ acts on $\A_{\param}$ by filtered algebra automorphisms so that the associated
graded action on
$\C[X_{\param}]$ coincides with the action from Proposition \ref{Prop:deform_X}.
Moreover, the actions of $W$ on $\A_{\param}$ and on $\param$ are compatible.
\end{enumerate}
\end{Prop}
\begin{proof}
The proof of (1) is standard, let us provide it for reader's convenience.

We start by proving the claim for $\A_\lambda$.
Recall,  Proposition \ref{Prop:coh_vanish}, that $\C[\tilde{X}]=\C[X]$ and
$H^i(\tilde{X},\mathcal{O}_{\tilde{X}})=0$ for $i>0$, in particular for $i=1$.
Consider the $\hbar$-adic completion $\mathcal{D}_{\lambda,\hbar}$ of the Rees sheaf
of $\mathcal{D}_\lambda$. This is a graded formal deformation of $\mathcal{O}_{\tilde{X}}$.
As in the proof of (2) of Corollary \ref{Cor:Q_quant}, $H^1(\tilde{X},\mathcal{D}_{\lambda,\hbar})=0$. This, in turn, implies that
$\Gamma(\mathcal{D}_{\lambda,\hbar})/(\hbar)\xrightarrow{\sim} \C[\tilde{X}]$.
The algebra $\A_\lambda=\Gamma(\mathcal{D}_\lambda)$ is obtained from
$\Gamma(\mathcal{D}_{\lambda,\hbar})$ by taking the $\C^\times$-finite sections and then
passing to quotient by $\hbar-1$. It follows that $\gr \A_\lambda\xrightarrow{\sim} \C[\tilde{X}]$.
This is an isomorphism of graded Poisson algebras. So $\A_\lambda$ is a quantization of
$\C[\tilde{X}]$.

Now we prove that $\A_\param$ is a quantization of $\C[X_\param]$. Let $\mathfrak{m}$
denote the maximal ideal of $0$ in $\C[\param]$.  For $k>0$, set
$\tilde{X}_{\param,k}:=\operatorname{Spec}(\C[\param]/\mathfrak{m}^k)\times_{\param}\tilde{X}_\param$. Since $H^i(\tilde{X},\mathcal{O}_{\tilde{X}})=0$ for all $i>0$, we use the long exact sequence in cohomology
to show that  $H^i(\tilde{X}_{\param,k},\mathcal{O}_{\tilde{X}_{\param,k}})=0$
for $i>0$. Thanks to the formal function theorem,  the completions
of $H^i(\tilde{X}_\param, \mathcal{O}_{\tilde{X}_\param})$ vanish. Thanks to the contracting $\C^\times$-action on $\tilde{X}_\param$, we conclude that
$H^i(\tilde{X}_\param, \mathcal{O}_{\tilde{X}_\param})=0$ for $i>0$. Now the proof that $\A_{\param}$
is a filtered quantization of $\C[\tilde{X}_\param]$ repeats the proof for $\A_\lambda$ in the previous paragraph.

Let us show that $\A_\lambda$ is the specialization of $\A_\param$. Let $\param_1\subset \param$
be an affine space. We write $\mathcal{D}_{\param_1,\hbar}$ for the completion of the Rees sheaf of
the specialization $\mathcal{D}_{\param_1}$ of $\mathcal{D}_\param$. Similarly to the previous paragraph, the sheaf $\mathcal{D}_{\param_1,\hbar}$ has no higher cohomology. Using the long exact sequences in cohomology and the descending induction on $\dim \param_1$ one proves that $\Gamma(\mathcal{D}_{\param_1,\hbar})$ is the specialization of $\Gamma(\mathcal{D}_{\param,\hbar})$.
In particular, $\A_\lambda$ is the specialization of $\A_{\param}$ to $\lambda$.

(2) is what is proved in the proof of \cite[Proposition 3.10]{BPW} (note that the statement of that
proposition is formally weaker).  The smooth locus $X_\param^{reg}$  in $X_\param$
is preserved by $W$. As argued in the proof of \cite[Proposition 3.10]{BPW},
since the restriction of $\mathcal{D}_\param$ to $X_\param^{reg}$ is  the canonical
quantization of $X_\param^{reg}$, the action of $W$ on $\mathcal{O}_{X_\param^{reg}}$
lifts to $\mathcal{D}_\param|_{X_\param^{reg}}$
and hence to $\Gamma(\mathcal{D}_\param|_{X_\param^{reg}})$. And, as argued in the proof of
\cite[Proposition 3.10]{BPW}, the global sections coincide with $\A_\param$.
By the construction, the action of $W$ on $\gr\A_\param$ coincides with the action on
$\C[\tilde{X}_\param]$.
\end{proof}

The following is one of the main results of this paper.

\begin{Thm}\label{Thm:quant_classif}
Any filtered quantization of $\C[X]$ is isomorphic
to $\A_\lambda$ for some $\lambda$. Moreover,
$\A_\lambda,\A_{\lambda'}$ are isomorphic as filtered quantizations if and only if $\lambda'\in W\lambda$.
\end{Thm}

In fact, the algebra $\A_{\param}^W$ enjoys a universal property similar to that of $\C[X_{\param}]^W$.
This property is the subject of the next proposition, which is a more technical version
of Theorem \ref{Thm:quant_classif} and also implies that theorem.

%

\begin{Prop}\label{Prop:quant_univ_property}
Let $B'$ be a finitely generated commutative positively graded algebra and $A'$ be a graded Poisson flat $B'$-algebra
such that $\C\otimes_{B'}A'=\C[X]$. Further, let $\A'$ be a $B'$-algebra that is a filtered quantization
of $A'$ such that the structure map $B'\rightarrow \A'$ is a filtered
algebra homomorphism whose associated graded map is $B'\rightarrow A'$. Then there is a unique filtered algebra homomorphism $\C[\param]^W\rightarrow B'$
with the following properties:
\begin{enumerate}
\item The associated graded of this homomorphism comes from the universal property of
$\C[X_{\param}]^W/\C[\param]^W$ (see Proposition \ref{Prop:X_univ_graded} and
Corollary \ref{Cor:X_univ_deform}).
\item We have a $B'$-linear isomorphism $B'\otimes_{\C[\param]^W}\A_{\param}^W\xrightarrow{\sim}
\A'$ of filtered quantizations of $A'$.
\end{enumerate}
\end{Prop}

Proposition \ref{Prop:quant_univ_property} follows from Proposition
\ref{Prop:univer_AY} and Lemma \ref{Lem:Struct_Y} that will be stated and proved later.

Our last task for this section is to explain how we can recover the parameter $W\lambda$
from a filtered quantization $\A_\lambda$. In order to do this, we will need to recall
the construction and properties of  quantum slice algebras following \cite[Section 3.2]{perv}.

Let $x\in \mathcal{L}_i$. Then we can form the completions $\C[X]^{\wedge_x}, \C[T_x\mathcal{L}_i]^{\wedge_0}, \C[\Sigma_i]^{\wedge_0}$. By \cite[Theorem 2.3]{Kaledin_symplectic},
we have an isomorphism
\begin{equation}\label{eq:classical_decomposition}
\C[X]^{\wedge_x}\cong \C[T_x\mathcal{L}_i]^{\wedge_0}\widehat{\otimes}
\C[\Sigma_i]^{\wedge_0}.
\end{equation}
One can lift (\ref{eq:classical_decomposition}) to a quantum level.
Form the Rees algebra $R_\hbar(\A_\lambda)$
of $\A_\lambda$ and consider the completion  $R_\hbar(\A_\lambda)^{\wedge_x}$ at the maximal
ideal of $R_\hbar(\A_\lambda)$ pulled back from the maximal ideal of $x$ in $\C[X]$ under
$R_\hbar(\A_\lambda)\twoheadrightarrow \C[X]$.  Then, by \cite[Lemma 3.3]{perv}, we have
the following decomposition
\begin{equation}\label{eq:quantum_decomposition}
R_\hbar(\A_\lambda)^{\wedge_x}\cong
\mathbb{A}_{\hbar}(T_x\mathcal{L}_i)^{\wedge_0}\widehat{\otimes}_{\C[[\hbar]]}
R_\hbar(\A_{\lambda_i}^i)^{\wedge_0},
\end{equation}
lifting (\ref{eq:classical_decomposition}),
where $\mathbb{A}_\hbar$ stands for the homogenized Weyl algebra of a symplectic vector space.

\begin{Prop}\label{Rem:quant_param}
In a quantization parameter $\lambda=(\lambda_0,\ldots,\lambda_k)$, the component $\lambda_0$ is the period of the quantization of $X^{reg}$ given by restricting  $\A_\lambda$. The orbit  $W_i\lambda_i$ is uniquely recovered from the formal quantization $R_\hbar(\A^i_{\lambda_i})^{\wedge_0}$ of $\C[\Sigma_i]^{\wedge_0}$.
\end{Prop}
\begin{proof}
The claim about $\lambda_0$
is a consequence of the following general claim: if $X'$ is such as in Proposition
\ref{Prop:BK_quant}, $U\subset X'$ is an open $\C^\times$-stable subvariety and $\mathcal{D}'$
is the filtered quantization of $X'$, then the period of the restriction $\mathcal{D}'|_U$
is the pullback of the period of $\mathcal{D}'$ to $U$. This general claim is a direct consequence of
the construction
of the period in \cite[Section 4]{BK_quant}. To get the characterization of $\lambda_0$
we notice that the restriction of $\A_\lambda$ to $X^{reg}$ coincides with the restriction
of the quantization $\mathcal{D}'_\lambda$ of $\tilde{X}^{reg}$ to $X^{reg}$. This is because the
morphism $\tilde{X}\rightarrow X$ is an isomorphism over $X^{reg}$.

Now we prove the claim about $W_i\lambda_i$ for $i>0$. The proof is in two steps.

{\it Step 1}.
First of all, we claim that if
$R_\hbar(\A^i_{\lambda_i})^{\wedge_0}$ and $R_\hbar(\A^i_{\lambda_i'})^{\wedge_0}$ are isomorphic as formal
quantizations of $\C[\Sigma_i']$, then $\A^i_{\lambda_i}$ and $\A^i_{\lambda_i'}$ are isomorphic
as filtered quantizations of $\C[\Sigma_i]$. This will follow if we  prove the following claim
\begin{itemize}
\item
any $\C[[\hbar]]$-algebra isomorphism $R_\hbar(\A^i_{\lambda_i})^{\wedge_0}\xrightarrow{\sim}R_\hbar(\A^i_{\lambda_i'})^{\wedge_0}$ that is the identity modulo $\hbar$
is the composition of a $\C[[\hbar]]$-algebra automorphism, say $\alpha$, of
$R_\hbar(\A^i_{\lambda_i'})^{\wedge_0}$ that is the identity modulo $\hbar$ and a
$\C[[\hbar]]$-algebra isomorphism
$R_\hbar(\A^i_{\lambda_i})^{\wedge_0}\xrightarrow{\sim} R_\hbar(\A^i_{\lambda_i'})^{\wedge_0}$
intertwining the $\C^\times$-actions.
\end{itemize}

Consider the group $\mathcal{G}$ of algebra automorphisms of
$R_\hbar(\A^i_{\lambda'_i})^{\wedge_0}$ that rescale $\hbar$. This is a pro-algebraic group.
As such it is the semidirect product of a reductive algebraic group and pro-unipotent pro-algebraic
group. The projection from the reductive part to the automorphism group of $\C[\Sigma_i]'$
is injective. It follows that every two embeddings of $\C^\times$ into $\mathcal{G}$ that give
the dilation action on $\Sigma'_i$ are conjugate by an element that is the identity on $\Sigma_i'$.
We apply this to the $\C^\times$-actions on $R_\hbar(\A^i_{\lambda_i'})^{\wedge_0}$
coming from $R_\hbar(\A^i_{\lambda_i}),R_\hbar(\A^i_{\lambda'_i})$ and let $\alpha$ be a conjugating
element. This establishes the claim above.

{\it Step 2}.
Now suppose $\A^i_{\lambda_i}\xrightarrow{\sim} \A^i_{\lambda_i'}$, an isomorphism of filtered quantizations (meaning that the associated graded isomorphism intertwines the isomorphisms with $\C[\Sigma_i]$). By  Theorem \ref{Thm:quant_classif},
we have $\widehat{W}_i\lambda_i'=\widehat{W}_i\lambda_i$. Recall, Section
\ref{SS_deform_Q_fact}, that $\widehat{W}_i$ is the Weyl group
of the same ADE type as $\Sigma_i$ and $W_i=\widehat{W}_i^\Xi$ for a suitable group $\Xi$ of diagram
automorphisms. So it is enough to show that for $\lambda_i,\lambda_i'\in \widehat{\param}_i^\Xi$, the equality $\widehat{W}_i\lambda_i'=\widehat{W}_i\lambda_i$ implies $W_i\lambda_i'=W_i\lambda_i$.

Pick fundamental weights $\varpi_j, j\in J,$ for $\widehat{W}_i$ and let $C$ denote the
Weyl chamber for $\widehat{W}_i$ spanned by the weights $\varpi_j$. The group $\Xi$ acts on $J$. Then $C^\Xi=C\cap \widehat{\param}_i^\Xi$ is spanned by $\sum_{k\in \Xi j}\varpi_j$.
This is a fundamental chamber for $W_i=\widehat{W}_i^\Xi$.
Note that the locus  $\{\lambda_i\in \widehat{\param}_i^{\Xi}| \widehat{W}_i\lambda_i \cap
\widehat{\param}_i^{\Xi}=W_i\lambda_i\}$ is the union of finitely many vector subspaces
defined over $\mathbb{Q}$. So in the proof it is enough to assume that $\lambda_i,\lambda_i'$
are real. Conjugating them by elements of $\widehat{W}_i^\Xi$ we can assume that
they lie   in $C^\Xi$, hence in $C$. Since they are $\widehat{W}_i$-conjugate, they must coincide.
\end{proof}

\subsection{Scheme $Y$}\label{SS_scheme_Y}
We start by constructing a finite type affine scheme $Y$ over $\C$ together with an action of
$\C^\times\ltimes U$, where $U$ is a unipotent group.
This scheme will, in a sense, parameterize deformations of $\C[X]$ compatible with the Poisson bracket,
and the group action will correspond to isomorphisms of deformations.

\begin{defi}\label{defi_def_data}
Let $R$ be a finitely generated commutative $\C$-algebra with $1$.
By a {\it deformation datum} (over $R$) on $R[X]$ we mean a pair $(*,\langle\cdot,\cdot\rangle)$ of $R$-bilinear
maps  $R[X]\otimes_R R[X]\rightarrow R[X]$ satisfying the following condition:
\begin{itemize}
\item[(i)] $*$ is an associative product such that $1\in R[X]$ is a unit and
$f*g-fg\in \bigoplus_{k< i+j}R[X]_k$ for any $i,j$ and $f\in R[X]_i,g\in R[X]_j$
(here, as usual, the subscript stands for the graded component of that degree).
\item[(ii)] $\langle\cdot,\cdot\rangle$ is a skew-symmetric bracket on $R[X]$ such that $\langle f,g\rangle-\{f,g\}\in
\bigoplus_{k<i+j-d}R[X]_k$ for $i,j,f,g$ as in (i).
\item[(iii)] There is $z\in R$ such that $f*g-g*f=z\langle f,g\rangle$. Note that (ii) implies that $z$ is unique.
\item[(iv)] We have $\langle f*g, h\rangle=f*\langle g,h\rangle+\langle f,h\rangle*g$ for all $f,g,h\in \C[X]$
(the Leibniz identity) and also the Jacobi identity for $\langle\cdot,\cdot\rangle$.
\end{itemize}
\end{defi}


Clearly, if $R=\C$ and  $z=0$, then $*$ is a commutative product and $\langle\cdot,\cdot\rangle$
is a Poisson bracket so that $(\C[X],*,\langle\cdot,\cdot\rangle)$ defines a filtered
Poisson deformation of $\C[X]$. If, on the other hand, $z=1$, then $\langle\cdot,\cdot\rangle$
is recovered from $*$, and $(\C[X],*)$ is a filtered quantization of $\C[X]$.

\begin{defi}\label{defi:def_data_iso}
By an isomorphism of deformation data $(*,\langle\cdot,\cdot\rangle), (*',\langle\cdot,\cdot\rangle')$
we mean an $R$-linear map $\varphi:R[X]\rightarrow R[X]$ with the following properties:
\begin{itemize}
\item[(I)] $\varphi(f)-f\in \bigoplus_{i<k}R[X]_i$ for any $k$ and $f\in R[X]_k$.
\item[(II)] $\varphi$ intertwines $*$ with $*'$, as well as $\langle\cdot,\cdot\rangle$
with $\langle\cdot,\cdot\rangle'$.
\end{itemize}
\end{defi}

Clearly, isomorphic deformation data  correspond to isomorphic
filtered Poisson deformations (for $z=0$) and quantizations (for $z=1$).

Now we proceed to constructing $Y$. For $n\in \Z_{\geqslant 0}$, set
$V_{\leqslant n}:=\bigoplus_{i=0}^n \C[X]_i$.
Let $f_1,\ldots,f_k$ be a minimal set of homogeneous
generators of $\C[X]$ and let $m$ be the maximum of the degrees of the
generators $f_i$.  Further, let
$G_1,\ldots,G_\ell$ be a minimal set of homogeneous (with respect to
the grading on $\C[X]$) relations between the generators $f_1,\ldots,f_k$.
Consider $\deg G_i$, the degree of $G_i$ with respect to  $f_1,\ldots,f_k$ and set
$$e:=1+\max(2,\deg G_1,\ldots, \deg G_\ell).$$ Set $T:=\Hom(\bigoplus_{i=1}^e
V_{\leqslant m}^{\otimes i},V_{\leqslant me})\oplus \Hom(V_{\leqslant m}^{\otimes 2},
V_{\leqslant 2m-d})$. A deformation pair defines an element $(\alpha,\beta)$ of $T$,
where $\alpha$ comes from the iterated product $*$ restricted to $V_{\leqslant m}$ and
$\beta$ comes from the bracket $\langle\cdot,\cdot\rangle$ restricted to $V_{\leqslant m}$.

We will realize $Y$ as a closed subscheme in $T$. First, note that conditions
(i)-(iii) give polynomial equations on $T$. For example, by (ii) we have
$\langle\cdot,\cdot\rangle\neq 0$. Then (iii) means that, for $f,g\in V_{\leqslant m}$,
the elements $f*g-g*f$ and
$\langle f,g\rangle$ are proportional, which results in polynomial
equations on $\alpha,\beta$. Let $Y^1$ denote the
subscheme defined by these polynomial equations. Note that, by the construction,
$z$ can be viewed as an element of $\C[Y^1]$. Indeed, both $(f,g)\mapsto f*g-g*f$
and $(f,g)\mapsto \langle f,g\rangle$ give linear maps $V_{\leqslant m}^{\otimes 2}\mapsto
V_{\leqslant 2m-d}$. The locus, say $Z$, in $\operatorname{Hom}_\C(V_{\leqslant m}^{\otimes 2},
V_{\leqslant 2m-d})^2$, where two linear maps, say $\varphi_1,\varphi_2$, are proportional
is an algebraic subscheme, on its open subscheme $Z^0$ where $\varphi_2\neq 0$, the ratio
$\varphi_1/\varphi_2$ is a regular function. Now, by the construction, we  have a morphism $Y^1\rightarrow Z^0$, the element $z\in \C[Y^1]$ is the pullback of the regular function $\varphi_1/\varphi_2$
(note that $\beta$ is nonzero thanks to (ii) in Definition \ref{defi_def_data}).

Now consider a finite type commutative algebra $R$ and an algebra homomorphism $(\alpha,\beta):\C[Y^1]\rightarrow R$
(where the meaning of $\alpha,\beta$ as before: $\alpha$ corresponds to $*$ and $\beta$
corresponds to $\langle\cdot,\cdot\rangle$).
We  construct the unital associative algebra
\begin{equation}\label{eq:A_alpha_defn}
\A_{(\alpha,\beta)}:=\left(R\otimes T(V_{\leqslant me})\right)/(x-\alpha(x), \alpha(y)-\alpha^{opp}(y)-z\beta(y)),
\end{equation}
where $x$ runs over $\bigoplus_{i=0}^e V_{\leqslant m}^{\otimes i}$,
$y$ runs over $V_{\leqslant m}^{\otimes 2}$ and we write $\alpha^{opp}$
for $\alpha|_{V_{\leqslant m}^{\otimes 2}}\circ \sigma$, where $\sigma$
is the permutation of tensor factors.

The $R$-algebra
$\A_{(\alpha,\beta)}$ comes with a filtration induced from $V_{\leqslant me}$ and
$\deg R=0$. We have a natural epimorphism
\begin{equation}\label{eq:graded_epimorphism}
R\otimes \C[X]\rightarrow \gr \A_{(\alpha,\beta)}.
\end{equation}

\begin{Lem}\label{Lem:alpha_polynom}
The condition that (\ref{eq:graded_epimorphism}) is an isomorphism is equivalent to a system of polynomial
equations on $\alpha,\beta$. These equations are independent of $R$.
\end{Lem}
\begin{proof}
The condition that (\ref{eq:graded_epimorphism}) is an isomorphism is equivalent to
\begin{itemize}
\item[(*)] For each $j\geqslant 0$, the filtered piece $(\A_{(\alpha,\beta)})_{\leqslant j}$
is a free $R$-module of rank $\dim V_{\leqslant j}$. Note that $(\A_{(\alpha,\beta)})_{\leqslant j}$ is automatically the quotient of a free $R$-module of  rank $\dim V_{\leqslant j}$ -- because
(\ref{eq:graded_epimorphism}) is an epimorphism.
\end{itemize}

Let $I$ denote the kernel of $R\otimes T(V_{\leqslant me})\rightarrow
\A_{(\alpha,\beta)}$. The ideal $I$ comes with the $R$-module map
$$R\otimes\left((\bigoplus_{i=0}^e V_{\leqslant m}^{\otimes i})\oplus V_{\leqslant m}^{\otimes 2}\right)
\rightarrow R\otimes T(V_{\leqslant me}), (x,y)\mapsto x-\alpha(x)+ \alpha(y)-\alpha^{opp}(y)-z\beta(y).$$
The image generates $I$ as an $R\otimes T(V_{\leqslant me})$-bimodule.
Consider the filtration on the $R\otimes T(V_{\leqslant me})$-bimodule $I$ induced
from the natural filtration on  $R\otimes\left((\bigoplus_{i=0}^e V_{\leqslant m}^{\otimes i})\oplus V_{\leqslant m}^{\otimes 2}\right)$.
The associated graded of $I$ contains the ideal of relations
of $R\otimes \C[X]$ (including the commutativity relations).
The condition (*) is equivalent to the following condition
\begin{itemize}
\item[(**)] For each $j\geqslant 0$, we have that
$R\otimes T(V_{\leqslant me})_{\leqslant j}/I_{\leqslant j}$ is a free
module of rank $\dim V_{\leqslant j}$ (while, a priori, the former is the quotient
of a free module of the given rank).
\end{itemize}

The quotient $R\otimes T(V_{\leqslant me})_{\leqslant j}/I_{\leqslant j}$
is the cokernel of a matrix whose entries are polynomials in the entries of $\alpha$ and $\beta$.
For an arbitrary pair $(\alpha,\beta)$, the matrix is such that  the cokernel is the quotient of
a free $R$-module of rank $\dim V_{\leqslant j}$.
So the claim that the cokernel is a free module of that rank is equivalent to the vanishing
of all minors of a specified size. This gives polynomial conditions on $(\alpha,\beta)$
that are independent of $R$.
\end{proof}

Let $Y^2\subset Y^1$ denote the (scheme-theoretic) vanishing locus of polynomial equations from Lemma
\ref{Lem:alpha_polynom}. For any algebra homomorphism $(\alpha,\beta): \C[Y^2]\rightarrow R$
we get the filtered associative algebra $\A_{(\alpha,\beta)}$ with $R\otimes \C[X]
\xrightarrow{\sim} \gr A_{(\alpha,\beta)}$. There is at most one bracket $\langle\cdot,\cdot\rangle$
on $\A_{(\alpha,\beta)}$ satisfying (iv) whose pullback to $V_{\leqslant m}^{\otimes 2}$
coincides with $\beta$. Such a bracket then automatically satisfies (ii) and (iii).
Note that (iv) translates into a collection on polynomial equations on $\alpha$
and $\beta$. Let $Y\subset Y^2$ denote the closed subscheme defined by these equations.

By the construction, $Y$ represents the functor
of taking deformation data:  to give a deformation datum over $R$ is
the same thing as to give an $R$-point of $Y$. Denote this set of points
by $Y(R)$.


Now we proceed to group actions on $Y$. Define the unipotent group $U$. We take the subgroup of
$\GL(V_{\leqslant me})$ consisting of all linear maps $\Phi:V_{\leqslant me}\rightarrow V_{\leqslant me}$
with $\deg (\Phi(f)-f)<\deg f$ for all $f\in V_{\leqslant me}$. We have an induced action on $T$
that preserves the defining ideal of $Y$ as well as $z$.

Also define an action of $\C^\times$ on $V_{\leqslant me}$ by $t.f:=t^{-\deg f}f$ for a homogeneous
element $f\in V_{\leqslant me}$. Then $\C^\times$ normalizes $U\subset \GL(V_{\leqslant me})$ and also
preserves the ideal of $Y$. Recall  that $z$ can be viewed as an element of $\C[Y^1]$. Hence $z$
defines an element of $\C[Y]$ (which is a quotient of $\C[Y^1]$), equivalently, a morphism
$Y\rightarrow \C$.

Note that, by the construction, we have the following property:
\begin{itemize}
\item[($\spadesuit$)]  The morphism
$Y\rightarrow \C$ given by $z$  is $U$-invariant and $\C^\times$-equivariant, $t\in \C^\times$
acts on $\C$ by multiplication by $t^{-d}$. Moreover, $\C[Y],\C[U]$ are positively graded
with respect to the $\C^\times$-action.
\end{itemize}

We write $U(R)$ for the group of $R$-points of $U$. Note that $U(R)$ acts on $Y(R)$.
We write $U_{\Spec(R)}$ for the constant group scheme over $\Spec(R)$ with fiber
$U$. So $U(R)$ is the group of sections $\operatorname{Spec}(R)\hookrightarrow U_{\operatorname{Spec}(R)}$.

\subsection{Generating maps and the algebras $\A_{Y},\A_{\mathbb{Y}}$}\label{SS_A_Y}
Now we need another concept: that of a generating map.
Let $B$ be a finitely generated commutative $\C$-algebra (but not, a priori, a $\C[Y]$-algebra)
and let $\A_B$ be a filtered $B$-algebra with
\begin{itemize}
\item $B$ in degree $0$,
\item an isomorphism $\gr \A_B\cong B\otimes \C[X]$ of graded Poisson algebras.
\item a $B$-linear bracket $\langle \cdot,\cdot\rangle$ satisfying (ii)-(iv) of Definition
\ref{defi_def_data} for
some $z\in B$,
\end{itemize}

\begin{defi}
By a {\it generating map} for $\A_B$ we mean a filtered $B$-module map $B\otimes V_{\leqslant me}\rightarrow \A_B$
that becomes  the inclusion $B\otimes V_{\leqslant me}\rightarrow B\otimes \C[X]$
after passing to the  associated graded algebra.
\end{defi}
If $B'$ is a $B$-algebra, then $\A_{B'}:=B'\otimes_B \A_B$ inherits a generating map
from $\A_B$.

Note that any generating map is injective and the image of $B\otimes V_{\leqslant m}$ generates $\A_B$ (hence the name). Denote the set of generating maps by $\mathsf{Gen}(\A_B)$. It comes with a
$U(B)$-action. There is at least one generating map, and if we fix it, we get a $U(B)$-equivariant bijection
$U(B)\xrightarrow{\sim} \mathsf{Gen}(\A_B)$.

Let us explain a connection between generating maps and deformation data.
Every generating map for $\A_B$ gives rise to a deformation datum over $B$: by restricting
the product and the bracket from $\A_B$ to $B\otimes V_{\leqslant me}$. So we get a  map
of sets
\begin{equation}\label{eq:torsor_map}
\mathsf{Gen}(\A_B)\rightarrow Y(B).\end{equation}
This map is $U(B)$-equivariant.

We also need a certain automorphism group. The space $B\otimes (\bigoplus_{i=0}^{d-1}\C[X]_i)$ is a nilpotent Lie subalgebra
of $\A_B$ with respect to $\langle\cdot,\cdot\rangle$. So we can
consider the Lie algebra $$\mathfrak{h}_B:=(B\otimes (\bigoplus_{i=0}^{d-1}\C[X]_i))/B\otimes \C[X]_0$$
($\mathfrak{h}$ for Hamiltonian).
Consider the corresponding unipotent group scheme $H_{\operatorname{Spec}(B)}$
over $\operatorname{Spec}(B)$. Note that we have a group scheme monomorphism
$H_{\Spec(B)}\hookrightarrow U_{\Spec(B)}$ via $\exp(a)\mapsto \exp(\langle a,\cdot\rangle)$ for $a\in \mathfrak{h}_{B}$.

The following lemma explains the meaning of $H_{\operatorname{Spec(B)}}$.

\begin{Lem}\label{Lem:U0_B}
The group scheme over $\Spec(B)$ of filtered $B$-algebra automorphisms of $\A_B$ that are the identity on
$\gr \A_B$ coincides with $H_{\Spec(B)}$.
\end{Lem}
\begin{proof}
Let $\psi$ denote an automorphism of $\A_B$ as in the statement of the lemma.
Then $\ln(\psi)$ is well-defined and is a $B$-linear derivation of $\A_B$
that is zero on $\gr\A_B$. We need to show that $\ln(\psi)=
\langle a,\cdot\rangle$ for $a\in \mathfrak{h}_B$. Note that $\ln(\psi)$
gives rise to a homogeneous negative degree Poisson $B$-linear derivation
of $\gr\A_B=B\otimes \C[X]$, the top degree term of $\ln(\psi)$.  Proposition
\ref{Prop:Poisson_der} implies that such a derivation is Hamiltonian. Since
the Poisson center of $B\otimes \C[X]$ is $B$, we see that the top degree
term of $\ln(\psi)$ takes the form $\{f,\cdot\}$ for  a unique homogeneous element
$f\in B\otimes \bigoplus_{i=1}^{d-1}\C[X]_i$. This implies the analogous statement
for $\ln(\psi)$ itself finishing the proof of the lemma.
\end{proof}

Now let $R$ be a finitely generated $\C[Y]$-algebra, let us write
$(\alpha,\beta)$ for the corresponding homomorphism $\C[Y]\rightarrow R$.
We get the algebra  $\A_{(\alpha,\beta)}$ defined in
(\ref{eq:A_alpha_defn}).
The algebra corresponding to the identity automorphism of  $\C[Y]$ will
be denoted by $\A_{Y}$. Note that the algebra $\A_Y$ is graded and comes
with an action of $U$
by $\C$-algebra automorphisms that is compatible with the action of
$U$ on $\C[Y]$.
Also $\A_{Y}$ admits a $\Z_{\geqslant 0}$-filtration
with $\C[Y]$ in degree $0$ such that $\gr\A_{Y}\cong \C[Y]\otimes \C[X]$ as a graded Poisson
algebra (the Poisson  bracket on $\gr\A_Y$ comes from $\langle\cdot,\cdot\rangle$).
Note that this filtration does not come from the grading on $\A_{Y}$.

Every algebra $\A_{(\alpha,\beta)}$ defined by
(\ref{eq:A_alpha_defn}) comes with a canonical generating map that gives the deformation
datum specified by $(\alpha,\beta)$.
This deformation map gives an identification $\mathsf{Gen}(\A_{\alpha,\beta})\cong U(R)$, hence
a map $U(R)\rightarrow Y(R)$. This map is functorial in $R$ and hence comes
from a morphism of schemes $U_Y=U\times Y\rightarrow Y\times Y$. This morphism
is nothing else but the graph of the action of $U$ on $Y$. So the preimage of the
diagonal is  the automorphism group scheme $H_Y$.
To simplify the notation below we are going to denote the preimage of the diagonal
by $\mathbb{Y}$. Note that, by construction, $\mathbb{Y}\subset U\times Y$
is $U$-stable for the action given by $u.(u_1,y)=(uu_1u^{-1},uy)$ and $\C^\times$-stable.

We write $\A_{\mathbb{Y}}$ for the algebra $\C[\mathbb{Y}]\otimes_{\C[Y]}\A_Y$.
This algebra comes with a generating map. Namely, consider the diagonal action of
$U$ on the $\C[U\times Y]$-algebra $\C[U]\otimes \A_Y$.
There is a unique generating map for this algebra
that is $U\times U$-equivariant, where the first action of $U$ on itself from the left and
the second action is diagonal,
and whose fiber over $1\in U$ is the natural generating map
for $\A_Y$. Then we get a generating map for $\A_{\mathbb{Y}}$ by base change from $\C[U\times Y]$
to its quotient $\C[\mathbb{Y}]$. Note that this generating map is $U$-equivariant. A way to think about
this generating map is that at the points of the diagonal $\{1\}\times Y\subset \mathbb{Y}$
we get the natural generating map for $\A_Y$ and then we extend this map by the action of $H_Y$.
The following is a universal property of $\A_{\mathbb{Y}}$ and its generating map.

\begin{Prop}\label{Prop:univer_AY}
Let $\A_B$ be as in the beginning of the section. Fix a generating map for $\A_B$.
Then there is a unique  algebra homomorphism $\C[\mathbb{Y}]\rightarrow B$ and a
unique filtered $B$-algebra isomorphism
$$\A_B\xrightarrow{\sim} B\otimes_{\C[\mathbb{Y}]}\A_{\mathbb{Y}}$$  that intertwines the
brackets, the generating maps, and the isomorphisms $\gr \A_B, \gr (B\otimes_{\C[\mathbb{Y}]}\A_{\mathbb{Y}})
\xrightarrow{\sim}B\otimes \C[X]$.
\end{Prop}
\begin{proof}
As was mentioned before, the choice of a generating map for $\A_B$ gives
rise to a point in $Y(B)$, or, equivalently, an algebra homomorphism
$\C[Y]\rightarrow B$. This homomorphism equips $\A_B$ with another generating
map, which is pulled back from $\A_Y$.
This new map does  not need to coincide with the initial one, but gives the same deformation
datum. Since $\mathsf{Gen}(\A_B)$ is a torsor over $U(B)$,  a choice of two generating
maps for $\A_B$ gives a morphism $\operatorname{Spec}(B)\rightarrow U_Y$.
Since the generating maps give the same deformation data, the morphism factors
through $\operatorname{Spec}(B)\rightarrow \mathbb{Y}$. By the construction,
we have a filtered algebra isomorphism $$\A_B\xrightarrow{\sim} B\otimes_{\C[\mathbb{Y}]}\A_{\mathbb{Y}}$$
with the required properties. It is unique: the only automorphism of $\A_B$
that fixes a generating map is the identity.
\end{proof}

\subsection{Structure of $\mathbb{Y}$}\label{SS_Y_str}
Our goal here is to describe the structure of $\mathbb{Y}$ (in fact, of its very close
relative) and use this to give a proof of Proposition \ref{Prop:quant_univ_property}.

Let $\A^W_{\param,\hbar}$ denote the Rees algebra of the filtered algebra $\A_{\param}^W$.
This is a graded algebra over $\C[\param]^W[\hbar]$, where $\C[\param]^W\subset \A_{\param,\hbar}^W$ is graded with $\param^*$
of degree $d$ and the degree of $\hbar$ is $1$.

\begin{Cor}\label{Cor:AY_Section}
There is a $\C^\times$-equivariant scheme morphism $\psi:\param/W\times \C\rightarrow \mathbb{Y}$
such that $\psi^*(z)=\hbar^d$ and there is a $\C[\param/W,\hbar]$-algebra isomorphism $\A^W_{\param,\hbar}\cong
\C[\param/W,\hbar]\otimes_{\C[\mathbb{Y}]}\A_{\mathbb{Y}}$.
\end{Cor}
\begin{proof}
This is because $\A^W_{\param,\hbar}$ comes with a generating map (as any other deformation).
This generating map can be chosen to be  $\C^\times$-equivariant. The bracket $\langle\cdot,\cdot\rangle$ equals $[\cdot,\cdot]/\hbar^d$. These observations together with Proposition
\ref{Prop:univer_AY} imply the claim of the corollary.
\end{proof}

%

Now we  describe $\mathbb{Y}_0$, the scheme theoretic fiber of $\mathbb{Y}$ at $z=0$.
Corollary \ref{Cor:AY_Section} yields an induced scheme morphism $\psi_0:\param/W\rightarrow \mathbb{Y}_0$.

\begin{Prop}\label{Cor:AY0_iso}
The $U$-equivariant morphism $U\times \param/W\rightarrow \mathbb{Y}_0$ extending $\psi_0$ is
an isomorphism.
\end{Prop}
\begin{proof}
The proof is in several steps.

{\it Step 1}. Consider the algebra $\C[U]\otimes \C[X_{\param}]^W$ over $\C[U]\otimes \C[\param]^W$.
It comes with a  generating map produced as follows. We pick a generating map
in the fiber over $1\in U$ and then extend it to $\C[U]\otimes \C[\param]^W\otimes V_{\leqslant me}
\rightarrow \C[U]\otimes \C[X_{\param}]^W$ so that corresponding scheme morphism
$U\times X_{\param}/W\rightarrow U\times \param/W\times V_{\leqslant me}^*$ is $U$-equivariant.
Here in the target $U$ acts by $u.(u',p,\alpha)=(uu',p,u.\alpha)$. Compare with the
description of the generating map for $\A_{\mathbb{Y}}$ in the previous section.

{\it Step 2}. Now let $B$ be a positively graded algebra and let $\A_B$ be a graded Poisson algebra
deforming $\C[X]$. Then we have a $\C^\times$-equivariant generating map $B\otimes V_{\leqslant me}\rightarrow \A_B$.
We claim that there is a unique  morphism of schemes $\operatorname{Spec}(B)\rightarrow
U\times \param/W$ and a Poisson algebra isomorphism $\A_B\xrightarrow{\sim} B\otimes_{\C[U]\otimes
\C[\param/W]}(\C[U]\otimes \C[X_{\param}]^W)$ intertwining the generating maps.

By the universal property of $\C[X_{\param}]^W$, see Proposition \ref{Prop:X_univ_graded},
there is a unique graded algebra homomorphism
$\C[\param]^W\rightarrow B$ and a graded Poisson $\C[B]$-algebra isomorphism
$\A_B\cong B\otimes_{\C[\param]^W}\C[X_{\param}]^W$ that is the identity modulo the augmentation
ideal in $B$.   Since the generating maps form a torsor over $U(B)$,
we further see that there is a unique homomorphism $\C[U]\otimes \C[\param]^W\rightarrow B$
and a unique isomorphism $\A_B\xrightarrow{\sim} B\otimes_{\C[U]\otimes
\C[\param/W]}(\C[U]\otimes \C[X_{\param}]^W)$ intertwining the generating maps.

{\it Step 3}. By Proposition \ref{Prop:univer_AY}, there is a unique scheme
morphism $\iota_1: U\times \param/W\rightarrow \mathbb{Y}_0$ and a unique
Poisson $\C[U\times \param/W]$-algebra isomorphism $\C[U\times \param/W]\otimes_{\C[\mathbb{Y}_0]}\A_{\mathbb{Y}_0}\rightarrow
\C[U]\otimes \C[X_{\param}]^W$ intertwining the generating maps. On the other hand,
$\A_{\mathbb{Y}_0}$ is positively  graded. So, by Step 2, there is a unique scheme morphism
$\iota_2:\param/W\times U\rightarrow \mathbb{Y}_0$ and a unique Poisson algebra isomorphism
$\A_{\mathbb{Y}_0}\xrightarrow{\sim}\C[\mathbb{Y}_0]\otimes_{\C[U\times \param/W]}(\C[U]\otimes \C[X_{\param}]^W)$
intertwining the generating maps. So the morphisms $\iota_1,\iota_2$ are mutually inverse.
\end{proof}

Now we will describe  the structure of a slight modification of $\mathbb{Y}$.
Set $$\mathbb{Y}':=\operatorname{Spec}(\C[\mathbb{Y}][\hbar]/(z-\hbar^d)).$$

\begin{Lem}\label{Lem:Struct_Y}
We have a $\C^\times\times U$-equivariant isomorphism $U\times \param/W\times \operatorname{Spec}(\C[\hbar])\xrightarrow{\sim}
\mathbb{Y}'$ of schemes over $\operatorname{Spec}(\C[\hbar])$. It gives rise to
a  $\C[\mathbb{Y}']$-linear and $U\times \C^\times$-linear algebra isomorphism
$\C[\hbar]\otimes_{\C[z]}\A_{\mathbb{Y}}\xrightarrow{\sim}\C[U]\otimes \A^W_{\param,\hbar}$.
\end{Lem}
\begin{proof}
Consider the algebra $\C[U]\otimes \A_{\param,\hbar}^W$. It comes with a generating map
defined as in Step 1 of the proof of Proposition \ref{Cor:AY0_iso}. This generating map
is $\C^\times$-equivariant by the construction. So we get a unique
graded algebra homomorphism $\C[\mathbb{Y}]\rightarrow \C[U]\otimes \C[\param]^W[\hbar]$
and a unique graded algebra isomorphism $$\C[U]\otimes\A_{\param,\hbar}^W
\xrightarrow{\sim} \left(\C[U]\otimes \C[\param]^W[\hbar]\right)\otimes_{\C[\mathbb{Y}]}\A_Y$$
that intertwines the brackets $\langle\cdot,\cdot\rangle$ and the generating  maps.
The homomorphism $\C[\mathbb{Y}]\rightarrow \C[U]\otimes \C[\param]^W[\hbar]$ maps
$z$ to $\hbar^d$ and so extends to a graded $\C[\hbar]$-algebra homomorphism
$\C[\mathbb{Y}']\rightarrow \C[U]\otimes \C[\param]^W[\hbar]$.  At $\hbar=0$, it specializes to
the isomorphism $\C[\mathbb{Y}_0]\xrightarrow{\sim} \C[U]\otimes \C[\param]^W$.   Note that both
$\C[\mathbb{Y}']$ and $\C[U]\otimes \C[\param]^W[\hbar]$ are positively graded.
It follows that $\C[\mathbb{Y}']\xrightarrow{\sim} \C[U]\otimes \C[\param]^W[\hbar]$.
\end{proof}

\begin{proof}[Proof of Proposition \ref{Prop:quant_univ_property}]
Let $B',\A'$ be as in the statement of the proposition. Consider
the Rees algebras $B'_\hbar:=R_\hbar(B'),\A'_\hbar:=R_\hbar(\A')$.
Take $z:=\hbar^d$. Thanks to Proposition \ref{Prop:univer_AY},
for each choice of a generating map for
$\A'_\hbar$, we get a unique $\C[z]$-linear homomorphism
$\C[\mathbb{Y}]\rightarrow B'_\hbar$ and a unique $B'_\hbar$-linear isomorphism
$B'_\hbar\otimes_{\C[\mathbb{Y}]}\A_{\mathbb{Y}}\rightarrow \A'_\hbar$ intertwining
the generating maps. Thanks to the uniqueness, these homomorphisms are
graded. Different choices of generating maps for $\A'_\hbar$
result in an action of $U(B'_\hbar)$ and hence do not change the homomorphism
$\C[\param]^W[\hbar]\rightarrow B'_\hbar$.
Hence there  is a unique $\C[\hbar]$-linear graded algebra homomorphism
$\C[\param]^W[\hbar]\rightarrow B'_\hbar$ and a
graded $B'_\hbar$-algebra isomorphism $B'_\hbar\otimes_{\C[\param]^W[\hbar]}\A_{\param,\hbar}^W
\xrightarrow{\sim} \A'_\hbar$. We specialize at $\hbar=1$
and arrive at the claim of the proposition.
\end{proof}

\subsection{Spherical symplectic reflection algebras}\label{SS_SRA}
Let us consider the case of a symplectic quotient singularity $X=V/\Gamma$.
Let $\pi:V\rightarrow V/\Gamma$ denote the quotient morphism.
Recall that by a symplectic reflection in $\Gamma$ we mean an element
$s\in \Gamma$ with $\operatorname{rk}(s-\operatorname{id})=2$.
To a symplectic reflection $s$ we assign the subgroup $\Gamma^s\subset \Gamma$,
the pointwise stabilizer of $V^s$, and the quotient $\Xi^s:=N_\Gamma(\Gamma^s)/\Gamma^s$.
The codimension $2$ symplectic leaves are in one-to-one correspondence with the conjugacy
classes of the subgroups $\Gamma^s$. The leaf corresponding to $\Gamma^s$
is of the form $\{v\in V| \Gamma_v=\Gamma^s\}/\Xi^s$. So we see that
the fundamental group is $\Xi^s$. From here we deduce that the irreducible
components of $\pi^{-1}(\mathcal{L})$ are labelled by the non-trivial $N_{\Gamma}(\Gamma^s)$-conjugacy
classes in $\Gamma^s$. 

The following lemma is proved in \cite[Lemma 2.4]{Bellamy}.

\begin{Lem}\label{Lem:param_space}
We have $H^2(X^{reg},\C)=0$.
\end{Lem}

From  Lemma \ref{Lem:param_space} and the preceding paragraph  we conclude that the dimension of
$\param$ coincides with the number of conjugacy classes of symplectic
reflections in $\Gamma$, this result was  obtained in \cite[Theorem 1.3]{Bellamy}.

There is a way to deform $\C[V]^\Gamma$ discovered by Etingof and
Ginzburg, \cite{EG}. Namely, we first deform the smash-product algebra
$\C[V]\#\Gamma$. Let $t\in \C$ and $c$ be a $\Gamma$-invariant
function $S\rightarrow \C$, where $S$ is the set of all symplectic
reflections. Let $\omega$ denote the symplectic form on $V$.
For $s\in S$, we write $\omega_s$ for the rank $2$ form on $V$
whose kernel coincides with $V^s$ and whose restriction to
$\operatorname{im}(s-1)$ coincides with the restriction of $\omega$. Then we can form
the algebra $H_{t,c}$ (known as a symplectic reflection algebra) by
$$H_{t,c}=T(V)\#\Gamma/\left(u\otimes v-v\otimes u-t\omega(u,v)-
\sum_{s\in S}c(s)\omega_s(u,v)s| u,v\in V\right).$$
This is a filtered deformation of $\C[V]\#\Gamma$, see \cite[Theorem 1.3]{EG}.
Now take the averaging
idempotent $e\in \C\Gamma$. We can form the so called {\it spherical subalgebra}
$eH_{t,c}e\subset H_{t,c}$ that is a filtered associative algebra with unit
$e$. It is a filtered deformation of $\C[V]^\Gamma$ that induces
the Poisson bracket $t\omega^{-1}$ on $\C[V]^\Gamma$, see the proof of
\cite[Theorem 1.6]{EG}. So we get a
filtered quantization when $t=1$.

Let $\mathfrak{c}$ denote the space of $\Gamma$-invariant functions
$S\rightarrow \C$, this is the space of parameters $c$.
The space $\mathfrak{c}$  is in an affine bijection with $\param$.
Namely, we can split $S$ into the union $S_1\sqcup S_2\sqcup\ldots\
\sqcup S_k$, where $S_i$ stands for the symplectic reflections in
$\Gamma_i$. Consequently, $\mathfrak{c}$ splits into the direct sum
$\bigoplus_{i=1}^k \mathfrak{c}_i$. Let
$\widehat{\mathfrak{c}}_i$ denote the space of $\Gamma_i$-invariant functions
$\Gamma_i\setminus\{1\}\rightarrow \C$. It comes with a natural action of $\Xi_i$. The space $\mathfrak{c}_i$ embeds
as the $\Xi_i$-invariants into $\widehat{\mathfrak{c}}_i$.

An affine isomorphism between $\widehat{\mathfrak{c}}_i$ and $\widehat{\param}_i$ in the form we need
was explained in \cite[Section 6.2]{quant_iso}. Namely, define the element $C_i\in \C\Gamma_i$
by $|\Gamma_i|^{-1}\left(1+\sum_{\gamma\in \Gamma_i}c(\gamma)\gamma_i\right)$. Let $N_1,\ldots,N_{r_i}$
denote the nontrivial irreducible representations of $\Gamma_i$. Then we
send an element $\{c(\gamma)\}|_{\gamma\in \Gamma_i}\in \widehat{\mathfrak{c}}_i$
to $\sum_{j=1}^{r_i}(\operatorname{tr}_{N_j}C_i)\varpi_j$, where $\varpi_1,\ldots,
\varpi_{r_i}$ are the fundamental weights in $\widehat{\param}_i$. Note that this isomorphism
$\widehat{\mathfrak{c}}_i\xrightarrow{\sim} \widehat{\mathfrak{P}}_i$ is $\Xi_i$-equivariant.

Let us denote the resulting affine isomorphism $\mathfrak{c}\xrightarrow{\sim}\param$ by $\iota$.

The following proposition generalizes  \cite[Theorem 6.2.1]{quant_iso}.

\begin{Prop}\label{Prop:SRA_univer}
We have an isomorphism $eH_{1,c}e\xrightarrow{\sim} \A_{\iota(c)}$
of filtered quantizations of $\C[V]^\Gamma$.
\end{Prop}

In particular, every filtered quantization of $\C[V]^\Gamma$ is a spherical symplectic
reflection algebra (see \cite{Boddington,Hodges,Levy} for various special cases of this
result). An analog of this result for $t=0$ was obtained by Bellamy in \cite[Corollary 1.6]{Bellamy}.

\begin{proof}[Proof of Proposition \ref{Prop:SRA_univer}]
Thanks to Theorem \ref{Thm:quant_classif}, we already know that $eH_{1,c}e\cong \A_\lambda$ for some $\lambda\in \param$
and we need to show that $\lambda=\iota(c)$.
Consider the Rees algebra $R_\hbar(eH_{1,c}e)$ and its completion $R_\hbar(eH_{1,c}e)^{\wedge_x}$
at a point $x\in \mathcal{L}_i$.
According to \cite[Theorem 1.2.1]{sraco}, we get an isomorphism of formal quantizations
$$R_\hbar(eH_{1,c}e)^{\wedge_x}\cong \mathbb{A}_{\hbar}(V^{\Gamma_i})^{\wedge_0}\widehat{\otimes}_{\C[[\hbar]]}
R_\hbar(e^i H^i_{1,c_i}e^i)^{\wedge_0}.$$
Here the notation is as follows. We write $c_i$ for the projection of
$c$ to $\mathfrak{c}_i$. The notation $\mathbb{A}_\hbar(V^{\Gamma_i})$ is for the Rees algebra of the Weyl
algebra of the symplectic vector space $V^{\Gamma_i}$. The notation $H^i_{1, c_i}$
is for the SRA associated to $(\Gamma_i, V/V^{\Gamma_i})$ and $e^i$ is the averaging
idempotent in $\Gamma_i$.

By Lemma \ref{Lem:homol_vanish}, $\param_0=\{0\}$.
So, by Proposition \ref{Rem:quant_param}, we reduce to proving that, in the notation of that
proposition, we have an isomorphism of filtered quantizations
$e^i H^i_{1,c_i}e^i\cong \A^i_{\lambda_i}$. This follows from
\cite[Theorem 6.2.2]{quant_iso}.
\end{proof}

This proposition establishes an isomorphism $e H_{1,c} e \cong e H_{1,c'}e$
of quantizations when $\iota(c)$ and $\iota(c')$ are $W$-conjugate.
In the next section we will determine when $e H_{1,c}e, e H_{1,c'}e$
are isomorphic as filtered algebras.

\begin{Rem}
One application of Proposition \ref{Prop:SRA_univer} is to construct shift $H_{1,c+\psi}$-$H_{1,c}$-bimodules $\mathsf{S}_{c,\psi}$, where $\psi$ is
an integral element of $\mathfrak{c}$, compare with \cite{BC,rouq_der}. For a
fixed $\psi$ and a Zariski generic $c$, the bimodule $\mathsf{S}_{c,\psi}$
gives a Morita equivalence between $H_{1,c}$ and $H_{1,c+\psi}$, which
can be established similarly to \cite[Corollary 3.5]{rouq_der} using the fact that the
algebra $H_{1,c}$ is simple for a Weil generic $c$, \cite[Theorem 4.2.1]{sraco}.
We do not provide details in the present paper, see \cite{SRA_der} instead.
\end{Rem}

\subsection{Automorphisms and isomorphisms}\label{SS_iso}
Here we are going to study the relationship between three different objects:
\begin{enumerate}
\item The reductive part of the automorphism group  of the graded Poisson algebra $\C[X]$ (note that this group is algebraic). The reductive part will be denoted by $\mathcal{G}$.
\item Filtered Poisson algebra isomorphisms $\C[X_\lambda]\xrightarrow{\sim} \C[X_{\lambda'}]$.
\item Filtered algebra isomorphisms $\A_\lambda\xrightarrow{\sim} \A_{\lambda'}$.
\end{enumerate}
In what follows we always assume that $d$ is even: we can always rescale the action of
$\C^\times$ on $X$ to achieve this.

Note that $\mathcal{G}$ acts on the set of isomorphism classes of filtered Poisson deformations
(resp, quantizations) by replacing the isomorphism $\iota: \gr\A\xrightarrow{\sim}\C[X]$
with $g\circ \iota$, for $g\in \mathcal{G}$. So we have two, a priori different, actions of $\mathcal{G}$ on $\param/W$
viewed as the space of parameters for filtered Poisson deformations and for filtered quantizations.
These actions will be called Poisson and quantum below. Note that the universal properties
for the algebras $\C[X_{\param}]^W$ (Corollary \ref{Cor:X_univ_deform}) and
$\A_{\param}^W$ (Proposition \ref{Prop:quant_univ_property}) yield $\mathcal{G}$-actions
on $\C[X_{\param}]^W$ (by graded Poisson algebra automorphisms) and
on $\A_{\param}^W$ (by filtered algebra automorphisms). Let us explain why this is the
case for quantizations. Any element $g\in \mathcal{G}$
gives rise to an automorphism $\tilde{g}:\A_{\param}^W\rightarrow \A_{\param}^W$ by the universal
property. The automorphism $\tilde{g}$ is defined uniquely up to composing with a $\C[\param]^W$-linear automorphism of $\A_\param^W$ that gives the identity on $\C[X]$. The group of such automorphisms is easily seen to be unipotent. So we can choose $\tilde{g}$ so that the map $g\mapsto\tilde{g}$ is a
group homomorphism.

These actions preserve the
subalgebras $\C[\param]^W$ and induce the Poisson and quantum actions
on $\param/W$. Note that the quantum action on $\C[\param]^W$ is filtration preserving and the  Poisson action on $\C[\param]^W$ is obtained
from the quantum action by passing to the associated graded action.

\begin{Lem}\label{Lem:isom_deform}
We have a filtered Poisson algebra isomorphism $\C[X_\lambda]\xrightarrow{\sim} \C[X_{\lambda'}]$
(resp., filtered associative algebra isomorphism $\A_\lambda\xrightarrow{\sim} \A_{\lambda'}$)
if and only if $W\lambda,W\lambda'$ lie in the same $\mathcal{G}$-orbit for the Poisson (resp.,
quantum) action on $\param/W$.
\end{Lem}
\begin{proof}
The if part follows from the preceding discussion. So we need to show, say, that if $\A_\lambda,\A_{\lambda'}$ are isomorphic as filtered algebras, then $W\lambda,W\lambda'$
are in the same $\mathcal{G}$-orbit.

An isomorphism $\A\rightarrow \A'$ of filtered (associative/Poisson) algebras
induces an isomorphism $\gr\A\rightarrow \gr\A'$ of graded Poisson algebras.
So $W\lambda,W\lambda'$ lie in the same orbit for the group of graded Poisson automorphisms of
$\C[X]$. We need to show that the unipotent radical of this group acts trivially on $\param/W$.
Let $g$ be an element in the unipotent radical. Since it is unipotent, $\ln(g)$ makes sense.
It is a graded Poisson algebra derivation of $\C[X]$. By Proposition \ref{Prop:Poisson_der},
it is inner. So it is the Poisson bracket with a degree $d$ element in $\C[X]$. Lift this
element to an element $a\in (\A_\param^W)_{\leqslant d}$. Since $\ln(g)$ is nilpotent, so
is $[a,\cdot]$. So we can integrate $[a,\cdot]$ to a $\C[\param]^W$-linear filtered algebra
isomorphism of $\A_{\param}^W$ lifting $g$. From here we deduce that $g$ acts trivially
on $\C[\param]^W$.
\end{proof}

We are now going to  show
that the Poisson and quantum actions are the same.
%
In the proof we will need a lemma describing the group of graded Poisson automorphisms
of $\C[V]^\Gamma$.
Namely, let $V$ be a symplectic vector space and $\Gamma$ be a finite
group of its linear symplectomorphisms. Set $\Theta:=N_{\operatorname{Sp}(V)}(\Gamma)/\Gamma$.
This groups acts on $\C[V]^\Gamma$ faithfully.

\begin{Lem}\label{Lem:symplectomorphisms}
The group $\mathcal{G}$ of graded Poisson automorphisms of $\C[V]^\Gamma$ coincides with $\Theta$.
\end{Lem}
\begin{proof}
 We have an inclusion  $\Theta\hookrightarrow \mathcal{G}$.
Note that $V^0$, the free locus for the $\Gamma$-action, is  the simply-connected cover
of $V^0/\Gamma=(V/\Gamma)^{reg}$. The  Galois group of this cover is  $\Gamma$.
Every automorphism of $\C(V/\Gamma)=\C(V)^\Gamma$ lifts to an automorphism of $\C(V)$.
The lift is unique up to composing with an element of $\Gamma$. An automorphism of $V^0/\Gamma$
therefore lifts to an automorphism of $V^0.$
So the $\mathcal{G}$-action on $V^0/\Gamma$
lifts to an action of an extension $\tilde{\mathcal{G}}$ of $\mathcal{G}$ by $\Gamma$ on $V^0$ by automorphisms.

The action of
$\tilde{\mathcal{G}}$ on $V^0$ extends to an action on $V$ because $\operatorname{codim}_V (V\setminus V^0)\geqslant 2$. It commutes with the dilating $\C^\times$
and preserves the symplectic form. So it is via  a group homomorphism
$\tilde{\mathcal{G}}\rightarrow \operatorname{Sp}(V)$. Also it descends to $V/\Gamma$ and so normalizes $\Gamma$.
We deduce that $\tilde{\mathcal{G}}\subset N_{\operatorname{Sp}(V)}(\Gamma)$ and hence $\mathcal{G}\subset\Theta$.
\end{proof}


\begin{Prop}\label{Lem:autom_action}
The Poisson and quantum actions coincide. Moreover, both are trivial on $\mathcal{G}^\circ$.
\end{Prop}
\begin{proof}
%
Recall that the Poisson action on $\C[\param]^W$ is the associated graded action
for the quantum action. So we need to prove that the quantum action on $\C[\param]^W$
is by graded algebra automorphisms and is trivial on $\mathcal{G}^\circ$.

Take $g\in \mathcal{G}$. Recall, Lemma \ref{Lem:Cartan_compute},
that $\param/W=H^2(X^{reg},\C)\times \prod_{i=1}^k \param_i/W_i$.  We claim that
\begin{itemize}
\item[(*)]
the projection
$\param/W\twoheadrightarrow H^2(X^{reg},\C)$ intertwines the action of $g$
with the usual linear action on $H^2(X^{reg},\C)$. \end{itemize}
Indeed, by Proposition \ref{Rem:quant_param},
the image of $W\lambda$ under the projection to $H^2(X^{reg},\C)$ is the period of the restriction of
$\A_\lambda$ to $X^{reg}$. The period map is functorial by the construction in
\cite[Section 4]{BK_quant}. This implies (*). Note that (*) in particular implies that the action of
$\mathcal{G}^\circ$ on the factor $H^2(X^{reg},\C)$ of $\param/W$ is trivial.

 Recall that $\Sigma_i$ denotes a slice to
the symplectic leaf $\mathcal{L}_i, \Sigma_i=\C^2/\Gamma_i$, and
$\Xi_i$ is the group of  diagram automorphisms of the corresponding
Dynkin diagram coming from the action of $\pi_1(\mathcal{L}_i)$.

The element $g$  permutes the symplectic
leaves (we denote the corresponding permutation of $\{1,\ldots,k\}$ again by $g$).
We have $\Gamma_i\cong \Gamma_{g(i)}$ and $\Xi_i\cong \Xi_{g(i)}$. Fix
some identifications. Then $g$ gives rise to automorphisms of $\Gamma_i$
and $\Xi_i$.

Fix a parameter $\lambda\in \param$ and a point $x\in \mathcal{L}_i$. Let $\lambda'$
be a representative of $g(W\lambda)$. So $g$ gives rise to a filtered isomorphism
$\A_\lambda\xrightarrow{\sim}\A_{\lambda'}$.
Recall, (\ref{eq:quantum_decomposition}), the decompositions
$$R_\hbar(\A_\lambda)^{\wedge_x}\cong
\mathbb{A}_{\hbar}(T_x\mathcal{L}_i)^{\wedge_0}\widehat{\otimes}_{\C[[\hbar]]}
R_\hbar(\A_{\lambda_i}^i)^{\wedge_0}, R_\hbar(\A_{\lambda'})^{\wedge_{gx}}\cong
\mathbb{A}_{\hbar}(T_{gx}\mathcal{L}_{gi})^{\wedge_0}\widehat{\otimes}_{\C[[\hbar]]}
R_\hbar(\A_{\lambda'_{gi}}^{gi})^{\wedge_0}.
$$
So $g$ gives an isomorphism of the right hand sides in these decompositions:
\begin{equation}\label{eq:g_isomorphism}
\alpha: \mathbb{A}_{\hbar}(T_x\mathcal{L}_i)^{\wedge_0}\widehat{\otimes}_{\C[[\hbar]]}
R_\hbar(\A_{\lambda_i}^i)^{\wedge_0}\xrightarrow{\sim}
\mathbb{A}_{\hbar}(T_{gx}\mathcal{L}_{gi})^{\wedge_0}\widehat{\otimes}_{\C[[\hbar]]}
R_\hbar(\A_{\lambda'_{gi}}^{gi})^{\wedge_0}.
\end{equation}
We claim that we can compose $\alpha$ with a $\C[[\hbar]]$-linear automorphism of the target so
that the composition is $\C^\times$-equivariant. Let $\gamma,\gamma'$ denote the actions
of $\C^\times$ on the source and the target of (\ref{eq:g_isomorphism}) and let $\gamma_0,\gamma_0'$
denote the actions modulo $\hbar$, note that the corresponding isomorphism $\alpha_0$ can be viewed as
an automorphism.

We argue similarly to Step 1 of the proof of Proposition \ref{Rem:quant_param}.
We can modify $\alpha$ so that $\alpha\gamma\alpha^{-1}$ and $\gamma'$ commute. In particular,
$\alpha_0\gamma_0\alpha_0^{-1}$ preserves $$(T_{gx}\mathcal{L}_{gi})^*\subset
\mathbb{A}_{\hbar}(T_{gx}\mathcal{L}_{gi})^{\wedge_0}\widehat{\otimes}_{\C[[\hbar]]}
R_\hbar(\A_{\lambda'_{gi}}^{gi})^{\wedge_0}/(\hbar)$$ by the degree reasons: this is the subspace of
elements of degree $d/2$. So it preserves its Poisson centralizer,  $R_\hbar(\A_{\lambda'_{gi}}^{gi})^{\wedge_0}/(\hbar)$. Consider the maximal ideal $\mathfrak{m}$
in that algebra so that $\mathfrak{m}/\mathfrak{m}^2$ is 3-dimensional. Since $\alpha_0$ is
an automorphism, the eigenvalues of $\alpha_0\gamma_0\alpha_0^{-1}$ and $\gamma_0'$ on $\mathfrak{m}/\mathfrak{m}^2$ coincide. Combining this with the observation that these actions commute and rescale the bracket by $t\mapsto t^{-d}$, we can use an easy case by case argument to see that they coincide on $\mathfrak{m}/\mathfrak{m}^2$. So the two actions on the target of
(\ref{eq:g_isomorphism}) differ by a pro-unipotent family of automorphisms and hence are conjugate
by a pro-unipotent automorphism. Modifying $\alpha$ accordingly, we get a $\C^\times$-equivariant
isomorphism in (\ref{eq:g_isomorphism}).


The conclusion is that we have a filtered algebra isomorphism
$\mathbb{A}(T_x\mathcal{L}_i)\otimes
\A_{\lambda_i}^i\xrightarrow{\sim}
\mathbb{A}(T_{gx}\mathcal{L}_{gi})\otimes
\A_{\lambda'_{gi}}^{gi}$.

Note that the algebras above are filtered quantizations of the symplectic quotient singularity
$T_x\mathcal{L}_i\times \C^2/\Gamma_i$. The space of quantization parameters is $\widehat{\param}_i/\widehat{W}_i$. The map $\widehat{\param}_i/\widehat{W}_i\rightarrow
\widehat{\param}_i/\widehat{W}_i$ coming from $g$ is in the image of the action of
the group $\Theta$ from Lemma \ref{Lem:symplectomorphisms}. This group coincides with
$N_{\SL_2}(\Gamma_i)/\Gamma_i$. The discussion of the isomorphism $\widehat{\mathfrak{c}}_i
\xrightarrow{\sim} \widehat{\param}_i$ in Section \ref{SS_SRA} shows that the action of
$N_{\SL_2}(\Gamma_i)/\Gamma_i$ on $\widehat{\param}_i$ is by linear automorphisms.
So the automorphism of  $\widehat{\param}_i/\widehat{W}_i$ induced by $g$ is $\C^\times$-equivariant.

Now we show that the automorphism of $\param_i/W_i$ induced by $g$ is
$\C^\times$-equivariant.  Step 2 of the proof of Proposition \ref{Rem:quant_param}
shows that the map $\param_i/W_i\rightarrow \widehat{\param}_i/\widehat{W}_i$
is injective. It also intertwines the actions of $g$. It follows that the action of
$g$ on $\param_i/W_i$ is $\C^\times$-equivariant. This finishes the proof of the claim that
$g$ acts on $\param/W$ by a $\C^\times$-equivariant automorphism.

The claim that the action of $\mathcal{G}^\circ$ on $\param/W$ is trivial follows from the observation that
the action of $(N_{\SL_2}(\Gamma_i)/\Gamma_i)^\circ$ on $\widehat{\param}_i$ is trivial.
\end{proof}

\begin{Cor}\label{Cor:filt_isom}
The filtered Poisson algebras $\C[X_\lambda],\C[X_{\lambda'}]$
are isomorphic if and only if the filtered associative algebras $\A_\lambda,\A_{\lambda'}$
are. So we have a bijection between
\begin{itemize}
\item Filtered Poisson deformations of $\C[X]$ viewed up to a filtered Poisson algebra isomorphism,
\item and filtered quantizations of $\C[X]$ viewed up to a filtered algebra isomorphism.
\end{itemize}
\end{Cor}
\begin{proof}
The first claim follows from Proposition \ref{Lem:autom_action}. Recall that both the filtered Poisson
deformations and the filtered quantizations are classified by $\param/W$ (Proposition
\ref{Prop:deform_X} and Corollary \ref{Cor:X_univ_deform} for Poisson deformations
and Theorem \ref{Thm:quant_classif} for quantizations). Up to filtered algebra isomorphisms
both are classified by the orbits of the group $\mathcal{G}$ on $\param/W$.
\end{proof}

Let us get back to spherical symplectic reflection algebras.

\begin{Prop}\label{Prop:sSRA_filt_iso}
We have a filtered algebra isomorphism $e H_{1,c}e\cong e H_{1,c'}e$ if and only if
$Wc$ and $Wc'$ are $\Theta$-conjugate.
\end{Prop}
\begin{proof}
This follows from Lemma \ref{Lem:isom_deform} and Lemma \ref{Lem:symplectomorphisms}.
\end{proof}

\begin{Rem}\label{Rem:equiv_action}
We will need a relative version of results of this section. Let $G\subset \mathcal{G}$ be a connected
reductive subgroup. By Proposition \ref{Lem:autom_action}, $G$ acts trivially on $\param/W$, hence
it acts on each $\A_\lambda,\A^0_\lambda$ by filtered algebra automorphisms. As in Lemma
\ref{Lem:isom_deform}, we have a
$G$-equivariant  filtered algebra isomorphism $\A_\lambda\xrightarrow{\sim}\A_{\lambda'}$
if and only if $W\lambda,W\lambda'$ are in the same $Z_{\mathcal{G}}(G)$-orbit. The same holds for
filtered Poisson deformations. Proposition \ref{Lem:autom_action} implies that there is
a $G$-equivariant filtered algebra isomorphism $\A_\lambda\xrightarrow{\sim}\A_{\lambda'}$
if and only if there is a $G$-equivariant filtered Poisson algebra isomorphism $\A^0_\lambda\xrightarrow{\sim} \A^0_{\lambda'}$. Also note that the stabilizer of $W\lambda$
in $Z_{\mathcal{G}}(G)$ acts on $\A_\lambda$ and $\A^0_\lambda$ by  $G$-equivariant filtered algebra
automorphisms.
\end{Rem}

\section{Birational induction and sheets}\label{S_bir_sheet}
From now on, $G$ is a connected reductive algebraic  group  over $\C$ with Lie algebra $\g$.
\subsection{Lusztig-Spaltenstein induction}\label{SS_LS_induction}
We use the notation of Section \ref{SS_bir_ind}. Recall that
$\lf\subset\g$ is a Levi subalgebra, $P\subset G$ is  a parabolic
subgroup such that $\mathfrak{p}:=\operatorname{Lie}(P)$ has Levi $\lf$, $\Orb'\subset \lf^*$ is a nilpotent
orbit, and $\xi\in (\lf/[\lf,\lf])^*$.
Form the variety $G\times^P(\xi+\overline{\Orb}'+\mathfrak{p}^\perp)$, the homogeneous
 bundle over $G/P$ with fiber $\xi+\overline{\Orb}'+\mathfrak{p}^\perp$.
This variety maps to
$\g^*$. Let $\pi$ denote the map onto its image, to be called the {\it generalized
Springer map}. Let $\Orb_\xi$ denote the open orbit in $\operatorname{im}\pi$, it is
known to be independent of the choice of $P$. Note that, similarly to \cite[Theorem 1.3]{LS}, $(\xi+\overline{\Orb}'+\mathfrak{p}^\perp)\cap \Orb_\xi$ is a single $P$-orbit.

\begin{Lem}\label{Lem:orbit_independ}
The open $G$-orbit in  $G\times^P(\xi+\overline{\Orb}'+\mathfrak{p}^\perp)$
depends only on $(\lf,\Orb',\xi)$, not on the choice of $P$.
\end{Lem}

In the proof we will need the following construction, also to be used later.
Let $\Orb$ denote a nilpotent orbit in $\g^*$. Pick an element
$\chi\in \Orb$ and let $Q$ be a maximal reductive subgroup of the stabilizer $G_\chi$.
Also we can pick an invariant form on $\g$ and identify $\g\cong \g^*$.
Let $e\in \g$ be the element corresponding to $\chi$. Include $e$
into an $\slf_2$-triple $e,h,f$. We can assume that $h,f$ are $Q$-stable,
then $Q=Z_G(e,h,f)$. The component group $Q/Q^\circ$ is identified with the $G$-equivariant fundamental group of $\Orb$, to be denoted by
$A(\Orb)$.

Consider the Slodowy slice $e+\mathfrak{z}_\g(f)$
and let $S$ denote its image in $\g^*$. Then $S$ is a transversal slice
to $\Orb$. Note that it is a Poisson variety with a $Q$-action.
We also have a contracting $\C^\times$-action on $S$: let $\gamma:
\C^\times\rightarrow G$ denote the one-parameter subgroup corresponding
to $h$, then we define a $\C^\times$-action on $S$ via
$t.s:=t^{-2}\gamma(t)s$. This action contracts $S$ to $\chi$. The proofs can be found e.g. in
\cite[Section 2]{GG}.

\begin{proof}[Proof of Lemma \ref{Lem:orbit_independ}]
The proof is in several steps.

{\it Step 1}.
The claim easily reduces
to the case when $\xi$ is central (by replacing $G$ with the stabilizer $G_{\xi}$).
We can then assume that $\xi=0$. Let $\Orb$ denote the open orbit
in $\operatorname{im}\pi$. Let $X':=\operatorname{Spec}(\C[\Orb'])$.
Set $\check{X}=
G\times^P((\lf/[\lf,\lf])^*\times X'\times \mathfrak{p}^\perp)$.
The open orbit in $\check{X}$ coincides with that in
$G\times^P(\overline{\Orb}'+\mathfrak{p}^\perp)$.

The variety $\check{X}$ naturally
maps to $\g^*$ and also to $(\lf/[\lf,\lf])^*$. The morphism to $(\lf/[\lf,\lf])^*$ is flat, $G$-invariant
and $\C^\times$-equivariant. Note that
\begin{equation}\label{eq:dim_formula}
\dim \check{X}=\dim (\lf/[\lf,\lf])^*+\dim \Orb.
\end{equation}

{\it Step 2}. Consider the preimage $\check{S}$ of $S$ in $\check{X}$ and the restriction of $\check{X}\rightarrow (\lf/[\lf,\lf])^*$ to
$\check{S}$. Our goal is to prove that this restriction is finite and \'{e}tale.

We start by showing $\dim \check{S}=\dim (\lf/[\lf,\lf])^*$.
The slice $S$ is transverse to all $G$-orbits it intersects, see, e.g., \cite[Section 2.2]{GG}.
Equivalently, the action morphism $G\times S\rightarrow \g^*$ is smooth of relative dimension $\dim S$. Since being smooth is stable under pullback, the action morphism $G\times \check{S}\rightarrow \check{X}$ is also
smooth of relative dimension $\dim S$. It follows that
\begin{align*}\dim \check{S}=\dim \check{X}+\dim S-\dim \g=(\dim (\lf/[\lf,\lf])^*+\dim \Orb)+\dim S-\dim\g=\dim (\lf/[\lf,\lf])^*.
\end{align*}
The first equality follows from (\ref{eq:dim_formula}), and the second  follows because $\dim \g=\dim S+\dim \Orb$.

{\it Step 3}. The variety $\check{S}$ inherits actions of $\C^\times$ and $Q$ from $S$. The morphism
$\check{S}\rightarrow (\lf/[\lf,\lf])^*$ is $Q$-invariant and $\C^\times$-equivariant by the construction.
We claim that the $\C^\times$-action on $\check{S}$ is contracting. Indeed, the morphism
$\check{X}\rightarrow \g^*\times (\lf/[\lf,\lf])^*$ is projective. It follows that the morphism
$\check{S}\rightarrow S\times (\lf/[\lf,\lf])^*$ is projective. This morphism is $\C^\times$-equivariant.
The $\C^\times$-action on $S\times (\lf/[\lf,\lf])^*$ contracts $S\times (\lf/[\lf,\lf])^*$ to a point. Hence the action
on $\check{S}$ is contracting as well.

{\it Step 4}. We claim that the morphism $\check{S}\rightarrow (\lf/[\lf,\lf])^*$ is unramified. Thanks to
the contracting $\C^\times$-action we need to show that the fiber over zero is a reduced finite scheme. This fiber coincides with the scheme theoretic preimage of $S$ in $G\times^P(X'\times \mathfrak{p}^\perp)$.
The morphism $G\times^P(X'\times \mathfrak{p}^\perp)\rightarrow \overline{\Orb}$ is \'{e}tale over
$\Orb$. The slice $S$ intersects $\Orb$ transversally at the single point, $\chi$, and does not intersects
the boundary of $\Orb$. It follows that the scheme theoretic preimage of $S$ in $G\times^P(X'\times \mathfrak{p}^\perp)$ coincides with the preimage of $\chi$ in the open orbit of
$G\times^P(X'\times \mathfrak{p}^\perp)$. This and the claim that the intersection of
$S$ with $\Orb$ is transversal show the claim in the beginning of the step.
Note that the preimage of $S$ in $G\times^P(X'\times \mathfrak{p}^\perp)$ is a single $Q/Q^\circ$-orbit.

{\it Step 5}. We have seen in Steps 2-4 that $\dim \check{S}=\dim (\lf/[\lf,\lf])^*$ and the morphism $\check{S}\rightarrow (\lf/[\lf,\lf])^*$ is
$\C^\times$-equivariant and etale.  Since the $\C^\times$-action on $\check{S}$ is contracting
(Step 3) and the preimage of $0\in (\lf/[\lf,\lf])^*$ is finite, the morphism $\check{S}\rightarrow (\lf/[\lf,\lf])^*$
is finite.
The group $Q$ acts on $\check{S}$ and the map
$\check{S}\rightarrow (\lf/[\lf,\lf])^*$ is $Q$-invariant.
It follows that $Q^\circ$ acts on $\check{S}$ trivially.

{\it Step 6}. As we have pointed out in Step 4,
the preimage of $\chi$ in $\check{S}$ is a single $Q/Q^\circ$-orbit, that is
the preimage of $\chi$ in the open orbit in $G\times^P(X'\times \mathfrak{p}^\perp)$.
Since the morphism $\check{S}\rightarrow (\lf/[\lf,\lf])^*$ is
\'{e}tale and finite, the fiber of $\check{S}$ over any point of $(\lf/[\lf,\lf])^*$ is the same $Q/Q^\circ$-orbit. Also note that the fiber of $\check{X}\rightarrow (\lf/[\lf,\lf])^*$ over
a generic point  $\xi_1\in(\lf/[\lf,\lf])^*$ is independent of the choice of $P$: this fiber is
$G\times^L (\xi_1\times X')$. The morphisms $G\times^L (\xi_1\times X')\rightarrow \g^*,(\lf/[\lf,\lf])^*$
are also independent of the choice of $P$.
So the fiber of $\check{S}\rightarrow (\lf/[\lf,\lf])^*$ over such a point is
independent of the choice of $P$. This is also the fiber over $\chi$. The claim of the lemma follows.
\end{proof}

Below we write $H$ for the stabilizer $G_x$.

Recall, Definition \ref{defi:induction}, that $(\lf,\Orb',\xi)$
is called an induction datum. When $\xi=0$ we omit it and say that
$(\lf,\Orb')$ is an induction datum.  We also can talk about
birationally induced and birationally rigid orbits.
More generally, we say that the open orbit in
$G\times^P(\xi+\overline{\Orb}'+\mathfrak{p}^\perp)$, an equivariant
cover of a coadjoint orbit,
is {\it birationally induced} from $(\lf,\Orb',\xi)$.
Finally, recall, Definition \ref{defi:induction}, that if $\pi$ is birational and $\Orb'$
is birationally rigid, then we say that $(\lf,\Orb',\xi)$
is a {\it birationally minimal induction datum}.

We have the following properties of birationally minimal induction data.

\begin{Prop}\label{Prop:bir_rigid}
For any fixed $\lf$ and birationally rigid $\Orb'\subset \lf$,
the set $(\lf/[\lf,\lf])_{\Orb'}^{*reg}$ of all $\xi\in (\lf/[\lf,\lf])^*$ such that
$(\lf,\Orb',\xi)$ is birationally minimal, is the complement to a finite union of
subspaces. Also this locus is independent of the choice of  $P$.
\end{Prop}
\begin{proof}
The independence of $P$ directly follows from Lemma \ref{Lem:orbit_independ}.

Now let $\xi\in (\lf/[\lf,\lf])^*$. We can naturally embed $\lf^*$ into $\g^*$ because there is
a unique $\lf$-invariant complement to $\lf$ in $\g$.
Hence we can view $\xi$ as an element of $\g^*$. Set $\tilde{\lf}:=\g_\xi$,
this is a Levi subalgebra.
Clearly, $\xi\in (\lf/[\lf,\lf])^{*reg}_{\Orb'}$ if and only if
the induction from $(\lf,\Orb')$ to $\tilde{\lf}$ is birational.
So the inclusion     $\xi\in (\lf/[\lf,\lf])^{*reg}_{\Orb'}$
depends only on $\tilde{\lf}$.
There is a finite number of choices of $\tilde{\lf}$.
So in order to complete the proof of the proposition, we only
need to show that if the induction from $(\lf,\Orb')$ is birational, then
so is the induction from $(\lf,\xi+\Orb')$ for any $\xi\in (\lf/[\lf,\lf])^*$.
In order to see this consider the varieties $\check{S}\subset \check{X}$
from the proof of Lemma \ref{Lem:orbit_independ}. The claim that the induction from
$(\lf,\Orb')$ is birational is equivalent to the condition that the preimage of
$0\in (\lf/[\lf,\lf])^*$ in $\check{S}$ is a single point. According to Step 5 of that
proof, the morphism $\check{S}\rightarrow (\lf/[\lf,\lf])^*$ is finite and etale.
It follows that it is an isomorphism. In particular, the preimage of $S$
in $G\times^P(\xi\times X'\times \mathfrak{p}^\perp)$ is a single point for
each $\xi\in (\lf/[\lf,\lf])^*$. So the morphism from the open
orbit in   $G\times^P(\xi\times X'\times \mathfrak{p}^\perp)$ to $\g^*$
must be generically injective. Equivalently, the induction from
$(\lf,\xi+\Orb')$ is birational.
\end{proof}


\subsection{Sheets and birational sheets}\label{SS_bir_sheet}
Recall that by a {\it sheet} in $\g$ we mean an irreducible component
of $\{x\in \g| \dim Gx=d\}$ for some $d$. Each sheet contains
a single nilpotent orbit. The rigid nilpotent orbits are precisely the
orbits $\Orb$ such that there is only one sheet containing $\Orb$ and this
sheet is $\Orb+(\g/[\g,\g])^*$, see, e.g., \cite{Borho,BK_sheet}.
In general, the sheets are indexed by the pairs $(\lf,\Orb')$, where $\Orb'$
is a rigid orbit in $\lf$: the corresponding sheet consists of orbits
induced from $(\lf,\Orb',\xi)$ for $\xi\in (\lf/[\lf,\lf])^*$.

Let $Z$ be a sheet and $\Orb$ be the unique nilpotent orbit contained in $Z$.
It turns out that the action of $G$ on $Z$ admits a geometric quotient.
Recall the subgroup $Q\subset G$ and an affine subspace $S\subset \g^*$
from Section \ref{SS_LS_induction}.   The following is
\cite[Theorems 0.3,0.4]{Katsylo}. Note that the component group of $Q$ is $A(\Orb)$.

\begin{Prop}\label{Prop:sheet_Katsylo}
The following is true:
\begin{enumerate}
\item The action of $Q^\circ$ on $Z\cap S$ is trivial and $A(\Orb)$
permutes the components of $Z\cap  S$ transitively.
\item The variety $(Z\cap S)/A(\Orb)$ is the geometric quotient
for the action of $G$ on $Z$, in particular, it is a categorical quotient
and each fiber of $Z\rightarrow (Z\cap S)/A(\Orb)$ is a single
$G$-orbit.
\end{enumerate}
\end{Prop}

%

In general, sheets do not behave well: they may intersect, may fail to be
smooth, etc. {\it Birational sheets} to be introduced now do not intersect and are very often smooth
(and always smooth up to a bijective normalization). However, there are birational sheets that do not
contain nilpotent orbits.

Pick a Levi subalgebra $\lf$ and a birationally rigid nilpotent orbit $\Orb'\subset\mathfrak{l}$.
Let $Z^{bir}_{\lf,\Orb'}$ denote the set of all orbits birationally induced from
$(\lf,\Orb',\xi)$, where $\xi\in(\lf/[\lf,\lf])_{\Orb'}^{*reg}$.

To $(\lf,\Orb')$ we assign a finite group $W(\lf,\Orb')$ acting on $(\lf/[\lf,\lf])^*$ as follows.
The group $N_G(L)$ naturally acts on the set of nilpotent orbits in $\lf^*$.
Let $N_G(L,\Orb')$ denote the stabilizer of $\Orb'$. We set $W(\lf,\Orb'):=N_G(L,\Orb')/L$.
Note that this finite group naturally acts on $(\lf/[\lf,\lf])^*$. It follows from
Proposition \ref{Prop:bir_rigid} (namely from the independence of $P$ part) that
$(\lf/[\lf,\lf])_{\Orb'}^{*reg}$ is $W(\lf,\Orb')$-stable.

The following theorem describes basic properties of birational sheets.

\begin{Thm}\label{Thm:bir_sheets}
The following is true:
\begin{enumerate}
\item Any coadjoint orbit is induced from a unique (up to $G$-cojugacy)
birationally minimal induction datum. In particular,  we have $\g=\bigsqcup_{(\lf,\Orb')}Z^{bir}_{\lf,\Orb'}$.
\item $Z^{bir}_{\lf,\Orb'}$ is a  locally closed subvariety of $\g$.
Its normalization is smooth and the morphism from the normalization
to $Z^{bir}_{\lf,\Orb'}$ is bijective.
\item A geometric quotient for the $G$-action on $Z^{bir}_{\lf,\Orb}$
exists. The normalization of the quotient  is $(\lf/[\lf,\lf])_{\Orb'}^{*reg}/W(\lf,\Orb')$.
This is a smooth variety. The morphism from the normalization to
the quotient is bijective.
\end{enumerate}
\end{Thm}

This theorem will be proved in Section \ref{SS_sheet_conseq}.

\subsection{Structure of $\Q$-terminalizations}\label{SS_nilp_Q_term}
In order to prove Proposition \ref{Prop:bir_rigid} and Theorem \ref{Thm:bir_sheets}
we will need to examine the structure of $\Q$-terminalizations. Results of this section have been
already obtained by Namikawa, \cite{Namikawa_induced}, for classical types,
and Fu, \cite{Fu}, for exceptional types. Their proofs used case-by-case
arguments, while our proof is conceptual.

Our main result in this section is the following proposition.

\begin{Prop}\label{Prop:birrig_Q_termin}
Let $\Orb$ be a birationally rigid nilpotent orbit and let
$X:=\operatorname{Spec}(\C[\Orb])$. Then the following claims hold:
\begin{enumerate}
\item $\param=\{0\}$, equivalently (Proposition \ref{Prop:X_univ_graded}),
$\C[X]$ has no nontrivial filtered Poisson deformations.
\item
$X$ is $\Q$-factorial
and terminal, and $H^2(X^{reg},\C)=\{0\}$.
\end{enumerate}
\end{Prop}
\begin{proof}
In the proof we can assume that $G$ is semisimple and simply connected.
We also can assume that the action of $G$ on $\Orb$ is faithful, equivalently,
$\Orb$ projects nontrivially to all direct summands of $\g$.

First, we show that (1) implies (2) (in fact, they are equivalent but we do not need this).
By \cite[Theorem 5.5(c)]{Namikawa1}, for a generic element $\lambda\in \param$, the morphism
$\tilde{X}_\lambda\rightarrow X_\lambda$ is an isomorphism. In particular, if $\param=\{0\}$,
then $\tilde{X}\xrightarrow{\sim}X$, hence $X$ is $\Q$-factorial and terminal. The claim that
$\param=\{0\}\Rightarrow H^2(X^{reg},\C)=\{0\}$ follows from Lemma \ref{Lem:Cartan_compute}.



It remains to prove that if $\Orb$ is birationally rigid, then (1) holds, i.e.,
$\param=\{0\}$. The proof of this is in several steps.

{\it Step 1}. Here we consider the situation when $\Orb$ is an arbitrary nilpotent orbit. We consider the  deformation $X_{\param}$. We claim that
we have a Hamiltonian $G$-action on $X_{\param}$ with moment map deforming that
on $X$.

Consider the degree $1$ component $\C[X_{\param}]_1$. Note that
$\g\hookrightarrow \C[X]_1$ and $\C[X_{\param}]_1=\C[X]_1\oplus \param$. Also note that since
$d=1$, $\C[X_{\param}]_1$ is a Lie algebra with respect to $\{\cdot,\cdot\}$
and the embedding $\g\hookrightarrow \C[X]_1$ is an inclusion of Lie algebras.
So we get an extension of $\g$ by the abelian Lie algebra $\param$.
Therefore $\g$ canonically splits. Note also that $\C[X_{\param}]$ acquires
a $\C[\param]$-linear action of $\g$ via $\{\cdot,\cdot\}$.
This action preserves the grading, hence it is locally finite. Its weights are the same as in
$\C[X]$. Hence it integrates
to $G$. So we get the required
Hamiltonian action.  Let $\mu_{\param}:X_{\param}\rightarrow \g^*$ denote the moment map
and $\mu_\lambda$ be the restriction of $\mu_{\param}$ to $X_\lambda$.

We remark that $\mu_{\param}$ is $W_X$-invariant, equivalently, the image of $\g$ in $\C[X_\param]$
is $W_X$-invariant. Recall that  $W_X$ acts on
$\C[X_\param]$ by graded Poisson automorphisms. The subspace $\C[X_{\param}]_1$ is $W_X$-stable and the projection $\C[X_{\param}]_1\twoheadrightarrow \C[X]_1$ is $W_X$-equivariant. The action on $\C[X_{\param}]_1$
is by Lie algebra automorphisms. It follows that the preimage of $\g$ in $\C[X_{\param}]_1$ is $W_X$-stable. Since the action is by Lie algebra automorphisms and $\g$ splits canonically,
$\g$ is $W_X$-stable. And since the projection of $\g$ to $\C[X]_1$ is injective, the action of
$W_X$ on $\g$ is trivial. So $\mu_\param$ is indeed $W_X$-invariant.

{\it Step 2}. We still assume that $\Orb$ is an arbitrary nilpotent orbit. We claim that, for every $\lambda\in \param$,
the $G$-action on $X_\lambda$ has an open orbit.
There is   an open $G$-orbit on $X=X_0$. The locus of $x\in X$ with $\dim Gx\geqslant \dim X$
is open and $\C^\times$-stable. Since the action of $\C^\times$ is contracting, there is an open orbit in every fiber $X_\lambda$.
So $\mu_\lambda$ is generically finite and $\mu_\lambda(X_\lambda)$
is the closure of a single orbit to be denoted by $\Orb_\lambda$. Since $\mu_0$ is finite
and the action of $\C^\times$ on $X_\param$ is contracting, $\mu_\lambda$ is finite for all $\lambda$. Note that $\overline{\Orb}\subset\overline{\C^\times \Orb_\lambda}$ and $\dim \Orb_\lambda=\dim \Orb$.

{\it Step 3}. By Step 2, $X_{\param}\quo G=\param$. It follows that $\mu_{\param}$ induces a morphism
$\param\rightarrow \g^*\quo G$. We claim that if $\param\neq \{0\}$, then the image is different from
$\{0\}$. Indeed, otherwise $\Orb_\lambda=\Orb$ for all $\lambda$.
Moreover, since $X_{\param}$ is a flat deformation of $X$, we see that
the open orbit in $X_\lambda$ is $\Orb$ and not its proper cover. We get the map $X_{\param}
\rightarrow \param\times \overline{\Orb}$ induced by $\mu_{\param}$.
This map is  finite and birational so it is the normalization. Also
it is Poisson and  $\C^\times$-equivariant. So it lifts to a
$\C^\times$-equivariant Poisson isomorphism $X_{\param}\xrightarrow{\sim} X\times \param$.
Since $X_\param/W_X$ is a universal deformation, this implies $\param=\{0\}$. Contradiction. We see that if $\param\neq \{0\}$,
then the image of $X_{\param}$ in $\g$ contains non-nilpotent elements.

{\it Step 4}. Now assume that $\Orb$ is birationally rigid. Pick a Zariski generic $\lambda\in  \param$. Let $\xi$ be the
semisimple part of an element in $\Orb_\lambda$, $\lf$ be the stabilizer
of $\xi$ in $\g$ and $\Orb'$ be the nilpotent orbit in $\lf$ such that
$\xi+\Orb'\subset \Orb_\lambda$. We claim that $\Orb$ is birationally induced
from $(\lf,\Orb')$. Namely, let $X'$ be the normalization of $\overline{\Orb}'$
and $\tilde{X}'$ be its $\Q$-terminalization.
Consider the variety $\tilde{X}^1_{\C\xi}:=G\times^P(\C\xi\times \tilde{X}'\times \mathfrak{p}^\perp)$.
Here $\mathfrak{p}$ be a parabolic subalgebra with Levi subalgebra $\lf$ and
$P$ is the corresponding subgroup. The  variety $\tilde{X}^1_{\C\xi}$ is a normal Poisson $\C\xi$-scheme
and has a Hamiltonian $G$-action. Its algebra of regular functions is finite over $\C[\g^*]$
hence is finitely generated. Let $X^1_{\C\xi}:=\operatorname{Spec}(\C[\tilde{X}^1_{\C\xi}])$.
Note that the fiber of $X^1_{\C\xi}$ over $\xi$ is the normalization of $\overline{\Orb}_\lambda$. The fiber over $0$ is $X^1:=\operatorname{Spec}(\C[\tilde{X}^1])$,
where $\tilde{X}^1=G\times^P(\tilde{X}'\times \mathfrak{p}^\perp)$. Note that we get
a finite $G$-equivariant morphism $X^1\twoheadrightarrow \overline{\Orb}$ that factors
through $X^1\twoheadrightarrow X$.

Since the morphism $X_\param\rightarrow \param$ is $G$-invariant and flat, for any irreducible $G$-module $V$, we have that $\Hom_G(V,\C[X_{\param}])$ is flat over $\C[\param]$. As was mentioned in the beginning of Step 3, $\C[X_{\param}]^G\xrightarrow{\sim}\C[\param]$. So $\Hom_G(V,\C[X_{\param}])$ is also finitely generated $\C[\param]$-module, hence is projective. It follows that $\C[X]\cong \C[X_\lambda]$ as $G$-modules.
Analogously, $\C[X^1]\cong \C[X^1_\lambda]$. We have $G$-equivariant finite dominant
morphisms $X^1\twoheadrightarrow X$ and $X_\lambda\twoheadrightarrow X^1_\lambda$,
the latter is induced by the moment map $\mu_\lambda$. Hence $\C[X^1]\cong \C[X]$ as $G$-modules.
We conclude that $X^1\xrightarrow{\sim} X$. This contradicts $\Orb$ being birationally
rigid and completes the proof of the proposition.
\end{proof}

\begin{Cor}\label{Cor_nilp_Q_termin}
Let $\Orb$ be a nilpotent orbit. Then the following claims are true.
\begin{enumerate}
\item  There is a unique (up to $G$-conjugacy)
birationally minimal induction datum  $(\lf,\Orb')$
for $\Orb$.
\item  The variety  $\tilde{X}:=G\times^P (X'\times \mathfrak{p}^\perp)$ (where $X'$
stands for $\operatorname{Spec}(\C[\Orb'])$) is a $\Q$-terminalization
of $X$.
\end{enumerate}
\end{Cor}
\begin{proof}
We start by proving (2). We can find some minimal birational induction datum $(\lf,\Orb')$
for $\Orb$.
The variety $X'$ is $\Q$-factorial terminal by Proposition \ref{Prop:birrig_Q_termin}.
Hence so is $\tilde{X}$.  Hence it is a $\Q$-terminalization
of $X$.

Let us prove (1). We  need to prove that $(\lf,\Orb')$
is unique up to conjugacy. Note that the morphism
$X_{\param}/W_X\twoheadrightarrow \param/W_X$ is uniquely recovered from $X$
thanks to Proposition \ref{Prop:X_univ_graded} and does not depend on the choice of the $\Q$-terminalization
$\tilde{X}=G\times^P (X'\times \mathfrak{p}^\perp)$, i.e., the choices of $P$ and $(\lf,\Orb')$.
For $\lambda\in \param$, the moment map $X_\lambda\rightarrow \g^*$ is also independent of these
choices.
According to Step 4 of the proof Proposition \ref{Prop:birrig_Q_termin}, we recover $(\lf,\Orb')$
as follows. Consider the fiber $X_\lambda$ of $X_{\param}$ over a Zariski generic point $\lambda\in \param$. Let $\Orb_\lambda$ denote the open orbit in the image of $X_\lambda$ in $\g^*$ and let $x\in \Orb_\lambda$. Then $\lf$ is the stabilizer of the semisimple part of $x$ and
and $\Orb'$ is the $L$-orbit of the nilpotent part of $x$.
\end{proof}

\subsection{Computation of Weyl groups}\label{SS_Weyl_comput}
In this section we will  get some information on the Namikawa-Weyl group $W_X$
for $X:=\operatorname{Spec}(\C[G/H])$, where $G/H$ is the open $G$-orbit in
 $\tilde{X}=G\times^P(X'\times \mathfrak{p}^\perp)$. Here $X'=\operatorname{Spec}(\C[\Orb'])$
for a birationally rigid nilpotent orbit  $\Orb'\subset \mathfrak{l}^*$.

Our main result is as follows.

\begin{Prop}\label{Prop:NW_comput}
Assume $G$ is semisimple. The following claims are true:
\begin{enumerate}
\item
We have $\param=(\lf/[\lf,\lf])^*$.
\item The universal deformation $\tilde{X}_\param$ of $\tilde{X}$ over $\param$
(Proposition \ref{Prop:KV_deform}) is isomorphic to $\check{X}:=G\times^P(\param\times X'\times \mathfrak{p}^\perp)$ as Poisson scheme over $\param$ with a $\C^\times$-action.
\item
The Namikawa-Weyl group $W=W_X$ of $X$ is a normal subgroup in $W(\lf,\Orb')$. Moreover,
the quotient $W(\lf,\Orb')/W$ is isomorphic to the group $A$ of $G$-equivariant Poisson automorphisms
of $X$.
\end{enumerate}
\end{Prop}

The group $A$ is naturally identified with
$N_{Z_G(x)}(H)/H$, where $x\in G/H$. In particular, if $H=Z_G(x)$, we see that $W=W(\lf,\Orb')$. A formally weaker result
(where $\Orb'=\{0\}$) was  obtained in \cite[Section 2]{Namikawa2}. Note also that $A$
coincides with the group of graded $G$-equivariant automorphisms of $\C[X]$.

\begin{proof}[Proof of Proposition \ref{Prop:NW_comput}]
To compute $\param$ we note that since $H^i(X'^{reg},\C)=\{0\}$ for $i=1,2$ (see
Proposition \ref{Prop:birrig_Q_termin}) and $H^1(G/P,\C)=0$, we get
$H^2(\tilde{X}^{reg},\C)=H^2(G/P,\C)=(\lf/[\lf,\lf])^*$.

Now we prove (2). Note that $\check{X}$ is a graded deformation of $\tilde{X}$ in the sense of Section
\ref{SS_filt_deform}. By Proposition \ref{Prop:KV_deform}, there is a unique linear map
$(\lf/[\lf,\lf])^*\rightarrow \param$ and a $\C^\times$-equivariant Poisson isomorphism
$\check{X}\xrightarrow{\sim} (\lf/[\lf,\lf])^*\times_{\param} \tilde{X}_\param$. We need to show that
$(\lf/[\lf,\lf])^*\rightarrow \param$ is an isomorphism. Thanks to (1) it is enough to show it is
injective. Assume the contrary: let $\lambda\in (\lf/[\lf,\lf])^*$ go to $0$. This yields an
isomorphism $G\times^P(X'\times \mathfrak{p}^\perp)\xrightarrow{\sim} G\times^P(\{\lambda\}\times X'\times \mathfrak{p}^\perp)$. In particular, we have an embedding $G/P\hookrightarrow G\times^P(\{\lambda\}\times X'\times \mathfrak{p}^\perp)$. The target is a homogeneous bundle over the affine variety  $G/G_\lambda$
with fiber $G_\lambda\times^{G_\lambda\cap P}(\{\lambda\}\times X'\times \mathfrak{p}^\perp)$.
Any morphism from $G/P$ to $G/G_\lambda$ maps $G/P$ to a single point because $G/P$ is projective and connected while $G/G_\lambda$ is affine. So we have a closed embedding
$G/P\hookrightarrow G_\lambda\times^{G_\lambda\cap P}(\{\lambda\}\times X'\times \mathfrak{p}^\perp)$.
The latter variety is a bundle over the projective variety $G_\lambda/(G_\lambda\cap P)$
whose fiber is affine. So the composition $$G/P\hookrightarrow G_\lambda\times^{G_\lambda\cap P}(\{\lambda\}\times X'\times \mathfrak{p}^\perp)\twoheadrightarrow G_\lambda/(G_\lambda\cap P)$$
is a finite morphism. The latter is impossible because $\dim G/P>\dim G_\lambda/(G_\lambda\cap P)$.
This contradiction finishes the proof of the claim that the linear map $(\lf/[\lf,\lf])^*\rightarrow \param$ is an isomorphism and hence the proof of (2).

(3) will be proved in several steps.

{\it Step 1}. Note that, for a Zariski generic $\xi\in (\lf/[\lf,\lf])^*$, the orbits
$\Orb_\xi$ and $\Orb_{\xi'}$ induced from $(\lf,\Orb',\xi)$ and $(\lf,\Orb',\xi')$,
respectively, coincide (here $\xi'\in(\lf/[\lf,\lf])^*$) if
and only if $\xi'\in W(\lf,\Orb')\xi$. It follows that, for a Zariski generic
$\lambda\in (\lf/[\lf,\lf])^*$, the equality $\mu_\lambda(X_\lambda)=
\mu_{\lambda'}(X_{\lambda'})$ implies $\lambda'\in W(\lf, \Orb')\lambda$.
But the moment map $\mu_{\param}$ is $W$-invariant, see
Step 1 of the proof of Proposition \ref{Prop:birrig_Q_termin}.
It follows that  $W\subset W(\lf,\Orb')$.

{\it Step 2}. We are going to produce a group homomorphism $W(\lf,\Orb')\rightarrow A$.
Pick a Zariski generic element $\lambda\in (\lf/[\lf,\lf])^*$.
Consider the deformation
$X_{\C\lambda}$ of $X$ over $\C\lambda$. It comes with the morphism $X_{\C\lambda}\rightarrow \C\lambda\times_{\g^*\quo G}
\overline{\C^\times \Orb_\lambda}$ that is a normalization morphism. Note that
$w\in W(\lf,\Orb')$ defines a $\C^\times$-equivariant morphism
$$\C\lambda\times_{\g^*\quo G}\overline{\C^\times \Orb_\lambda}\xrightarrow{\sim}
\C w\lambda\times_{\g^*\quo G}\overline{\C^\times \Orb_{w\lambda}}$$
and hence a $\C^\times$-equivariant isomorphism $X_{\C\lambda}\rightarrow X_{\C w\lambda}$.
This isomorphism is $G$-equivariant and intertwines the moment maps, hence it is Poisson.
Specializing to $0$, we get an element of $A$ to be denoted by $a_{w,\lambda}$.
Note that the group $A$ is finite.
So varying $\lambda$, we get the same element $a_{w,\lambda}$, we will write $a_{w}$ for $a_{w,\lambda}$.
Since $a_{w_1w_2,\lambda}=a_{w_1,w_2\lambda}a_{w_2,\lambda}$ by the construction, we see that $w\mapsto
a_w$ is a group homomorphism.

{\it Step 3}. Let us show that $W\subset W(\lf,\Orb')$ is the kernel of the homomorphism
$w\mapsto a_w$. By the previous step, $w$ induces a filtered Poisson algebra
isomorphism $\C[X_\lambda]\rightarrow \C[X_{w\lambda}]$ such that the induced
automorphism of $\C[X]$ is $a_w$. So  $w\in W$ if and only if
the isomorphism $\C[X_\lambda]\xrightarrow{\sim}\C[X_{w\lambda}]$ is that of
filtered deformations if and only if $a_w=1$.

{\it Step 4}. Let us prove that the homomorphism $W(\lf,\Orb')\rightarrow A$
is surjective. Recall that $A$ acts on $\C[\param]^W$ by graded algebra
automorphisms and on $\C[X_{\param}]^W$ by graded Poisson algebra automorphisms.
Under this action, $\g\subset \C[X_{\param}]^W$ stays fixed. It follows that if $W\lambda, W\lambda'$ are $A$-conjugate, then
$\mu_\lambda(X_\lambda)=\mu_{\lambda'}(X_{\lambda'})$.

Also if $a\in A$
acts trivially on $\C[\param]^W$, then it acts trivially on $\param/W\times_{\g^*\quo G}\g^*$
and hence on its normalization $X_{\param}/W$. Therefore $a=1$.

 In particular,
we can take a Zariski generic $\lambda$ and use Step 1 to see that if $a$ sends $W\lambda$
to $W\lambda'$, then $\lambda,\lambda'$
are $W(\lf,\Orb')$-conjugate. Since $A$ acts faithfully on $\param/W$, this implies the surjectivity of $W(\lf,\Orb')\rightarrow A$.
\end{proof}

Below we will need to relate Weyl groups for inductions to  $\g$ and to some Levi
subalgebra of $\g$. Namely, take a Levi $\lf\subset\g$ and a birationally rigid
nilpotent orbit $\Orb'\subset\lf^*$. Let $\underline{\g}$ denote a Levi subalgebra
of $\g$ containing $\lf$. Let $\underline{\Orb}$ be the nilpotent orbit in
$\underline{\g}^*$ induced from $(\lf,\Orb')$ and assume that the induction
is birational. Let $\underline{X}$ denote the normalization of $\overline{\underline{\Orb}}$
and let $X:=\operatorname{Spec}(\C[G\times^P(X'\times\mathfrak{p}^\perp)])$.

\begin{Lem}\label{Lem:Weyl_stab}
The group $W_{\underline{X}}$ is contained in  the pointwise stabilizer of
$(\underline{\g}/[\underline{\g},\underline{\g}])^*\subset (\lf/[\lf,\lf])^*$ in $W_X$.
\end{Lem}
\begin{proof}
In the proof we will need a slightly different construction of the homomorphism $W(\lf,\Orb')\rightarrow A$. Pick an element $w\in W(\lf,\Orb')$ and its lift $g$ to $N_G(\lf,\Orb')$. Set $P^1:=gPg^{-1}$, this is a parabolic subgroup of $G$ with Levi subgroup $L$.
Consider the universal deformations of $\Q$-factorial terminalizations associated to $P,P^1$, i.e.,
$\tilde{X}_{\param}:=G\times^P((\lf/[\lf,\lf])^*\times X'\times \mathfrak{p}^\perp),
\tilde{X}^1_{\param}:=G\times^{P^1}((\lf/[\lf,\lf])^*\times X'\times \mathfrak{p}^{1,\perp})$.
Note that $g$ gives rise to an isomorphism of varieties $\tilde{X}_\param\rightarrow \tilde{X}^1_\param$.
It is not an isomorphism of schemes over $\param$, rather the induced map $\param\rightarrow \param$
is $w$. Passing to the spectra of the rings of regular functions we get an isomorphism
$X_\param\rightarrow X_\param$. Its specialization to $0$ coincides with $a_w$: in order to see this we can restrict to $\C\lambda\subset \param$ for a Zariski generic $\lambda\in \param$, we recover the construction from Step 2 of the proof of Proposition \ref{Prop:NW_comput}.

Now we need to produce an embedding of $W_{\underline{X}}$ into the pointwise stabilizer of
$(\underline{\g}/[\underline{\g},\underline{\g}])^*$ inside $W_X$. Thanks to our assumption that the induction from $L$ to $\underline{G}$ is birational, we have $W_{\underline{X}}=\underline{W}(\lf,\Orb')$,
where $\underline{W}(\lf,\Orb')$ is the analog of $W(\lf,\Orb')$ for $\underline{G}$.
Note that $\underline{W}(\lf,\Orb')$ coincides with the pointwise stabilizer of
$(\underline{\g}/[\underline{\g},\underline{\g}])^*$ in $W(\lf,\Orb')$. So we need to prove that
$\underline{W}(\lf,\Orb')$ maps trivially to $A$.

Set $\underline{X}:=\C[\underline{\Orb}]$ and let $\underline{\param}:=((\lf\cap [\underline{\g},\underline{\g}])/[\lf,\lf])^*$. Every element  $w\in\underline{W}(\lf,\Orb')$
defines an automorphism of $\underline{X}_{\underline{\param}}$ that is the identity on $\underline{X}$.
Now choose $P$ such that $\overline{P}:=P\underline{G}$ is a parabolic subgroup in $G$.
Construct $P^1,\tilde{X}^1_\param$ from $w$.
The morphisms $\tilde{X}_{\param},\tilde{X}^1_\param\rightarrow X_\param$ factor through
$G\times^{\overline{P}}((\underline{\g}/[\underline{\g},\underline{\g}])^*\times \underline{X}_{\underline{\param}}\times \overline{\mathfrak{p}}^\perp)$. An element $g\in N_{\underline{G}}(\lf,\Orb')$ lifting $w$ defines an automorphism of the latter variety
that is the identity on $G\times^{\overline{P}}((\underline{\g}/[\underline{\g},\underline{\g}])^*\times \underline{X}\times \overline{\mathfrak{p}}^\perp)$ because $w\in W_{\underline{X}}$. It follows that the induced automorphism of
$X$ is the identity finishing the proof.
%
\end{proof}

In the proof of Theorem \ref{Thm:bir_sheets} in the next section we will see that in Lemma \ref{Lem:Weyl_stab} we actually have an equality.

\subsection{Consequences}\label{SS_sheet_conseq}
In this section we prove 
Theorem \ref{Thm:bir_sheets}.

We start with a criterium for $\C[X_\xi],\C[X_{\xi'}]$ with $\xi,\xi'\in (\lf/[\lf,\lf])_{\Orb'}^{*reg}$
to be isomorphic as filtered algebras.

\begin{Lem}\label{Lem:deform_coinc}
Let $\xi,\xi'\in (\lf/[\lf,\lf])_{\Orb'}^{*reg}$. Then the following are equivalent:
\begin{enumerate}
\item $\C[X_\xi],\C[X_{\xi'}]$ are $G$-equivariantly isomorphic as filtered Poisson algebras.
\item The $G$-orbits induced from $(\lf,\Orb',\xi),(\lf,\Orb',\xi')$ coincide.
\item $\xi'\in W(\lf,\Orb')\xi$.
\end{enumerate}
\end{Lem}
\begin{proof}
Clearly (3) implies (2). Let $\Orb_\xi,\Orb_{\xi'}$ denote the orbits
induced from $(\lf,\Orb',\xi)$ and $(\lf,\Orb',\xi')$. Note that
$\C[X_\xi]\cong\C[\Orb_\xi]$ (a $G$-equivariant Poisson isomorphism)
because $\xi\in (\lf/[\lf,\lf])_{\Orb'}^{*reg}$ and similarly
$\C[X_{\xi'}]\cong\C[\Orb_{\xi'}]$. The argument of Step 4 of the proof of
Proposition \ref{Prop:birrig_Q_termin} shows that if $\Orb_\xi=\Orb_{\xi'}$,
then, in the notation there, we have a $\C^\times$-equivariant isomorphism
$X^1_{\C\xi}\cong X^1_{\C\xi'}$. Since $\C[X_\xi]$ is the specialization
of $\C[X^1_{\C\xi}]$ at $1$ -- and similarly for $\C[X_{\xi'}]$,
 the resulting isomorphism  $\C[X_{\xi}]\cong \C[X_{\xi'}]$ is $G$-equivariant
and of filtered algebras. So (2) implies (1).

Let us show that (1) implies (3). The group of the graded $G$-equivariant Poisson automorphisms of $\C[X]$ is $A$. The implication (1)$\Rightarrow$(3) now follows from Remark \ref{Rem:equiv_action}.
\end{proof}

\begin{proof}[Proof of Theorem \ref{Thm:bir_sheets}]
The proof is in several steps.

{\it Step 1}.
Let us prove (1). Clearly, the birational sheets cover $\g$. So we  need
to prove that an orbit determines a birationally minimal induction datum
uniquely up to $G$-conjugacy (which, in particular, implies that the birational
sheets do not intersect). Recall that if an orbit $Gy$ with $y\in \g^*$ is induced from
$(\lf,\Orb',\xi)$, then $\xi$ is $G$-conjugate to $y_s$. This allows to reduce
the proof to the claim that every nilpotent orbit is birationally induced from
a unique birationally minimal induction datum. This is  (1) of Corollary \ref{Cor_nilp_Q_termin}.

{\it Step 2}.
To prove (2) and (3) we first need  to establish the following claim:
\begin{itemize}
\item[(*)] The action of $A=W(\lf,\Orb')/W_X$ on $(\lf/[\lf,\lf])_{\Orb'}^{*reg}/W_X$ is free.
\end{itemize}
Indeed, let $\xi\in (\lf/[\lf,\lf])^{*reg}_{\Orb'}$. Let $\underline{G}$ stand for the stabilizer
of $\xi$ in $G$. Let $\underline{\Orb}$ denote the nilpotent orbit in $\underline{\g}$ induced
from $(\lf,\Orb')$. Set $\underline{X}=\operatorname{Spec}(\C[\underline{\Orb}])$.
By Lemma \ref{Lem:Weyl_stab},  we have  $W_{\underline{X}}\subset  W_{X,\xi}$.
On the other hand, the group $\underline{W}(\lf,\Orb')$,
the analog of $W(\lf,\Orb')$ for $\underline{\g}$, coincides with the stabilizer $(W(\lf,\Orb'))_\xi$.
But the induction from $(\lf,\Orb')$ to $\underline{\Orb}$ is birational, so
by Proposition \ref{Prop:NW_comput}, we have $W_{\underline{X}}=\underline{W}(\lf,\Orb')$. Together with the equality $\underline{W}(\lf,\Orb')=(W(\lf,\Orb'))_\xi$,  this shows (*).

{\it Step 3}. Now let us prove  (3).  Let $X_{\param}^0$ denote the locus in
$\tilde{X}_{\param}$ consisting of orbits of maximal dimension, it embeds into $X_\param$.
Recall, (1) and (2) of Proposition \ref{Prop:NW_comput}, that $\param=(\lf/[\lf,\lf])^*$ and $\tilde{X}_\param=G\times^P(\param\times X'\times \mathfrak{p}^\perp)$.
As in the proof of Lemma \ref{Lem:orbit_independ}, consider the preimage
$\check{S}\subset \tilde{X}_{\param}$ of the Slodowy slice $S$ to $\Orb$ under
$\mu_{\param}:\tilde{X}_{\param}=G\times^P(\param\times X'\times \mathfrak{p}^\perp)
\rightarrow \g^*$. Clearly, $\check{S}\subset X^0_{\param}$. The morphism $\mu_{\param}:\check{S}\rightarrow
S$ is still proper. As we have seen in the proof of Lemma \ref{Lem:orbit_independ}, see Steps 4 and 5,
the morphism $\check{S}\rightarrow \param$ is  finite and \'{e}tale.

As we have seen in the proof of Lemma \ref{Lem:Weyl_stab},
we have a $W(\lf,\Orb')$-action on $X_{\param}$. Since $X_{\param}
\rightarrow \g^*$ is $W_X$-invariant and $S\subset \g^*$ is $Q$-stable,
we see that  $\check{S}\subset X^0_{\param}
\subset X_{\param}$ is $W(\lf,\Orb')$-stable. By Proposition \ref{Prop:sheet_Katsylo},
every $G$-orbit in $Z$ intersects $S$ in a single $A(\Orb)$-orbit. From here we deduce that the projection $\check{S}\twoheadrightarrow \param$ induces an isomorphism $\check{S}/W(\lf,\Orb')\xrightarrow{\sim} \param/W(\lf,\Orb')$.


Note that $Z^{bir}_{\lf,\Orb'}\cap S$ coincides with the image of
$\check{S}\cap \pi^{-1}((\lf/[\lf,\lf])_{\Orb'}^{*reg})$ under $\mu_{\param}$, where we write
$\pi$ for the projection $\tilde{X}_{\param}\rightarrow \param$. This image is the
complement of a closed subset in the image of the proper morphism $\mu_{\param}|_{\check{S}}$.
It follows that $Z^{bir}_{\lf,\Orb'}\cap S$ is a locally closed subvariety.
The proper epimorphism
$\check{S}\cap  \pi^{-1}((\lf/[\lf,\lf])_{\Orb'}^{*reg})\twoheadrightarrow
Z^{bir}_{\lf,\Orb'}/G$ factors through $(\check{S}\cap\pi^{-1}((\lf/[\lf,\lf])_{\Orb'}^{*reg}))/W(\lf,\Orb')$.
Moreover, it follows from Lemma \ref{Lem:deform_coinc} that if the images of two points
from $\check{S}\cap \pi^{-1}((\lf/[\lf,\lf])_{\Orb'}^{*reg})$ in $Z^{bir}_{\lf,\Orb'}/G$
coincide, then the points are $W(\lf,\Orb')$-conjugate.  So  the morphism
$(\check{S}\cap  \pi^{-1}((\lf/[\lf,\lf])_{\Orb'}^{*reg}))/W(\mathfrak{l},\mathbb{O}')\twoheadrightarrow
Z^{bir}_{\lf,\Orb'}/G$ is also injective.
Therefore it  is a bijective normalization morphism. The variety $\left(\check{S}\cap\pi^{-1}((\lf/[\lf,\lf])_{\Orb'}^{*reg})\right)/W(\lf,\Orb')=
(\lf/[\lf,\lf])_{\Orb'}^{*reg}/W(\lf,\Orb')$ is smooth by (*) in Step 2. This proves (3).

{\it Step 4}. Let us prove (2). Note that $X^0_{\param}\cap\pi^{-1}((\lf/[\lf,\lf])_{\Orb'}^{*reg})=\mu_{\param}^{-1}(Z^{bir}_{\lf,\Orb'})$.
The subvariety $X^0_{\param}$ is the union of orbits of maximal dimension in $X_\param$.
Since the action of $W(\lf,\Orb')$ on $X_{\param}$ commutes with that of $G$,
 we see that $W(\lf,\Orb')$ preserves $X^0_{\param}$.  By Lemma \ref{Lem:deform_coinc},
the induced morphism $(X^0_{\param}\cap\pi^{-1}((\lf/[\lf,\lf])_{\Orb'}^{*reg}))/W(\lf,\Orb')\rightarrow Z^{bir}_{\lf,\Orb'}$
is injective. It is also proper so it is finite. It remains to prove that $(X^0_{\param}\cap\pi^{-1}((\lf/[\lf,\lf])_{\Orb'}^{*reg}))/W(\lf,\Orb')$ is smooth.
This will follow if we show that $W(\lf,\Orb')$ acts on $X^0_{\param}\cap \pi^{-1}((\lf/[\lf,\lf])_{\Orb'}^{*reg})$
as a group generated by reflections (by a reflection in this case we mean an automorphism whose fixed
locus is a divisor). We have seen in Step 2 that $(W(\lf,\Orb'))_\xi=W_{\underline{X}}$ is a reflection
group (in its action on $(\lf/[\lf,\lf])^*$).
So it remains to check that $(W(\lf,\Orb'))_{\xi}$ fixes $X_\xi$ pointwise. Recall, Step 2,
that $W_{X,\xi}=(W(\lf,\Orb'))_\xi$. Assume that an element $w\in W_{X,\xi}$
acts on $X_\xi$ nontrivially. The action preserves the filtration on $\C[X_\xi]$ and is the identity
on the associated graded. But a finite group of automorphisms cannot contain a non-unit element with
these properties. This finishes the proof of (2).
\end{proof}

\begin{Rem}\label{Rem:bit_sheet_smooth}
In fact, in most cases, the birational sheets as well as their quotients are smooth.
Namely, $Z^{bir}_{\lf,\Orb'}$ is smooth provided $\Orb$ (the orbit induced from $(\lf,\Orb')$)
is not one of the seven orbits in \cite[Table 0]{PT} and is not induced
from one of these orbits. In particular, the birational sheets in classical
Lie algebras are always smooth. Let us sketch a proof of the smoothness.

First, let us consider the birational sheet $Z^{bir}_{\lf,\Orb'}$ containing $\Orb$.
Let $Z(\Orb)$ denote the union of sheets containing $\Orb$. One can show that
$Z^{bir}_{\lf,\Orb'}\cap S=(Z(\Orb)\cap S)^{A(\Orb)}$. Using techniques of \cite[Section 5]{PT}
(where an analogous result was proved in the quantum case), one shows
that $(Z(\Orb)\cap S)^{A(\Orb)}$ is an affine space provided $\Orb$ is not one
of the orbits in Table 0 in the introduction of {\it loc.cit.}. So we see that
$Z^{bir}_{\lf,\Orb'}\cap S=Z^{bir}_{\lf,\Orb'}/G$ is smooth. But, for $s\in S\cap Z^{bir}_{\lf,\Orb'}$,
we have $T_s Z^{bir}_{\lf,\Orb'}=T_s Gs\oplus T_s(Z^{bir}_{\lf,\Orb'}\cap S)$. It follows that
$Z^{bir}_{\lf,\Orb'}$ is smooth.

Now consider the  case when $Z^{bir}_{\lf,\Orb'}$
does not contain a nilpotent orbit. Pick $y\in Z^{bir}_{\lf,\Orb'}$. Let $\underline{G}=G_{y_s}$
and let $\underline{Z}^{bir}_{\lf,\Orb'}$ be the birational sheet in $\underline{\g}$ corresponding
to $\lf,\Orb'$, this birational sheet contains $y_n$. Then we have an \'{e}tale morphism $G\times^{\underline{G}}(y_s+\underline{Z}^{bir}_{\lf,\Orb'})\rightarrow Z^{bir}_{\lf,\Orb'}$
with $y$ lying in the image. We deduce from the previous paragraph that both
$Z^{bir}_{\lf,\Orb'}$ and $Z^{bir}_{\lf,\Orb'}/G$ are smooth provided $\Orb$ is not induced
from one of the seven orbits in \cite[Table 0]{PT}.
\end{Rem}

\begin{Rem}\label{Rem:bir_sheet_char}
One can ask for an intrinsic characterization of birational sheets.
As in Step 4 of the proof of Proposition \ref{Prop:birrig_Q_termin}
we see that for a pair $\Orb_1,\Orb_2$ of $G$-orbits lying in the same
birational sheet, the $G$-modules $\C[\Orb_1],\C[\Orb_2]$
are isomorphic. We conjecture that the converse is also true:
if the $G$-modules $\C[\Orb_1],\C[\Orb_2]$ are isomorphic
as $G$-modules, then $\Orb_1,\Orb_2$ lie in the same birational sheet.
\end{Rem}

\section{W-algebras and Orbit method}\label{S_W_Orb}
\subsection{W-algebras}\label{SS_W_algebras}
We start by recalling (finite) W-algebras that were originally defined by Premet
in \cite{Premet1}, although we will follow an approach from \cite{Wquant}.
Throughout the section $G$ is a semisimple group.

Pick a nilpotent orbit $\Orb\subset\g^*$. Let $\chi\in \Orb$. Recall that $S$
stands for a Slodowy slice to $\Orb$ in $\chi$. It is acted on by $Q\times \C^\times$,
where $Q$ is a maximal reductive subgroup of $G_\chi$.

We recall a filtered associative algebra $\Walg$
equipped with a Hamiltonian $Q$-action.  Namely, consider the universal enveloping
algebra $\U=U(\g)$ with its standard PBW filtration $\U=\bigcup_{i\geqslant 0}\U_{\leqslant i}$.
It will be convenient for us to double the filtration and set $F_i \U:=
\U_{\leqslant [i/2]}$. Form the Rees algebra $\U_\hbar:=\bigoplus_{i}(F_i\U) \hbar^i$.
The quotient $\U_\hbar/(\hbar)$ coincides with $S(\g)=\C[\g^*]$.
Consider the completion $\U_\hbar^{\wedge_\chi}$ in the topology induced by
the preimage of the maximal ideal of $\chi$. The space $V:=T_\chi \Orb$ is symplectic.
So we can form
the  homogenized Weyl algebra $\Weyl_\hbar$ of $V$, i.e., $\Weyl_\hbar$
is  the Rees algebra of the usual
Weyl algebra $\Weyl(V)$. We consider the  completion $\Weyl_\hbar^{\wedge_0}$ in the topology
induced by the maximal ideal of $0\in V$. Both $\U^{\wedge_\chi}_\hbar$ and
$\Weyl_\hbar^{\wedge_0}$ come equipped with  actions of $Q\times \C^\times$.
The action of $Q$ on $\U_\hbar^{\wedge_\chi},\Weyl^{\wedge_0}_\hbar$ is
induced from the natural actions of $Q$ on $\g$ and $V$,
respectively. The group $\C^\times$ acts on $\g^*$ via $t.\alpha:=t^{-2}\gamma(t)\alpha$,
where $\gamma:\C^\times\rightarrow G$ is the one-parameter subgroup associated  to
the element $h$ as in Section \ref{SS_LS_induction}.
The group $\C^\times$ naturally acts on $V$. Finally, we set $t.\hbar:=t\hbar$,
this defines $\C^\times$-actions on $\U_\hbar^{\wedge_\chi},\Weyl_\hbar^{\wedge_0}$
by topological algebra automorphisms that  commute with the $Q$-actions.

It was checked in \cite[Section 3.3]{Wquant}, see also \cite[Section 2.2]{W_prim}
that there is a $Q\times\C^\times$-equivariant
$\C[\hbar]$-linear embedding $\Weyl_\hbar^{\wedge_0}\hookrightarrow
\U_\hbar^{\wedge_\chi}$ such that we have the decomposition
\begin{equation}\label{eq:decomp}
\U_\hbar^{\wedge_\chi}\cong \Weyl_\hbar^{\wedge_0}\widehat{\otimes}_{\C[[\hbar]]}\Walg_\hbar',
\end{equation}
where we write $\Walg_\hbar'$ for the centralizer of $\Weyl_\hbar^{\wedge_0}$
in $\U_\hbar^{\wedge_\chi}$. The algebra  $\Walg'_{\hbar}$ comes with an action of $Q\times \C^\times$.
Let us write $\Walg_\hbar$ for the $\C^\times$-finite part of $\Walg_\hbar'$,
then $\Walg'_\hbar$ is naturally identified with the completion $\Walg_\hbar^{\wedge_\chi}$.
Set $\Walg:=\Walg_\hbar/(\hbar-1)$. This is a filtered algebra with a Hamiltonian
$Q$-action that does not depend on the choice of the embedding $\Weyl_\hbar^{\wedge_0}
\hookrightarrow \U_\hbar^{\wedge_\chi}$ up to an isomorphism preserving the filtration
and the action. See \cite[Section 2.1]{W_prim}. The associated graded algebra $\gr\Walg$
coincides with $\C[S]$, where $S$ is the Slodowy slice.


\subsection{Restriction functor for HC bimodules}\label{SS_HC_restr}
By a $G$-equivariant Harish-Chandra $\U$-bimodule  (or $(\U,G)$-module) we mean a
finitely generated $\U$-bimodule $\B$ such that the adjoint $\g$-action is locally
finite and integrates to an action of $G$.  We can also introduce the
notion of  a $Q$-equivariant HC $\Walg$-bimodule, see \cite[Section 2.5]{HC}.
We write $\HC^G(\U),\HC^{Q}(\Walg)$ for the categories of equivariant HC bimodules.

In \cite[Sections 3.3,3.4]{HC}, we have constructed an exact functor $\bullet_{\dagger}:\HC^G(\U)
\rightarrow \HC^{Q}(\Walg)$. Let us recall the construction of the functor.
Pick a $G$-equivariant HC bimodule $\B$ and equip it with a good filtration
compatible with the filtration $\F_i\U$. So the Rees $\C[\hbar]$-module $\B_\hbar:=
R_\hbar(\B)$ is a $G$-equivariant $\U_\hbar$-bimodule.  Consider the completion
$\B_\hbar^{\wedge_\chi}$ in the $\chi$-adic topology. This is a $Q\times \C^\times$-equivariant
$\U_\hbar^{\wedge_\chi}$-bimodule (the action of $Q$ is Hamiltonian,
while the action of $\C^\times$ is not). As was checked in \cite[Proposition 3.3.1]{HC},
$\B_\hbar^{\wedge_\chi}=\Weyl_\hbar^{\wedge_0}\widehat{\otimes}_{\C[[\hbar]]} \underline{\B}_\hbar'$,
where $\underline{\B}'_\hbar$ is the centralizer of $\Weyl_\hbar^{\wedge_0}$.
So $\underline{\B}'_\hbar$ is a $Q\times \C^\times$-equivariant $\Walg_\hbar^{\wedge_\chi}$-bimodule.
One can show that $\underline{\B}'_\hbar$ coincides with the completion of its  $\C^\times$-finite
part $\underline{\B}_\hbar$. We set $\B_{\dagger}:=\underline{\B}_\hbar/(\hbar-1)$.
This is an object in $\HC^Q(\Walg)$ that comes equipped with a good filtration.
This filtration depends on the choice of the filtration on $\B$, while $\B_{\dagger}$
itself does not.

Let us list properties of the functor $\bullet_{\dagger}$ established in \cite[Sections 3.3,3.4]{HC}.
See \cite[Theorem 1.3.3, Lemma 3.3.2, Proposition 3.4.1]{HC}.

\begin{Lem}\label{Lem:dag_prop}
The following is true:
\begin{enumerate}
\item $\U_\dagger=\Walg$.
\item $\bullet_{\dagger}$ is an exact functor.
\item $\bullet_{\dagger}$ is a monoidal functor.
\item $\gr\B_{\dagger}$ (with respect to the filtration above)
coincides with the pull-back of $\gr\B$ to $S$.
\item In particular, $\bullet_{\dagger}$ maps the category $\HC_{\overline{\Orb}}^G(\U)$
of all HC bimodules supported on $\overline{\Orb}$ to the category $\HC^Q_{fin}(\Walg)$
of all finite dimensional $Q$-equivariant $\Walg$-bimodules. Further, $\bullet_{\dagger}$
annihilates $\HC^G_{\partial\Orb}(\U)$.
\end{enumerate}
\end{Lem}

Recall that by a Dixmier algebra one means an associative algebra $\A$ together with a rational Hamiltonian action of $G$ such that the corresponding quantum comoment map $\U\rightarrow \A$
makes $\A$ a HC bimodule. Since $\bullet_\dagger$ is a monoidal functor it sends a Dixmier algebra
to an algebra $\A_\dagger$ that comes with a homomorphism $\Walg\rightarrow \A_\dagger$. Moreover,  a
good algebra filtration on $\A$ gives rise to an algebra filtration on $\A_\dagger$.
An isomorphism in (4) of the lemma is that of algebras.

The functor $\bullet_{\dagger}: \HC^G_{\overline{\Orb}}(\U)\rightarrow
\HC^Q_{fin}(\Walg)$ has the right adjoint $\bullet^{\dagger}:
\HC^Q_{fin}(\Walg)\rightarrow \HC_{\overline{\Orb}}^G(\U)$, \cite[Proposition 3.4.1,(4)]{HC}. We will need
the construction of the functor $\bullet^{\dagger}$ from \cite[Sections 3.3,3.4]{HC} below so let us recall it.

Pick $\underline{\B}\in \HC^Q_{fin}(\Walg)$ and equip it with a
$Q$-stable $\Walg$-bimodule filtration (for example, we can just take the trivial one).
Then form the Rees bimodule $\underline{\B}_\hbar$
and the $Q$-equivariant $\U_\hbar^{\wedge_\chi}$-bimodule
$\B'_\hbar=\Weyl_\hbar^{\wedge_0}\widehat{\otimes}_{\C[[\hbar]]}
\underline{\B}^{\wedge_\chi}_\hbar$. We can equip
$\B'_\hbar$ with a  $\g$-module structure via
$\xi.b:=\frac{1}{\hbar^2}[\xi,b]$. Let $(\B'_\hbar)_{\g-fin}$
denote the $\g$-finite part of $\B'_\hbar$. This module is
$\C^\times$-stable and we can twist a $\C^\times$-action (see \cite[Section 3.3]{HC}
for details) to get one commuting with $\g$. Let $(\B'_\hbar)_{fin}$
denote the $\C^\times$-finite part in $(\B'_\hbar)_{\g-fin}$,
this is a graded $\U_\hbar$-bimodule. It follows
from \cite[Lemma 3.3.3]{HC} that this bimodule is finitely generated.
Set $\underline{\B}^{\diamondsuit}=(\B'_\hbar)_{fin}/(\hbar-1)$.
This is a HC $\U$-bimodule supported on $\overline{\Orb}$ that comes
with a natural filtration. Note that $\underline{\B}^{\diamondsuit}$
has no sub-bimodules supported on $\partial\Orb$ by the construction.

Also note that $Q/Q^\circ$ naturally acts on $\underline{\B}^{\diamondsuit}$
by filtered $\U$-bimodule automorphisms. We set $\underline{\B}^{\dagger}:=(\underline{\B}^{\diamondsuit})^{Q/Q^\circ}$,
this gives a right adjoint functor of interest. We note that both the kernel and the cokernel
of the adjunction homomorphism $\B\rightarrow (\B_{\dagger})^{\dagger}$ are supported
on $\partial\Orb$, \cite[Proposition 3.4.1,(5)]{HC}.

Now assume $\mathsf{A}$ is a finite dimensional
algebra with a homomorphism $\Walg\rightarrow \mathsf{A}$
and a compatible action of $Q$ such that the composition $\mathfrak{q}\rightarrow \Walg\rightarrow
\mathsf{A}$ is a quantum comoment map for the $Q$-action.
The construction recalled above shows that both $\mathsf{A}^{\diamondsuit}$
and $\mathsf{A}^\dagger$ have Dixmier algebra structures. Moreover, for a $Q$-stable filtration on
$\mathsf{A}$  such that the homomorphism $\Walg\rightarrow \mathsf{A}$ is that of filtered algebras, we get natural filtrations on $\mathsf{A}^{\diamondsuit},\mathsf{A}^\dagger$.
These are good algebra filtrations.

Now let us investigate what happens with $\B_{\dagger}$ and $(\B_{\dagger})^{\dagger}$,
where $\B$ is a quantization of $\C[G/H]$, where $G/H$ is an equivariant cover of $\Orb$. Recall that
$G_\chi/H$ is the fiber of $G/H\rightarrow \Orb$ over $\chi$, it is finite.

\begin{Lem}\label{Lem:restr_quant}
The following is true:
\begin{enumerate}
\item As a filtered algebra $\B_{\dagger}$  is $Q$-equivariantly isomorphic to
the algebra of functions $\C[G_\chi/H]$
(with the trivial filtration).
\item We have $\gr (\B_{\dagger})^{\dagger}=\C[G/H]$ and the natural homomorphism
$\B\rightarrow (\B_{\dagger})^{\dagger}$ is an isomorphism of filtered algebras.
\end{enumerate}
\end{Lem}
\begin{proof}
Let us prove (1). That $\B_{\dagger}$ is an algebra with a natural filtration follows directly from the discussion following
Lemma \ref{Lem:dag_prop}.
The isomorphism $\gr\B_{\dagger}\cong \C[G_\chi/H]$ follows from (4) of  Lemma
\ref{Lem:dag_prop}. This is a graded algebra isomorphism, where $\C[G_\chi/H]$
is in degree $0$. This shows $\B_\dagger\cong \C[G_\chi/H]$.

Let us now prove (2).
Let $\I$ be an $H$-stable codimension $1$ ideal of $\Walg$. We can form the algebra $\mathsf{A}=
\bigoplus_{h\in G_\chi/H}\Walg/h.\I$. For example, $\B_{\dagger}$ is of this form by (1).
Consider the algebra $\mathsf{A}^{\dagger}$. The algebra $\gr\mathsf{A}^{\dagger}$ is
the quotient of the $\g$-finite part of $\Weyl_\hbar^{\wedge_0}\widehat{\otimes}_{\C[[\hbar]]}
\mathsf{A}^{\wedge_\chi}_\hbar$ by the ideal generated by $\hbar$. This embeds into
the $\g$-finite part of $(\Weyl_\hbar^{\wedge_0}/(\hbar))\otimes \gr\mathsf{A}$.
By \cite[Proposition 3.2.2,3.2.3]{HC}, the latter coincides with $\C[G/H]$.

Now consider the natural homomorphism $\B\rightarrow (\B_{\dagger})^{\dagger}$.
This is a homomorphism of filtered algebras. The kernel is supported on $\partial\Orb$ by
Lemma \ref{Lem:dag_prop}.
Note however that $\B$ has no zero
divisors and so is prime. It follows from \cite[Corollar 3.6]{BoKr} that
the kernel is zero. So $\B\hookrightarrow (\B_{\dagger})^{\dagger}$, an inclusion of
filtered algebras. Together
with the inclusion $\gr (\B_{\dagger})^{\dagger}\hookrightarrow \C[G/H]=
\gr \B$, this implies that $\B\hookrightarrow (\B_{\dagger})^{\dagger}$
is an isomorphism of filtered algebras.
\end{proof}

\subsection{Orbit method}\label{SS_Orb_method}
Let $\Orb_1$ be an adjoint orbit corresponding to a birationally minimal induction datum $(\lf,\Orb',\xi)$.
By Theorem \ref{Thm:bir_sheets}, $(\lf,\Orb',\xi)$ is recovered uniquely from
$\Orb_1$ up to $G$-conjugacy. Set $X:=\Spec(\C[G\times^P (X'\times \mathfrak{p}^\perp)])$,
where, as usual, $X':=\operatorname{Spec}(\C[\Orb'])$.
 Let $\A_\xi$
denote the quantization of $\C[X]$ with quantization parameter $\xi$. It is uniquely determined
by $\Orb_1$  as a filtered algebra with a $G$-action. Indeed, $(\lf,\Orb',\xi)$ is determined uniquely
up to $G$-conjugacy from $\Orb_1$, see (1) of Theorem \ref{Thm:bir_sheets}.
Then we can use Corollary \ref{Lem:deform_coinc} to see that $\C[X_\xi]$ is uniquely 
determined up to a $G$-equivariant isomorphism of filtered Poisson algebras. Remark \ref{Rem:equiv_action} then shows that $\A_\xi$ is uniquely determined up to a
$G$-equivariant isomorphism of filtered algebras.

The algebra $\A_\xi$   comes with
a quantum comoment map $\g\rightarrow \A_\xi$. The classical comoment map
$\g\rightarrow \C[X]$ lifts to a quantum comoment map $\g\rightarrow \A_\xi$, the argument is
similar to Step 1 of the proof of Proposition \ref{Prop:birrig_Q_termin}.

We will write $\A(\Orb_1)$
for $\A_\xi$, $\J(\Orb_1)$ for the kernel of $\U\rightarrow \A(\Orb_1)$. We note that
$\J(\Orb_1)$ is primitive. Indeed, the intersection of $\J(\Orb_1)$ with the center
of $\U$ is a maximal ideal in the center because $(\U/\J(\Orb_1))^G\hookrightarrow \A_\xi^G$
and $\dim \A_\xi^G=\dim \C[X]^G=1$. Further, the algebra $\A_\xi$ has no zero divisors
because its associated graded is $\C[X]$, which has that property. It follows
that the ideal $\J(\Orb_1)$ is completely prime, hence prime. It is known, see, for example, \cite[Section 7.3]{Jantzen}, that a prime ideal in $\U$ whose intersection with the center is a maximal ideal
is primitive.

The following is our version of the Orbit method.

\begin{Thm}\label{Thm:orb_method}
The following is true.
\begin{enumerate}
\item If there is a $G$-equivariant algebra isomorphism $\A(\Orb_1)\xrightarrow{\sim} \A(\Orb_2)$,
then $\Orb_1=\Orb_2$.
\item Moreover, assume that $\g$ is classical. If $\J(\Orb_1)=\J(\Orb_2)$, then
$\Orb_1=\Orb_2$.
\end{enumerate}
\end{Thm}

We note that (1) is a weaker version of \cite[Conjecture 3.9]{Vogan_notes} -- we only deal with coadjoint
orbits and not with their covers\footnote{In the forthcoming paper \cite{LMM} we prove
the full conjecture using the same method}.

\begin{proof}
%
Let $(\lf_1,\Orb'_1,\xi_1),(\lf_2,\Orb'_2,\xi_2)$ be the birationally minimal induction data giving
$\Orb_1,\Orb_2$. Let $\chi_1,\chi_2$ be points in the coadjoint orbits  induced from $(\lf_1,\Orb_1,0),
(\lf_2,\Orb_2,0)$ and let $H_i\subset G_{\chi_i}$ be the finite index subgroups
produced from $(\lf_i,\Orb_i,0), i=1,2$.

The proofs of (1),(2) are in several steps.

{\it Step 1}. If $\J(\Orb_1)=\J(\Orb_2)$, then the nilpotent orbits
induced from $(\lf_i,\Orb'_i)$ are the same. Indeed, the closure of the orbit induced
from $(\lf_i,\Orb'_i)$ is the associated variety of $\J(\Orb_i)$. Now note that $\J(\Orb_i)$ coincides with the left annihilator of
the HC bimodule $\mathcal{A}(\Orb_i)$. So if $\A(\Orb_1)\xrightarrow{\sim} \A(\Orb_2)$,
then $\J(\Orb_1)=\J(\Orb_2)$. Therefore we can assume that the nilpotent orbits
induced from $(\lf_i,\Orb_i)$ are the same. Let us write $\Orb$ for this common orbit
and $\chi$ for $\chi_i$.

{\it Step 2}. Assuming $\mathcal{A}(\Orb_1)\cong\mathcal{A}(\Orb_2)$, a $G$-equivariant algebra 
isomorphism, let us show that $H_1,H_2$ are  conjugate in $G_\chi$.  Consider the $W$-algebra $\Walg$ corresponding to
$\Orb$ and the corresponding restriction functor $\bullet_{\dagger}$. 
By (1) of Lemma \ref{Lem:restr_quant}, $\A(\Orb_i)_{\dagger}=\C [G_\chi/H_i]$.
This implies $H_1$ and $H_2$ are conjugate. We can conjugate them and assume they coincide.

{\it Step 3}.  So if we have a $G$-equivariant algebra isomorphism
$\A(\Orb_1)\xrightarrow{\sim} \A(\Orb_2)$, then $H_1=H_2,\J(\Orb_1)=\J(\Orb_2)$.
Set $X:=\operatorname{Spec}(\C[G/H_i])$.
Let us prove (1).
The algebra $\A(\Orb_i)_{\dagger}$ together with a homomorphism $\Walg\rightarrow \A(\Orb_i)_{\dagger}$ can be described as follows.
Pick an $H_i$-stable codimension $1$ ideal $\I\subset \Walg$ containing
$\J(\Orb_i)_{\dagger}$. Form the $G_\chi/G_\chi^\circ$-homogeneous bundle
of algebras over $G_\chi/H_i$ with fiber $\Walg/\I$, denote it  by $\mathsf{A}_i$. By the construction,
$\A(\Orb_i)_{\dagger}=\mathsf{A}_i$. Since $H_1=H_2,\J(\Orb_1)=\J(\Orb_2)$, we have a $Q$-equivariant isomorphism of algebras $\mathsf{A}_1\cong \mathsf{A}_2$.
From (2) of Lemma \ref{Lem:restr_quant} we deduce that $\A(\Orb_1)\cong \A(\Orb_2)$,
a $G$-equivariant isomorphism of filtered algebras. Remark \ref{Rem:equiv_action}
implies that we have a $G$-equivariant filtered Poisson algebra isomorphism
$\C[X_{\xi_1}]\cong \C[X_{\xi_2}]$. We use Lemma \ref{Lem:deform_coinc} to conclude
that $\Orb_1=\Orb_2$.

{\it Step 4}. Let us proceed to proving part (2). We only need to check that $H_1=H_2$.
We will prove a more general claim: we have $H_1=H_2$ provided $A(\Orb)$ is abelian, which is always the case for
classical types.

So suppose that $H_1\neq H_2$. Let $\I$ be an $H_1H_2$-stable codimension $1$ ideal in $\Walg$
containing $\J(\Orb_i)_{\dagger}$.
Let $\Gamma_i:=H_i/H_1\cap H_2, \Gamma:=(H_1H_2)/(H_1\cap H_2)$.
Consider the $G_\chi/G_\chi^\circ$-equivariant algebra $\mathsf{A}$ that is the homogeneous bundle over
$Z_G(e)/(H_1\cap H_2)$ with fiber $\Walg/\I$ over $1$. The group $\Gamma$
acts on $\mathsf{A}$ by $G_\chi/G_\chi^\circ$-equivariant filtered algebra automorphisms
fixing the image of $\Walg$. Moreover, $\A(\Orb_i)_{\dagger}=\mathsf{A}^{\Gamma_i}$.
So the group $\Gamma$ also acts on $\mathsf{A}^{\dagger}$ by filtered algebra automorphisms fixing
the image of $U(\g)$. Similarly to (2) of Lemma \ref{Lem:restr_quant}, we have
$G$-equivariant filtered algebra isomorphisms $\A(\Orb_i)\cong (\mathsf{A}^{\dagger})^{\Gamma_i}$. In particular,
we see that the group $\Gamma/\Gamma_i$ acts on $\A(\Orb_i)$ by filtered algebra
automorphisms lifting the action of this group on $\C[G/H_i]$.
This means that the parameter $W_i \xi_i$ of the quantization
$\A(\Orb_i)$ is stable under the action of $\Gamma/\Gamma_i$ (where we write $W_i$
for the Namikawa-Weyl group of $\C[G/H_i]$). It follows from Remark \ref{Rem:equiv_action}
that the filtered deformation $\C[X_{\xi_i}]$ of $\C[G/H_i]$
also carries an action of $\Gamma/\Gamma_i$ by $G$-equivariant filtered Poisson algebra
automorphisms. Since $G$ is semisimple, the moment map $X_{\xi_i}\rightarrow \g^*$
is unique hence is preserved by $\Gamma/\Gamma_i$.
This contradicts the assumption that the moment map $X_{\xi_i}\rightarrow \overline{\Orb}_{i}$
is birational and completes the proof.
\end{proof}

\begin{Rem}
There are 12 orbits in exceptional Lie algebras with noncommutative $A(\Orb)$:
it can be equal to $S_3$ (10 orbits, the easiest example is the subregular orbit for
$G_2$), $S_4$ (for a single orbit in type $F_4$) or $S_5$
(for a single orbit in type $E_8$). We have checked that the conclusion of (2)
is still true for most of these orbits and we do not know what happens for the rest.
\end{Rem}

%

\subsection{Toward description of the image}\label{SS_image}
An interesting question is to describe the image of the map $\g^*/G\hookrightarrow \operatorname{Prim}(\g)$.
We will sketch the conjectural results.

Note that the orbits lying in the union $Z(\Orb)$ of all sheets containing $\Orb$ get
mapped to $\operatorname{Prim}_{\Orb}(\g)$, the set of primitive ideals whose
associated variety is $\overline{\Orb}$. Let $\Walg$ denote the  W-algebra
corresponding to $\Orb$. Recall, \cite[Section 1.2]{HC}, that $\operatorname{Prim}_{\Orb}(\g)$
is the quotient of $\operatorname{Irr}_{fin}(\Walg)$, the set of isomorphism classes
of finite dimensional irreducible modules over $\Walg$, by the action of $A(\Orb)$.

By Lemma \ref{Lem:restr_quant},   the kernel $U(\g)\rightarrow \A_\xi$, where $\A_\xi$ is a
quantization of $\C[G/H]$ corresponds to a one-dimensional $\Walg$-module.
When $A(\Orb)$ is abelian, the argument of Step 4 of the proof of
Theorem \ref{Thm:orb_method} implies that   any such kernel lies in the image of $Z(\Orb)/G\hookrightarrow
\operatorname{Prim}_{\Orb}(\g)$. When $A(\Orb)$ is not abelian, then this does not
need to be the case, one gets a counter-example for the subregular orbit in $G_2$.

So suppose $A(\Orb)$ is abelian.
Now let us impose another assumption on $\Orb$: we assume that $\Orb$ is not one
of the six bad orbits in \cite[Introduction]{Premet3} and is not induced from such an orbit.
Recall that the six orbits (all in exceptional Lie algebras) are characterized by the
property that $\Orb\subsetneq \operatorname{Spec}(\C[\Orb])^{reg}$, see the tables
in \cite{FJLS}.

Results from \cite{PT} imply that, under our assumption on $\Orb$, every {\it multiplicity free} (in the terminology of
\cite{PT}) primitive ideal in $\operatorname{Prim}_{\Orb}(\g)$  arises as the kernel of $U(\g)\rightarrow \A_\xi$,
where $\A_{\xi}$ is a suitable quantization of $\C[\Orb]$. Conversely, any such kernel
is a multiplicity free primitive ideal.

From now on let us suppose that $\g$ is classical. We conjecture that the image of
the injective map $Z(\Orb)/G\rightarrow \operatorname{Prim}_{\Orb}(\g)$ from
Theorem \ref{Thm:orb_method} coincides with the subset of all primitive ideals
corresponding to the one-dimensional representations of $\Walg$. This conjecture
constitutes a right statement of the Orbit method for classical Lie algebras\footnote{As we have mentioned in Introduction,  this has been established in
\cite{Topley}.}.

\end{document}